\documentclass[11pt]{amsart}

\usepackage[utf8]{inputenc}
\usepackage[font={small}]{caption}
\usepackage{fullpage, mathrsfs, mathtools, amssymb, tikz,bm, hyperref, amscd, scrextend, tikz-cd, verbatim, enumerate}
\usepackage[inline]{enumitem}
\usepackage{microtype}

\usetikzlibrary{matrix, arrows}

\input xy
\xyoption{all}

\newtheorem{theorem}{Theorem}[section]
\newtheorem{lemma}[theorem]{Lemma}
\newtheorem{prop}[theorem]{Proposition}
\newtheorem{cor}[theorem]{Corollary}

\newtheorem{Step_2}{Step}

\theoremstyle{definition}
\newtheorem{definition}[theorem]{Definition}

\newtheorem{remark}[theorem]{Remark}
\newtheorem{conv}[theorem]{Convention}

\theoremstyle{remark}

\newenvironment{steps}{\setcounter{step}{0}}{}
\newcounter{step}

\newcommand{\proofstep}{\par \noindent \refstepcounter{step}\textbf{Step~\thestep:}\space}

\newcommand{\C}{{\mathbb C}}

\newcommand{\Q}{{\mathbb Q}}

\newcommand{\LL}{{\mathscr L}}

\newcommand{\PP}{{\mathbb P}}

\newcommand{\calO}{\mathcal O}
\newcommand{\calM}{\mathcal M}
\newcommand{\calF}{\mathcal F}

\newcommand{\calA}{\mathcal A}
\newcommand{\calB}{\mathcal B}

\newcommand{\calE}{\mathcal E}
\newcommand{\calU}{\mathcal U}
\newcommand{\calX}{\mathcal X}
\newcommand{\calC}{\mathcal C}

\newcommand{\calY}{\mathcal Y}
\newcommand{\calZ}{\mathcal Z}
\newcommand{\calT}{\mathcal T}
\newcommand{\calD}{\mathcal D}
\newcommand{\calW}{\mathcal W}

\newcommand{\frakU}{\mathfrak U}
\newcommand{\frakF}{\mathfrak F}

\newcommand{\cA}{{{\mathcal A}'}}

\newcommand{\Supp}{\mathrm{Supp}}
\newcommand{\Spec}{\mathrm{Spec}}
\newcommand{\Proj}{\mathrm{Proj}}

\newcommand{\wi}{\mathrm{W}_{\mathrm{I}}}
\newcommand{\wii}{\mathrm{W}_{\mathrm{II}}}
\newcommand{\wiii}{\mathrm{W}_{\mathrm{III}}}

\newcommand{\mc}[1]{\mathcal{#1}}
\newcommand{\mb}[1]{\mathbb{#1}}
\newcommand{\note}[1]{\textcolor{red}{#1}}

\subjclass[2010]{14J10, 14J27,14D23, 14E30}

\title{Moduli of Weighted Stable Elliptic Surfaces and Invariance of Log Plurigenera}

\begin{document}

\author[Ascher]{Kenneth Ascher}
\email{kascher@mit.edu}
\author[Bejleri]{Dori Bejleri}
\email{dbejleri@math.brown.edu}

\dedicatory{with an appendix by Giovanni Inchiostro}
\email{giovanni\_inchiostro@math.brown.edu}

\maketitle

\begin{abstract}
This is the third paper in a series of work on weighted stable elliptic surfaces -- elliptic fibrations with section and marked fibers each weighted  between zero and one. Motivated by Hassett's weighted pointed stable curves, we use the log minimal model program to construct compact moduli spaces parameterizing these objects. Moreover, we show that the domain of weights admits a wall and chamber structure, describe the induced wall-crossing morphisms on the moduli spaces as the weight vector varies, and describe the surfaces that appear on the boundary of the moduli space. The main technical result is a proof of invariance of log plurigenera for slc elliptic surface pairs with arbitrary weights.
\end{abstract}

\section{Introduction}
Elliptic fibrations are ubiquitous in mathematics, and the study of their moduli has been approached from many directions; e.g. via Hodge theory \cite{hl} and geometric invariant theory \cite{mir}. At the same time, moduli spaces often have many geometrically meaningful compactifications leading to different birational models. This leads to rich interactions between moduli theory and birational geometry. 

The compact moduli space $\overline{\calM}_g$ of genus $g$ stable curves and its pointed analogue $\overline{\calM}_{g,n}$ is exemplary. Studying the birational geometry of the moduli space of stable curves by varying the moduli problem has been a subject of active research over the past decade known as the \emph{Hassett-Keel program} (see \cite{fedorchuk} for a survey). Our hope is to produce one of the first instances of this line of study for surfaces (see also \cite{ol} which initiates a similar line of study for quartic K3 surfaces). This paper, along with our work in \cite{calculations} and \cite{tsm}, continues a study of the birational geometry of the moduli space of \emph{stable elliptic surfaces} initiated by La Nave \cite{ln}. 

One particular instance of the birational geometry of $\overline{\calM}_{g,n}$ is developed by Hassett in \cite{has}, where various compactifications $\overline{\calM}_{g, \calA}$ of the moduli space of \emph{weighted pointed curves} (see Section \ref{sec:hassettreduction}) are constructed. These compact moduli spaces parameterize degenerations of genus $g$ curves with marked points weighted by a vector $\calA$. A natural question is: what happens to the moduli spaces as one varies the weight vector? Among other things, Hassett proves that there are birational morphisms $\overline{\calM}_{g, \calB} \to \overline{\calM}_{g, \calA}$ when $\calA \leq \calB$ (Theorem \ref{hastheorem}). Furthermore, there is a wall-and-chamber decomposition of the space of weight vectors $\calA$ -- inside a chamber the moduli spaces are isomorphic and there are explicit birational morphisms when crossing a wall.

Hassett's space is the one dimensional example of moduli spaces of \emph{stable pairs}: pairs $(X,D)$ of a variety along with a divisor having mild singularities and satisfying a positivity condition (see Definition \ref{def:pair}). In this case, the variety is a curve $C$ with at worst nodal singularities, the divisor is a weighted sum $D = \sum a_i p_i$ of smooth points on the curve, and one requires $\omega_C(D)$ to be an ample line bundle. 

In this paper, we use stable pairs in higher dimensions to construct analogous compactifications of the moduli space of elliptic surfaces where the pair is given by an elliptic surface with section as well as $\calA$-weighted marked fibers. The outcome is a picture for surface pairs which is more intricate, but analogous to that of $\overline{\calM}_{g,\calA}$. 

In general, the story of compactifications of moduli spaces in higher dimensions is quite subtle and relies on the full power of the minimal model program (mmp). Many fundamental constructions have been carried out over the past few decades (e.g. \cite{ksb}, \cite{boundedness}, \& \cite{kp}). Although stable pairs have been identified as the right analogue of stable curves in higher dimensions, it has proven difficult to find explicit examples of compactifications of moduli spaces in higher dimensions (see \cite{aleicm} for some examples), and thus we take as one goal of this paper to establish a wealth of examples of compact moduli spaces of surfaces that illustrate both the difficulties, as well as methods necessary to overcome them.  

More specifically, for admissible weights $\calA$ (see Section \ref{sec:walls}), we construct and study $\calE_{v,\calA}$ (Definition \ref{def:eva}): the compactification by stable pairs of the moduli space of $(f: X \to C, S + F_{\calA})$, where $f: X \to C$ is a minimal elliptic surface with chosen section $S$, marked fibers weighted by $\calA$, and fixed volume $v$. 

\begin{theorem}[see Theorem \ref{thm:stack}] For admissible weights $\calA$, there exists a moduli pseudofunctor of $\calA$-weighted stable elliptic surfaces (see Definition \ref{def:es}) of volume $v$ so that the main component $\calE_{v,\calA}$ is representable by a finite type separated Deligne-Mumford stack. 
\end{theorem}

To construct $\calE_{v,\calA}$ as an algebraic stack, we use the notion of a family of stable pairs given by Kov\'acs-Patakfalvi in \cite{kp} and the construction of the moduli stack of stable pairs therein. However, representability of our functor does not follow immediately as we include the additional the data of the map $f : X \to C$ (see Section \ref{sec:modulifunctor}). Furthermore, the correct deformation theory for stable pairs has not yet been settled. As we are interested in the global geometry of the moduli space in this paper, we circumvent this issue by working exclusively with the normalization of the moduli stack. By the results of Appendix \ref{appendix}, this amounts to only considering the functor on the subcategory of normal varieties. 

\begin{theorem}[see Theorem \ref{thm:proper} and Theorem \ref{thm:boundary}]\label{intro:proper} The moduli space $\calE_{v,\calA}$ is proper. Its boundary parametrizes $\calA$-broken elliptic surfaces (see Definition \ref{def:broken} and Figure \ref{fig:brokensurface}). \end{theorem} 

As with the previous theorem, it does not follow immediately from known results about stable pairs because of the data of the map $f : X \to C$. Rather, we prove Theorem \ref{intro:proper} by explicitly describing in Section \ref{sec:stablereduction} an algorithm for stable reduction that produces, as a limit, a stable pair as well as a map to a nodal curve. This is a generalization of the work of La Nave in \cite{ln}. The main input is the use of twisted stable maps of Abramovich-Vistoli to produce limits of fibered surface pairs as discussed in \cite{tsm} as well as previous results in \cite{calculations} and \cite{ln}, that describe the steps of the minimal model program on a one parameter family of elliptic surfaces. The final key input is a theorem of Inchiostro in Appendix \ref{sec:appendix} (Theorem \ref{Teo:appendix:giovanni}) which guarantees these are the only steps that occur in the mmp.

Following Hassett, it is natural to ask how the moduli spaces change as we vary $\calA$. The strategy in \cite{has} is to understand how the objects themselves change as one varies $\calA$, and then prove a strong vanishing theorem which guarantees that the formation of the relative log canonical sheaf commutes with base change. This ensures that the process of producing an $\calA$-stable pointed curve from a $\calB$-stable pointed curve with $\calA \le \calB$ is functorial in families and leads to reduction morphisms on moduli spaces and universal families. 

In \cite{calculations}, we carried out a complete classification of the relative log canonical models of elliptic surface fibrations, and we extend this result here (see Section \ref{sec:local} and Theorem \ref{thm:thm1}).

In Section \ref{sec:vanishing}, we prove an analogous base change theorem which implies that the steps of the minimal model program described in Section \ref{sec:local} are functorial in families of elliptic surfaces. The main technical tool is a vanishing theorem (Theorem \ref{thm:vanishing}) which relies on a careful analysis of the geometry of broken slc elliptic surfaces. We do not expect this vanishing theorem to hold in full generality for other classes of slc surfaces. 

\begin{theorem}[Invariance of log plurigenera, Theorems \ref{thm:vanishing} and \ref{thm:basechange}]
Let $\pi: (X \to C, S + F_\calB) \to B$ be a family of $\calB$-broken stable elliptic surfaces  over a reduced base $B$. Let $0 \le \calA \le \calB$ be such that the divisor $K_{X/B} + S + F_\calA$ is $\pi$-nef and $\Q$-Cartier. Then $\pi_*\calO_X\big(m(K_{X/B} + S + F_\calA)\big)$ is a vector bundle on $B$ whose formation is compatible with base change $B' \to B$ for $m \geq2$ divisible enough.  
\end{theorem}

The main difficulty in the above theorem, and in the study of stable pairs in general, is that smooth varieties will degenerate into non-normal varieties with several irreducible components. In dimension greater than one these \emph{slc} varieties can become quite complicated: see Figure \ref{fig:brokensurface} for a $\calB$-broken elliptic surface that appears in the limit of such a degeneration. Note in particular the map $f : X \to C$ is not equidimensional; there are irreducible components of $X$ contracted to a point by $f$. 

These components were first observed in the work of La Nave \cite{ln} and were coined \emph{pseudoelliptic} surfaces. They are the result of contracting the section of an elliptic surface. In fact La Nave noticed in the study of stable reduction for elliptic surfaces with no marked fibers ($\calA = 0$), that a component of the section of $f$ is contracted by the minimal model program if and only if the corresponding component of the base nodal curve $C$ needs to be contracted to obtain a stable curve. We generalize this (Proposition \ref{prop:adjunction}) to the case of marked fibers and as a result obtain a morphism to the corresponding Hassett space by forgetting the elliptic surface and remembering only the base weighted curve:

\begin{theorem}[See Corollary \ref{cor:forget}] There are forgetful morphisms $\calE_{\calA} \to \overline{\calM}_{g, \calA}$. \end{theorem}

Next we identify a wall and chamber decomposition of the space of admissible weights $\calA$. In particular, we describe at which $\calA$ a one parameter family of broken elliptic surfaces undergo birational transformations leading to different objects on the boundary of the moduli stack. In Section \ref{sec:stablereduction} we classify three types of birational transformations leading to three types of walls:
\begin{itemize}
\item there are $\wi$ walls coming from the relative log minimal model program for the map $f : X \to C$ at which singular fibers change;

\item there are $\wii$ walls where a component of the section contracts to form a pseudoelliptic surface;

\item there are $\wiii$ walls where an entire component of a broken elliptic surface may contract onto a curve or point. 
\end{itemize}

Type $\wi$ and $\wiii$ transformations result in divisorial contractions of the total space of a family of elliptic surfaces while type $\wii$ result in small contractions which must then be resolved by a log flip. La Nave constructed this log flip explicitly and we show that this construction leads to a log flip of the universal family (see Section \ref{sec:flip} and Figure \ref{figure:3}). Putting this all together, our main theorem may be summarized as follows: 

\begin{theorem} \label{thm:main1}
Let $\calA, \calB \in \Q^r$ be weight vectors such that $0 \le \calA \le \calB \le 1$. We have the following:

\begin{enumerate}
\item If $\calA$ and $\calB$ are in the same chamber, then the moduli spaces and universal families are isomorphic.
\item If $\calA \le \calB$ then there are reduction morphisms $\calE_{v,\calB} \to \calE_{v,\calA}$ on moduli spaces which are compatible with the reduction morphisms on the Hassett spaces:
$$
\xymatrix{\calE_{v,\calB} \ar[r] \ar[d] & \calE_{v,\calA} \ar[d] \\ \overline{\calM}_{g,\calB} \ar[r] & \overline{\calM}_{g,\calA}}
$$

\item The universal families are related by a sequence of explicit divisorial contractions and flips $\calU_{v,\calB} \dashrightarrow \calU_{v,\calA}$ such that the following diagram commutes:
$$
\xymatrix{\calU_{v,\calB} \ar@{-->}[r] \ar[d] & \calU_{v,\calA} \ar[d] \\ \calE_{v,\calB} \ar[r] & \calE_{v,\calA}}
$$
More precisely, across $\wi$ and $\wiii$ walls there is a divisorial contraction of the universal family and across a $\wii$ wall the universal family undergoes a log flip. 
\end{enumerate}

\end{theorem}

The precise descriptions of the various wall crossing morphisms described above are given in Theorem \ref{thm:reduction}, Corollary \ref{cor:chambers}, Proposition \ref{prop:1-e}, Theorem \ref{thm:increase} and Proposition \ref{prop:flip}. \\

Now we will describe the objects that appear the boundary of $\calE_{v,\calA}$. While the minimal model program lends itself to an algorithmic approach towards finding minimal birational representatives of an equivalence class, it generally does \emph{not} lead to an explicit stable reduction process as prevalent in $\overline{\calM}_{g,n}$. However, using the minimal model program in addition to the theory of twisted stable maps developed by Abramovich-Vistoli \cite{av}, we are able to run an explicit stable reduction process, and classify precisely what objects live on the boundary of our moduli spaces. This is inspired by the work of \cite{ln}. 

The idea is that an elliptic fibration $f : X \to C$ with section $S$ can be viewed as a rational map from the base curve to $\overline{\calM}_{1,1}$, the stack of stable pointed genus one curves. One can use this to produce a birational model of $f$ which can then be studied using twisted stable maps. The outcome is a compact moduli space of \emph{twisted} fibered surface pairs studied in \cite{tsm} which forms the starting point of our analysis of one parameter degenerations in $\calE_{v,\calA}$.  

Combining these degenerations produced by twisted stable maps with the wall crossing transformations discussed above and Theorem \ref{Teo:appendix:giovanni}, in Section \ref{sec:stablereduction} we identify the boundary objects parametrized by $\calE_{v,\calA}$:

\begin{theorem}[see Theorem \ref{thm:boundary}]\label{thm:introboundary}
	The boundary of the proper moduli space $\calE_{v, \calA}$ parametrizes $\calA$-broken stable elliptic surfaces (see Definition \ref{def:broken}) which are pairs $(f : X \to C, S + F_{\calA})$ coming from a stable pair $(X, S + F_\calA)$ with a map to a nodal curve $C$ such that:
	\begin{itemize}
	\item $X$ is an slc union of elliptic surfaces with section $S$ and marked fibers, as well as
	\item chains of pseudoelliptic surfaces of Type $\mathrm{I}$ and $\mathrm{II}$ (Definitions \ref{def:pseudoI} and \ref{def:pseudoII}) contracted by $f$ with marked pseudofibers (Definition \ref{def:pseudo}). \end{itemize} 
\end{theorem} 

	\begin{figure}[h!]
			\includegraphics[scale=1]{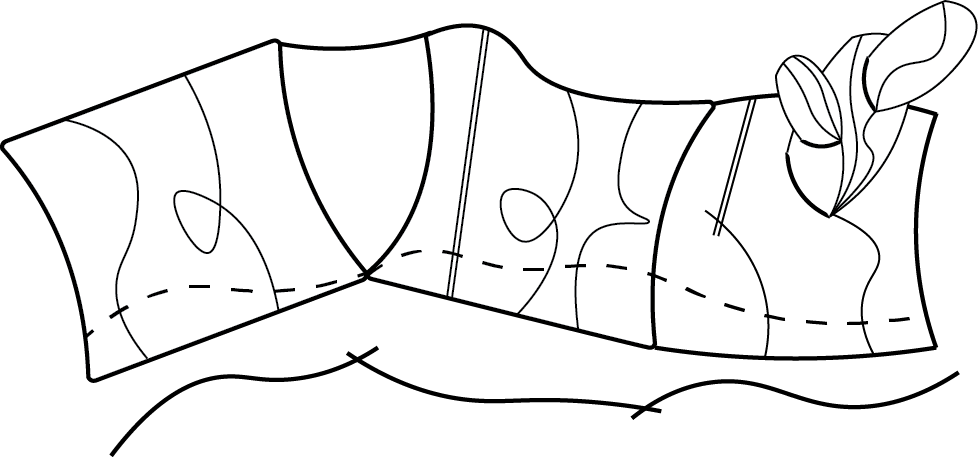}
		\caption{An $\calA$-weighted broken elliptic surface.}\label{fig:brokensurface}
		\end{figure}
		
Finally, in Appendix \ref{appendix}, we prove that in certain situations the normalization of an algebraic stack is uniquely determined by its values on normal base schemes (Proposition \ref{uniquenorm}) and that a morphism between normalizations of algebraic stacks can be constructed by specifying it on the category of normal schemes (Proposition \ref{normfunctor}). This material is probably well known but we include it here for lack of a suitable reference. 

\subsection{An example} We illustrate the main results in a specific example. Figure \ref{fig:intro} depicts the central fiber of a particular one parameter stable degeneration of a rational elliptic surface with twelve marked nodal fibers, ten of which are marked with coefficient one, and the other two with coefficient $\alpha$, as the coefficient $\alpha$ varies. The arrows depict the directions of the morphisms between the various models of the total space of the degeneration. 

\begin{figure}[h!]\includegraphics[scale=.30]{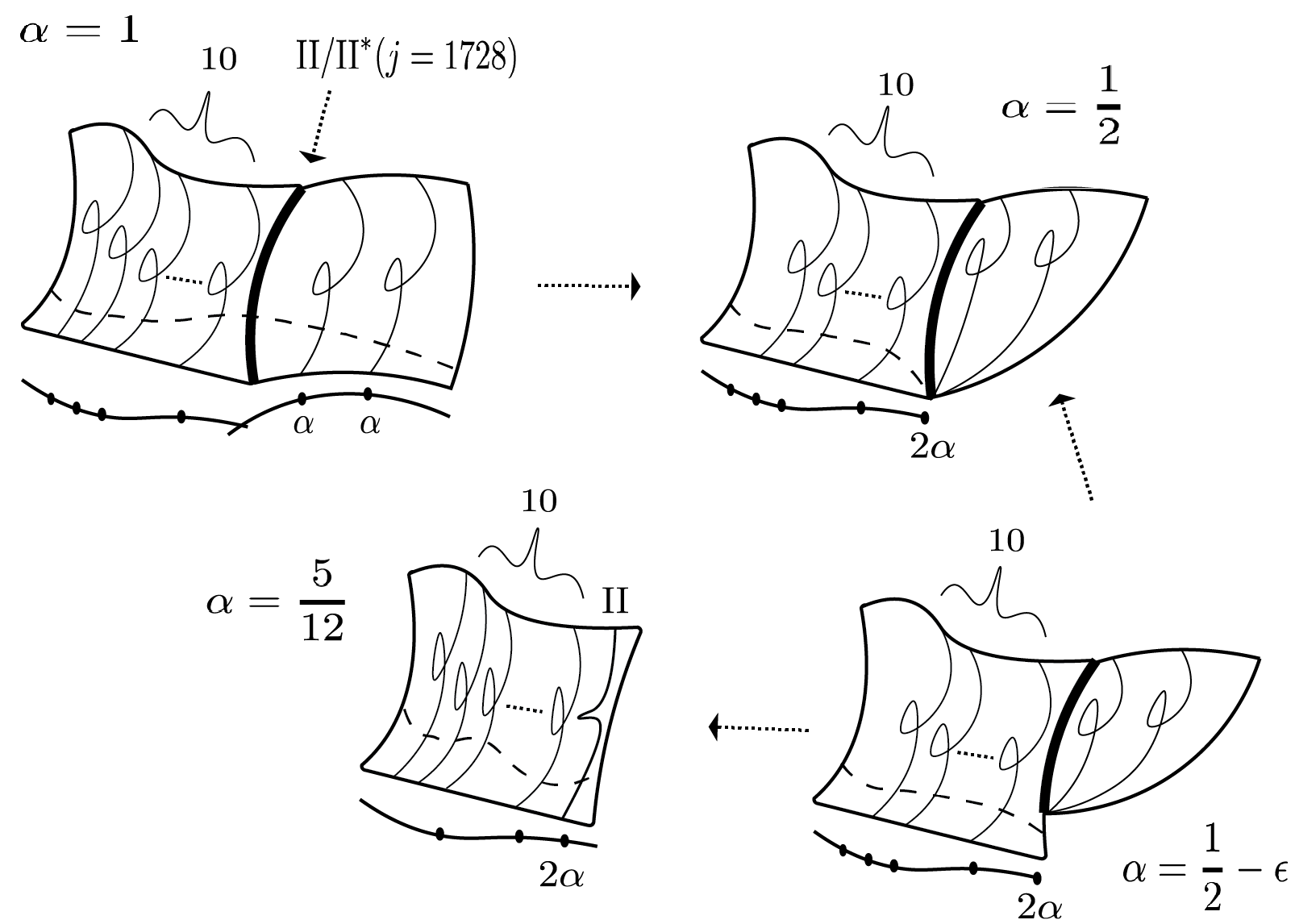}
\caption{The wall crossing transformations on the central fiber of a stable degeneration of a rational elliptic surface.}
\label{fig:intro}
\end{figure} 

The central fiber breaks up into a union of two components glued along twisted fibers of type $\mathrm{II}$ and $\mathrm{II}^*$, one containing $10$ marked nodal fibers with coefficient one and a type $\mathrm{II}$ twisted fiber, and the other containing two nodal fibers marked with coefficient $\alpha$ and a type $\mathrm{II}^*$ twisted fiber. At $\alpha = 1/2$ the section of the second component contracts to form a pseudoelliptic surface. At $\alpha = 1/2 - \epsilon$ for any small enough $\epsilon > 0$, this contraction of the section is a log flipping contraction of the total space of the degeneration and a flip results in a different stable model. Finally at $\alpha = 5/12$ the whole pseudoelliptic component contracts to a point yielding an elliptic surface with $10$ nodal fibers marked with coefficient one and a type $\mathrm{II}$ Weierstrass fiber with coefficient $2\alpha$. Each surface maps to the corresponding Hassett stable base curve as depicted.

\subsection{Applications and further work} A simple corollary of the results in this paper is a classification of the singularities of stable degenerations of smooth elliptic surfaces. Combining Theorem \ref{thm:introboundary} with the results of \cite{calculations} on singularities of log canonical models of elliptic surfaces (see also Section \ref{sec:local}) as well as Proposition \ref{stableattaching} we obtain the following:

\begin{cor} Let $\mathscr{X}^0 \to \mathscr{C}^0 \to \Delta^0$ be a family of smooth relatively minimal elliptic surfaces over the punctured disc $\Delta^0 = \Delta \setminus \{0\}$ and with a fixed section and all singular fibers marked by a nonzero coefficient. Then after a base change of $\Delta^0$, the family can be extended to $\mathscr{X} \to \mathscr{C} \to \Delta$ such that the central fiber $X \to C$ is a broken elliptic surface. Each component of $X$ is an elliptic or pseudoelliptic surface with only quotient singularities and the singularities are all rational double points except along type $\mathrm{II}$, $\mathrm{III}$ and $\mathrm{IV}$ fibers. In particular, the normalization $X^\nu$ has klt singularities. 

\end{cor}

As another application, in \cite{rational} we use the results of this paper to construct a stable pairs compactification of the moduli space of anti-canonically polarized del Pezzo surfaces of degree one. By studying the wall-crossing morphisms, we relate this compactification to the GIT compactification of the moduli space of rational elliptic surfaces of Miranda \cite{mir}. In addition, we calculate all walls in the domain of admissible weights for the case of \emph{rational} elliptic surfaces. Our future work expands upon these ideas, by comparing our compactifications to other compactifications of rational elliptic surfaces in the literature, e.g. the hodge theoretic approaches of Heckman-Looijenga \cite{hl}. As $\calE_{v,\calA}$ is modular, the explicit description of the boundary can be used to describe the boundaries of non-modular compactifications such as GIT models and compactifications of period domains. 

Finally, we remark on our choice of boundary divisor. We fix the coefficient of the section to be one. This is the key reason that the base curve of a stable elliptic surface is a Hassett stable curve (see Proposition \ref{prop:adjunction}). On the other hand, it is the reason for the formation of pseudoelliptic surfaces which leads to interesting yet complicated behavior across type $\mathrm{W_{\mathrm II}}$ and $\mathrm{W_{\mathrm III}}$ walls. 

Our marked fibers consist of log canonical models of marked Weierstrass fibers which are classified in Theorem \ref{thm:thm1} and the preceding discussion. In particular, there are three types of fibers -- Weierstrass fibers, twisted fibers obtained by stable reduction, and intermediate fibers which interpolate between the former two as the coefficient varies from zero to one. Since our fibers come as log canonical models of Weierstrass fibers, they have to be marked with coefficient one on any exceptional divisor of the rational map to the Weierstrass model.

It would be interesting to extend our results to the case where the coefficient of these components and of $S$ varies. When the coefficient of $S$ is very small compared to the other numerical data, one expects to obtain a compactification of the moduli space of Weierstrass fibrations by equidimensional slc elliptic fibrations. This generalization is being carried out by Inchiostro in \cite{giovanni}.

\subsection{Previous results} La Nave \cite{ln} used twisted stable maps of Abramovich-Vistoli to prove properness of the moduli stack parameterizing elliptic surfaces in Weierstrass form via explicit stable reduction.  He computes the stable models of one parameter families of elliptic surfaces in Weierstrass form.  Roughly, given an elliptic surface $f: X \to C$ with section $S$, the Weierstrass form is obtained by contracting all components of the singular fibers of $f: X \to C$ that do not meet the section $S$. We will make repeated use of his work throughout. In our setting, this corresponds to the case of $\calE_\calA$ where $\calA = 0$. 

Brunyate \cite{brunyate}, described stable pair limits of elliptic K3s with marked divisor $D = \delta S + \sum_{i=1}^{24} \epsilon F_i$, where $0<\delta \ll \epsilon \ll 1$, the divisor $S$ is  a section, and the $F_i$ are the 24 singular fibers.

In \cite{barcelona}, Alexeev provided another generalization of Hassett's picture for $\overline{\calM}_{g,\calA}$ to surfaces. He constructed reduction morphisms for the compact moduli spaces of \emph{weighted hyperplane arrangements} -- the moduli space parametrizing the union of hyperplanes in projective space.

Deopurkar in \cite{deopurkar} also suggested an alternate compactification of the moduli space of elliptic surfaces by admissible covers of the stacky curve $\overline{\calM}_{1,1}$. It would be interesting to compare his space to those discussed here and in \cite{tsm}.

We work over $\C$.

\subsection{Outline}
\begin{addmargin}[45pt]{0\linewidth}
\begin{itemize}
\item[Sec. 2 (Pg. \pageref{sec:preliminaries})] We give background on stable pairs, and recall some results about their moduli spaces. We also give preliminaries on elliptic surfaces, the log minimal model program and vanishing theorems.
\item[Sec 3 (Pg. \pageref{sec:ellipticbackground})]  We discuss log canonical models of elliptic surfaces (from \cite{calculations}). We introduce pseudoelliptic surfaces and classify log canonical contractions of elliptic surfaces.
\item[Sec 4 (Pg. \pageref{sec:weightedstable})] We define the objects that appear on the boundary of our moduli spaces motivated by stable reduction, define a moduli functor, and prove that it is algebraic. We further define pseudoelliptic surfaces and prove that the base curve of a stable elliptic surface is a weighted stable curve, 
\item[Sec 5 (Pg. \pageref{sec:vanishing})]  We prove a strong vanishing theorem for slc elliptic surfaces that implies invariance of log plurigenera. 
\item[Sec 6 (Pg. \pageref{sec:stablereduction})] We prove a stable reduction theorem for $\calE_{v,\calA}$ to obtain properness, also enabling us to also give an explicit description of the surface on the boundary of our moduli space. In the process we describe a wall and chamber decomposition of the space of admissible weights.
\item[Sec 7 (Pg. \pageref{sec:reduction})] We construct reduction morphisms on our moduli space and universal family, and show that these morphisms are compatible with Hassett's reduction morphisms.
\item[Sec 8 (Pg. \pageref{sec:flip})] We show that along certain types of walls, the universal family undergoes a log flip. 
\item[App A (Pg. \pageref{appendix})] We show that the normalization of an algebraic stack is uniquely determined by its values on normal bases.
\item[App B (Pg. \pageref{sec:appendix})] G. Inchiostro shows that the only flip that occurs when running stable reduction is the flip of La Nave.
\end{itemize}
\end{addmargin}

\subsection*{Acknowledgements}We thank our adviser Dan Abramovich, Valery Alexeev, Brendan Hassett, J\'anos Koll\'ar, S\'andor Kov\'acs, Radu Laza, Wenfei Liu, Zsolt Patakfalvi, and S\"onke Rollenske for many insightful discussions. Research supported in part by funds from NSF grant DMS-1500525.


\section{The minimal model program, moduli of stable pairs, and elliptic surfaces}\label{sec:preliminaries}

We work with $\Q$-divisors. Whenever we write equality for divisors, e.g. $K_X = \Delta$, unless otherwise noted, we mean $\Q$-linear equivalence. 

\subsection{Semi-log canonical pairs}

	To compactify the moduli space of pairs of log general type, one needs to introduce pairs on the boundary which have \emph{semi-log canonical (slc) singularities}. 
	
	\begin{definition}\label{def:loc} Let $(X, D = \sum d_i D_i)$ be a pair of a normal variety and a $\Q$-divisor such that $K_X + D$ is $\Q$-Cartier. Suppose that there is a log resolution $f: Y \to X$ such that $$K_Y + \sum a_E E = f^*(K_X + D),$$ where the sum goes over all irreducible divisors on $Y$. We say that the pair $(X,D)$ has \textbf{log canonical singularities} (or is lc) if all $a_E \leq 1$. \end{definition}
	
	\begin{definition}\label{def:slc} Let $(X, D)$ be a pair of a reduced variety and a $\Q$-divisor such that $K_X + D$ is $\Q$-Cartier. The pair $(X,D)$ has \textbf{semi-log canonical singularities} (or is an slc pair) if:
		\begin{itemize}
			\item The variety $X$ is S2,
			\item $X$ has only double normal crossings in codimension 1, and 
			\item If $\nu: X^{\nu} \to X$ is the normalization, then the pair $(X^{\nu}, \nu_*^{-1}D + D^{\nu})$ is log canonical, where $D^{\nu}$ denotes the preimage of the double locus on $X^{\nu}$. \end{itemize}
	\end{definition}

	\begin{definition} \label{def:pair} A pair $(X, D)$ of a projective variety and $\Q$-divisor is a \textbf{stable pair} if $(X,D)$ is an slc pair, and $\omega_X(D)$ is ample.
	\end{definition}
	
	\begin{definition} Let $(X,D)$ be an (s)lc pair and let $f : X \to B$ be a projective morphism. The \textbf{(semi-)log canonical model} of $f : (X,D) \to B$, if it exists, is the unique (s)lc pair $(Y,\mu_*D)$ given by
	$$
	Y:=\mathrm{Proj}_B\left(\bigoplus_m f_*\calO_X(m(K_X + D))\right) \to B
	$$
	and $\mu : X \dashrightarrow Y$. When $B$ is a point, $(Y, \mu_*D)$ is a stable pair. \end{definition}

We will make repeated use of abundance for slc surface pairs in computing log canonical models of slc surface pairs.

\begin{prop}[Abundance for slc surfaces, see \cite{afkm} and \cite{kawamata}] \label{prop:abundance} Let $(X,D)$ be an slc surface pair and $f : X \to B$ a projective morphism. If $K_X + D$ is $f$-nef, then it is $f$-semiample. \end{prop}

The following results are standard (see for example \cite[Section 3]{calculations}). 

\begin{lemma}\label{pushpull} Let $X$ be seminormal and $\mu : Y \to X$ a projective morphism with connected fibers. Then for any coherent sheaf $\mathcal{F}$ on $X$, we have that $\mu_*\mu^* \mathcal{F} = \mathcal{F}$. \end{lemma}

\begin{prop}\label{logcanonical} Let $(X, \Delta)$ be an slc pair and $\mu :Y \to X$ a (partial) semi-resolution. Write
$$
K_Y + \mu_*^{-1}\Delta + \Gamma = \mu^*(K_X + \Delta) + B
$$
where $\Gamma = \sum_i E_i$ is the exceptional divisor of $\mu$ and $B$ is effective and exceptional. Then
$$
\mu_*\calO_Y\big(m(K_Y + \mu_*^{-1}\Delta + \Gamma)\big) \cong \calO_X\big(m(K_X + \Delta)\big). 
$$
\end{prop} 

\begin{cor}\label{bigresolution} Notation as above; the morphism $\mu$ induces an isomorphism of global sections 
$$
H^0\Big(X, \calO_X\big(m(K_X + \Delta)\big)\Big) \cong H^0\Big(Y, \calO_Y\big(m(K_Y + \mu_*^{-1}\Delta + \Gamma)\big)\Big).
$$
In particular, $K_X + \Delta$ is big if and only if $K_Y + \mu_*^{-1}\Delta + \Gamma$ is big. 
\end{cor} 
\begin{proof} The first part is the definition of pushforwards. The second statement follows since $\dim Y = \dim X$. \end{proof} 

\begin{cor}\label{injectivity} Notation as above; the morphism $\mu$ induces an injection 
$$
H^1\Big(X, \calO_X\big(m(K_X + \Delta)\big)\Big) \hookrightarrow H^1\Big(Y, \calO_Y\big(m(K_Y + \mu_*^{-1}\Delta + \Gamma)\big)\Big).
$$ 

\end{cor} 
\begin{proof} This follows from the five-term exact sequence of the Leray spectral sequence for $\mu$ applied to $\calO_Y\big(m(K_Y + \mu_*^{-1}\Delta + \Gamma)\big)$. \end{proof}

Let $(X,D)$ be a pair consisting of a normal variety $X$ and a divisor $D$ such that the rounding up $\lceil D \rceil$ is a reduced divisor. We \emph{do not} assume that $(X,D)$ is log canonical. 

\begin{definition}The \textbf{log canonical model} of a pair $(X,D)$ as above is the log canonical model of the lc pair $(Y, \mu_*^{-1}D + \Gamma)$ where $\mu : Y \to X$ is a log resolution of $(X,D)$ and $\Gamma$ is the exceptional divisor.
\end{definition}

\begin{remark} By Proposition \ref{logcanonical} and its corollaries, the log canonical model of $(X,D)$ is independent of choice of log resolution and therefore is well defined. \end{remark} 

\begin{lemma}\cite[2.35]{km} If $(X,D+D')$ is an lc pair, and $D'$ is an effective $\Q$-Cartier divisor, then $(X,D)$ is also an lc pair. \end{lemma}

\begin{prop}\label{prop:RR} Let $X$ be a smooth projective surface and $D$ a divisor on $X$ such that\\ $H^2(X, \calO_X(D)) = 0$. If $D^2 > 0$, then $D$ is big. \end{prop}

\begin{proof} Applying Riemann-Roch for surfaces and using that $h^2(X, \calO_X(D)) = 0$, we get that 
$$
h^0\big(X, \calO_X(mD)\big) \geq \frac{1}{2} \left(m^2D^2 - mD.K_X\right) + \chi(\calO_X)
$$
so $D$ is big by definition.  \end{proof}

\begin{definition}\label{def:lct} Let $(X,D)$ be be a pair with (semi-)log canonical singularities and $A \subset X$ a divisor. The \textbf{(semi-)log canonical threshold $\mathrm{lct}(X,D,A)$} is 
$$
\mathrm{lct}(X,D,A) := \max\{a \mid (X, D + aA) \text{ has (semi-)log canonical singularities }\}.
$$
\end{definition}

\subsection{Moduli spaces of stable pairs}
\subsubsection{The curve case}\label{sec:hassett}

First we review Hassett's weighted stable curves, as these will be used extensively, and they illuminate some of the basic geometric concepts. 

\begin{definition} Let $\calA = (a_1, \dots, a_r)$ for $0 < a_i \leq 1$. An \textbf{$\calA$-stable curve} is a pair $(C, D = \sum a_i p_i)$, of a reduced connected projective curve $X$ together with a divisor $D$ consisting of $n$ weighted marked points $p_i$ on $C$ such that:
	
	\begin{itemize}
		\item $C$ has at worst nodal singularities, the points $p_i$ lie in the smooth locus of $C$, and for any subset $\{p_1, \cdots, p_s\}$ with nonempty intersection we have $a_1 + \cdots + a_s \leq1$;		\item $\omega_C(D)$ is ample.
	\end{itemize}
\end{definition}

In particular, if $\calA = (1, \dots, 1)$, then one obtains an $r$-pointed stable curve \cite{knudsen}. 

\begin{theorem}\cite{has} Let $\calA = (a_1, \dots, a_r)$ be a weight vector such that $0 < a_i \leq 1$. Suppose $g \geq 0$ is an integer. Then there is a smooth Deligne-Mumford stack $\overline{\calM}_{g, \calA}$ with projective coarse moduli space $\overline{M}_{g, \calA}$ parametrizing $\calA$-stable curves.  \end{theorem}

Moreover, if one considers the domain of admissible weights, there is a \emph{wall and chamber} decomposition -- we say that $(a'_1, \dots, a'_r) \leq (a_1, \dots, a_r)$ if $a'_i \leq a_i$ for all $i$. Hassett proved the following theorem.

\begin{theorem}\label{thm:hassett} \cite{has} There is a wall and chamber decomposition of the domain of admissible weights such that:
	\begin{enumerate}
		\item If $\calA$ and $\calA'$ are in the same chamber, then the moduli stacks and universal families are isomorphic.
		\item If $\calA' \leq \calA$, then there is a reduction morphism  $\overline{\calM}_{g, \calA} \to \overline{\calM}_{g,\calA'}$ and a compatible contraction morphism on universal families.
	\end{enumerate} \end{theorem}

		\subsubsection{Higher dimensions}\label{functors} In full generality, it has been difficult to construct a proper moduli space parametrizing stable pairs $(X,D)$ with suitable numerical data. An example due to Hassett (see Section 1.2 in \cite{kp}), shows that when the coefficients of $D$ are \emph{not} all $> 1/2$, the divisor $D$ might not deform as expected in a flat family of pure codimension 1 subvarieties of $X$ -- the limit of the divisor $D$ may acquire an embedded point. However, we first make the following remarks:
		
		\begin{remark}\label{rmk:moduli} $ $
			\begin{itemize}
				\item Hassett and Alexeev (see \cite{hasproper} and \cite{aleproper}) have demonstrated properness when all coefficients of $D$ are equal to 1. 
				\item It is well known that by results of Koll\'ar, the moduli space exists and is proper when the coefficients are all $> 1/2$ (see e.g. \cite[Sec 4.2]{kollarmodulibook} and \cite{kollarpluri}).
			\end{itemize}  
		\end{remark}
		
		While it is clear what the objects are (see Definition \ref{def:pair}), it is not clear what the proper definition for families are, and thus it is unclear what exactly the moduli functor should be. Many functors have been suggested, but no functor seems to be ``better" than any other. That being said, the projectivity results of \cite{kp}, namely Theorem 1.1 in \emph{loc. cit.}, is independent of the choice of functor, and applies to any moduli functor whose objects are stable pairs. We do remark that Kov\'acs-Patakfalvi demonstrate their results using a proposed functor of Koll\'ar (see Section 5 in \cite{kp}). We also note that it \emph{is} clear what the definition of a stable family (i.e. a family of stable pairs) is over a \emph{normal} base:
		
		\begin{definition}\label{def:stablefamily}\cite[Definition 2.11]{kp} A \textbf{family of stable pairs} of dimension $n$ and volume $v$ over a normal variety $Y$ consists of a pair $(X,D)$ and a flat proper surjective morphism $f: X \to Y$ such that
			\begin{enumerate}
				\item $D$ avoids the generic and codimension 1 singular points of every fiber,
				\item $K_{X/Y} + D$ is $\Q$-Cartier, 
				\item $(X_y, D_y)$ is a connected $n$-dimensional stable pair for all $y \in Y$, and
				\item $(K_{X_y} + D_y)^n = v$ for $y \in Y$.
			\end{enumerate}
			We denote a family of stable pairs by $f : (X,D) \to Y$. 
		\end{definition}
				
If we are satisfied working only over normal bases, then this definition of a family of stable pairs suffices. In fact, any moduli functor $\calM$ with
		$$
		\calM(Y) = \left\{ \parbox{20em}{families of stable pairs $f : (X,D) \to Y$ of dimension $n$ and volume $v$ as in Definition \ref{def:stablefamily}}\ \right\}
		$$
		for $Y$ normal has the same normalization by Proposition \ref{uniquenorm} (see also Definition 5.2 and Remark 5.15 of \cite{kp}). Therefore for many questions about moduli of stable pairs, one needs only consider families over a normal base.

\subsection{Vanishing theorems} 

The existence of reduction morphisms between the moduli spaces will rely on the proof of a vanishing theorem for higher cohomologies which implies invariance of log plurigenera for a family of $\calA$-weighted broken elliptic surfaces (see Section \ref{sec:vanishing}). There are various preliminary vanishing results we will use along the way that we record here for convenience.

The first is a version of Grauert-Riemenschneider vanishing theorem for surfaces. The proof is analagous to the proof of \cite[Theorem 10.4]{singmmp}.  

\begin{prop}(Grauert-Riemenschneider vanishing)\label{prop:gr} Let $X$ be an slc surface and $f : X \to Y$ a proper, generically finite morphism with exceptional curves $C_i$ such that $E = \bigcup_i C_i$ is a connected curve with arithmetic genus $0$. Let $L$ be a line bundle on $X$. Suppose
\begin{enumerate} 
\item  $C_i$ is a $\Q$-Cartier divisor for all $i$;
\item $C_i.E \le 0$ for all $i$; and
\item $\deg(L|_{C_i}) = 0$ for all $i$. 
\end{enumerate}
Then $R^1f_*L = 0$. \end{prop}

\begin{proof} Let $Z = \sum_{i =1}^s r_i C_i$ be an effective integral cycle. Then we prove using induction  that the stalk $(R^1 f_*L)_{Y, p} = \varprojlim_Z H^1(Z, L|_{Z}) = 0$. As $f$ is finite away from $p = f(E)$, this gives $R^1f_*L = 0$. 

Let $C_i$ be an irreducible curve contained in $\mathrm{Supp}(Z)$, and let $Z_i = Z - C_i$. Consider the short exact sequence: $$0 \to \calO_{C_i} \otimes \calO_X(-Z_i) \to \calO_Z \to \calO_{Z_i} \to 0. $$ Tensoring with $L$ we obtain: $$ 0 \to \calO_{C_i} \otimes L(-Z_i) \to L \otimes \calO_Z \to L \otimes \calO_{Z_i} \to 0.$$ By induction on $\sum r_i$, we know that $H^1(Z_i, L|_{Z_i}) = 0$. Therefore, it suffices to show that $H^1(C_i, \calO_{C_i} \otimes L(-Z_i)) = 0$ for some $i$. Moreover, by Serre duality it suffices to show that $$L \cdot C_i - Z_i \cdot C_i > \deg \omega_{C_i} = -2.$$ 

By assumption, $L \cdot C_i = 0$, so it suffices to show that $-Z_i \cdot C_i > -2$, i.e. that $Z_i \cdot C_i < 2$. This follows from Artin's results on intersection theory of exceptional curves for rational surface singularities \cite{artin} applied to the normalization of $X$, as $C_i$ and $E$ are rational exceptional curves.  \end{proof}

Next we will use Fujino's generalization of the Kawamata-Viehweg vanishing theorem for slc pairs. Before stating the result, we will need to make a preliminary definition. 

\begin{definition}
Let $(X, \Delta)$ be a semi-log canonical pair and let $\nu: X^{\nu} \to X$ be the normalization. Let $\Theta$ be a divisor on $X^{\nu}$, so that $(K_{X^{\nu}} + \Theta) = \nu^*(K_X + \Delta)$. A subvariety $W \subset X$ is called an \textbf{slc center} of $(X,\Delta)$ if there exists a resolution of singularities $f: Y \to X^{\nu}$ and a prime divisor $E$ on $Y$ such that the discrepancies $a(E, X^{\nu}, \Theta) = -1$ and $\nu \circ f(E) = W$. A subvariety $W \subset X$ is called an \textbf{slc stratum} if $W$ is an slc center, or an irreducible component.\end{definition}

Now we state Fujino's theorem.

\begin{theorem}\label{fujinovanishing} \cite[Theorem 1.10]{fundamental} Let $(X,\Delta)$ be a projective semi-log canonical pair, $L$ a $\Q$-Cartier divisor whose support does not contain any irreducible components of the conductor, and  $f : X \to S$ a projective morphism. Suppose $L - (K_X + \Delta)$ is $f$-nef and additionally is $f$-big over each slc stratum of $(X,\Delta)$. Then $R^if_*\calO_X(L) = 0$ for $i > 0$. 
\end{theorem}

We begin with general definitions, properties, and results on elliptic surfaces and their log canonical models.

\subsection{Standard elliptic surfaces} We point the reader to \cite{mir3} for a detailed exposition on the theory of elliptic surfaces. 

\begin{definition}\label{def:standardes} An irreducible \textbf{elliptic surface with section} $(f : X \to C, S)$ is an irreducible surface $X$ together with a surjective proper flat morphism $f: X \to C$ to a proper smooth curve and a section $S$ such that:
\begin{enumerate}
\item the generic fiber of $f$ is a stable elliptic curve, and
\item the generic point of the section is contained in the smooth locus of $f$. 
\end{enumerate}
We say $(f: X \to C, S)$ is \textbf{standard} if all of $S$ is contained in the smooth locus of $f$. 

\end{definition} 

This definition differs from the usual definition of an elliptic surface in that we only require the generic fiber to be a \emph{stable} elliptic curve. 

\begin{definition} A \textbf{Weierstrass fibration} $(f : X \to C, S)$ is an elliptic surface with section as above, such that the fibers are reduced and irreducible. \end{definition}

\begin{definition}
	A surface is \textbf{semi-smooth} if it only has the following singularities:\begin{enumerate}
		\item 2-fold normal crossings (locally $x^2 = y^2$), or
		\item pinch points (locally $x^2 = zy^2$).
	\end{enumerate}
\end{definition}

\begin{definition}
	A \textbf{semi-resolution} of a surface $X$ is a proper map $g: Y \to X$ such that $Y$ is semi-smooth and $g$ is an isomorphism over the semi-smooth locus of $X$.
\end{definition}

\begin{definition}
	An elliptic surface is called \textbf{relatively minimal} if it is semi-smooth and there is no $(-1)$-curve in any fiber. 
\end{definition}

Note that a relatively minimal elliptic surface with section is standard. \\

If $(f:X \to C, S)$ is a standard elliptic surface then there are finitely many fiber components not intersecting the section. We can contract these to obtain an elliptic surface with all fibers reduced and irreducible: 

\begin{definition} \label{weierstrassfibration} If $(f:X \to C, S)$ is a standard elliptic surface then the Weierstrass fibration $f' : X' \to C$ with section $S'$ obtained by contracting any fiber components not intersecting $S$ is the \textbf{Weierstrass model} of $(f:X\to C, S)$. If $(f : X \to C, S)$ is relatively minimal, then we refer to $f' : X' \to C$ as the \textbf{minimal Weierstrass model}. \end{definition}

\begin{definition}\label{linebundle} The \textbf{fundamental line bundle} of a standard elliptic surface $(f : X \to C, S)$ is $\LL := (f_*N_{S/X})^{-1}$ where $N_{S/X}$ denotes the normal bundle of $S$ in $X$. For $(f : X \to C, S)$ an arbitrary elliptic surface, we define $\LL := (f'_*N_{S'/X'})^{-1}$ where $(f' : X' \to C, S')$ is a minimal semi-resolution. \end{definition}

Since $N_{S/X}$ only depends on a neighborhood of $S$ in $X$, the line bundle $\LL$ is invariant under taking a semi-resolution or the Weierstrass model of a standard elliptic surface. Therefore $\LL$ is well defined and equal to $(f'_*N_{S'/X'})^{-1}$ for $(f' : X' \to C, S')$ a minimal semi-resolution of $(f : X \to C, S)$. 

The fundamental line bundle greatly influences the geometry of a minimal Weierstrass fibration. The line bundle $\LL$ has non-negative degree on $C$ and is independent of choice of section $S$ \cite{mir3}. Furthermore, $\LL$ determines the canonical bundle of $X$:

\begin{prop}\cite[Proposition III.1.1]{mir3}\label{classcan} Let $(f : X \to C, S)$ be either \begin{enumerate*} \item a Weierstrass fibration, or \item a relatively minimal smooth elliptic surface \end{enumerate*}. Then
$
\omega_X = f^*(\omega_C \otimes \LL).
$
\end{prop}

We prove a more general canonical bundle formula in \cite{calculations} (see Proposition \ref{canonical}).

\begin{definition} We say that $f : X \to C$ is \textbf{properly elliptic} if $\deg(\omega_C \otimes \LL) > 0$. \end{definition}

We note that $X$ is properly elliptic if and only if the Kodaira dimension $\kappa(X) = 1$. 

\subsection{Singular fibers}

When $(f : X \to C, S)$ is a smooth relatively minimal elliptic surface, then $f$ has finitely many singular fibers. These are unions of rational curves with possibly non-reduced components whose dual graphs are $ADE$ Dynkin diagrams. The possible singular fibers were classified independently by Kodaira and Ner\'on.

Table \ref{table} gives the full classification in Kodaira's notation for the fiber.  Fiber types $\mathrm{I}_n$ for $n \geq 1$ are reduced and normal crossings, fibers of type $\mathrm{I}^*_n, \mathrm{II}^*, \mathrm{III}^*$, and $\mathrm{IV}^*$ are normal crossings but nonreduced, and fibers  of type $\mathrm{II}, \mathrm{III}$ and $\mathrm{IV}$ are reduced but not normal crossings. 

For $f : X \to C$ isotrivial with $j = \infty$, La Nave classified the Weierstrass models with log canonical singularities in \cite[Lemma 3.2.2]{ln} (see also \cite[Section 5]{calculations}). They have equation $y^2 = x^2(x - t^k)$ for $k = 0,1$ and $2$ and we call these $\mathrm{N}_0$, $\mathrm{N}_1$ and $\mathrm{N}_2$ fibers respectively. 



\begin{table}
\centering
\caption{Kodaira's classification of singular fibers of a smooth minimal elliptic surface}\label{table:kodaira}
\label{table}
\begin{tabular}{|l|l|l|}
\hline
Kodaira Type        & \# of components           & Fiber                                                                                                 \\ \hline
$\mathrm{I}_0$               & 1                          & \begin{minipage}{.3\textwidth}\includegraphics[scale=.45]{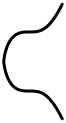}

\end{minipage}    \\ \hline

$\mathrm{I}_1$               & 1 (double pt)              & \begin{minipage}{.3\textwidth}\includegraphics[scale=.5]{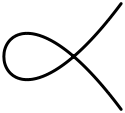}\end{minipage}   \\ \hline

$\mathrm{I}_2$               & 2 (2 intersection pts)     & \begin{minipage}{.3\textwidth}\includegraphics[scale=.5]{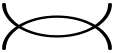}\end{minipage}   \\ \hline

$\mathrm{I}_n$, $n \geq 2$   & $n$ ($n$ intersection pts) & \begin{minipage}{.3\textwidth}\includegraphics[scale=.5]{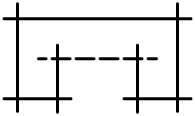}\end{minipage}   \\ \hline
$\mathrm{II}$                & 1 (cusp)                   & \begin{minipage}{.3\textwidth}\includegraphics[scale=.5]{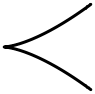}\end{minipage}     \\ \hline
$\mathrm{III}$               & 2 (meet at double pt)      & \begin{minipage}{.3\textwidth}\includegraphics[scale=.5]{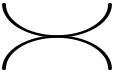}\end{minipage} \\ \hline
$\mathrm{IV}$                & 3 (meet at 1 pt)           & \begin{minipage}{.3\textwidth}\includegraphics[scale=.5]{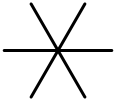}\end{minipage}    \\ \hline
$\mathrm{I}_0^*$             & 5                          & \begin{minipage}{.3\textwidth}\includegraphics[scale=.5]{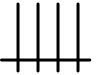}\end{minipage}     \\ \hline
$\mathrm{I}_n^*$, $n \geq 1$ & $5 + n$                    & \begin{minipage}{.3\textwidth}\includegraphics[scale=.5]{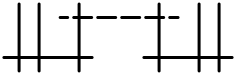}\end{minipage}   \\ \hline
$\mathrm{II}^*$              & 9                          & \begin{minipage}{.3\textwidth}\includegraphics[scale=.5]{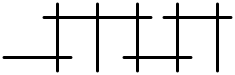}\end{minipage}   \\ \hline
$\mathrm{III}^*$             & 8                          & \begin{minipage}{.3\textwidth}\includegraphics[scale=.5]{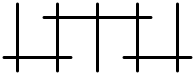}\end{minipage}    \\ \hline
$\mathrm{IV}^*$              & 7                          & \begin{minipage}{.3\textwidth}\includegraphics[scale=.5]{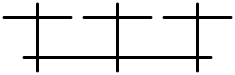}\end{minipage}    \\ \hline
$\mathrm{N}_{\mathrm{I}}$ & 1 (node) & \begin{minipage}{.3\textwidth}\includegraphics[scale=.5]{ii}\end{minipage}     \\ \hline

\end{tabular}
\end{table}

\section{Log canonical models of elliptic surfaces}\label{sec:ellipticbackground}

We begin with a discussion of results from \cite{calculations}.

\subsection{Log canonical models of $\calA$-weighted elliptic surfaces}\label{sec:local}

Let $\calA = (a_1, \ldots, a_r) \in (\Q \cap [0,1])^r$ be a rational \emph{weight vector} with $0 \leq a_i \leq 1$. We will consider \emph{Weierstrass} elliptic surfaces marked by an $\calA$-weighted sum $F_\calA = \sum_{i = 1}^r a_i F_i$.
Note that the weights come with a natural partial ordering. We say that $\cA = (a'_1, \dots, a'_r) < \calA$ if $a_i' \leq a_i$ for all $i$, and if the inequality is strict for at least one $i$. If $s \in \mathbb{Q}$ is a rational number, we write $\calA \le s$ ($\calA \geq s$) if $a_i \le s$ ($a_i \geq s$) for all $i$. Our goal is to compare stable pair compactifications of the moduli space of $\calA$-weighted elliptic surface pairs for various weight vectors $\calA$.

As a first step, we need to understand the log canonical models of Weierstrass elliptic surface pairs and how they depend on the weights $\calA$. That is, given a Weierstrass elliptic surface pair $(g:Y\to C, S)$ and an $\calA = (a_1, \ldots, a_n)$-weighted sum of marked fibers $$F_{\calA} = \sum_i a_i F_i,$$ we need to compute the log canonical model for all weights $\calA$. This is based on the computations in \cite{calculations}. 

Our study of log canonical models of an elliptic surface pair $(f : X \to C, S + F_\calA)$ proceeds in two steps: first we compute the relative canonical model of $(X, S + F_\calA)$ over the curve $C$ and then contract the section or whole components if necessary according to the log minimal model program.

Let $(g: Y \to C, S + F_\calA)$ be an $\calA$-weighted Weierstrass elliptic fibration over a smooth curve. We want to compute the relative log canonical model of the pair $(Y, S + F_\calA)$ relative to the fibration $g$. That is, we wish to take a suitable log resolution $\mu : Y' \to Y$ and compute the log canonical model of $(Y', \mu_*^{-1}S + \mu_*^{-1}(F_\calA) + \mathrm{Exc}(\mu))$ relative to $g \circ \mu : Y' \to C$.  In what follows, unless otherwise specified, by relative log canonical model we mean relative to the base curve $C$.

This computation is local on the base so for the rest of this subsection, we assume that $C = \Spec(R)$ is a spectrum of a DVR with closed point $s$ and generic point $\eta$. We then consider the log pair $(Y,S + aF)$ where $F = g^*(s)$ and $0 \le a \le 1$.

\begin{definition}\label{canonicalweir}\cite[Definition 3.2.3]{ln} A normal Weierstrass elliptic fibration $(g : Y \to C, S)$ over the spectrum of a DVR with Weierstrass equation $y^2 = x^3 + ax + b$ is called a \textbf{standard Weierstrass model} if $\mathrm{min}(\mathrm{val}(3n), \mathrm{val}(2m)) \le 12$. A non-normal Weierstrass fibration with equation $y^2 = x^2(x - at^k)$ is called a \textbf{standard Weierstrass model} if $k \le 2$. \end{definition} 

\begin{prop}\cite[Corollary 3.2.4]{ln} A Weierstrass elliptic fibration $g : Y \to C$ over the spectrum of a DVR is (semi-)log canonical if and only if it is a standard Weierstrass model. \end{prop}

\begin{definition}\label{def:twisted} Let $(g: Y \to C, S' + aF')$ be a Weierstrass elliptic surface pair over the spectrum of a DVR and let $(f: X \to C, S + F_a)$ be its relative log canonical model. We say that $X$ has a:  
\begin{enumerate} 
\item \textbf{twisted fiber} if the special fiber $f^*(s)$ is irreducible and $(X,S + E)$ has (semi-)log canonical singularities where $E = f^*(s)^{red}$;   
\item \textbf{intermediate fiber} if $f^*(s)$ is a nodal union of an arithmetic genus zero component $A$, and a possibly non-reduced arithmetic genus one component supported on a curve $E$ such that the section meets $A$ along the smooth locus of $f^*(s)$ and the pair $(X, S + A + E)$ has (semi-)log canonical singularities. 
\item  \textbf{standard (resp. minimal) intermediate fiber} if $(g : Y \to C, S)$ is a standard (resp. minimal) Weierstrass model.
\end{enumerate} 
\end{definition}

Let $X$ be the relative log canonical model of $(Y, S' + aF') \to C$, and let $\mu : X \dashrightarrow Y$ be the birational map to $Y$. Then the divisor $E$ in both the twisted and intermediate cases is an exceptional divisor for $\mu$. Therefore $\mathrm{lct}(X,0,E) = 1$ and $E$ appears with coefficient one in the log canonical pair $(X, \mu^{-1}_*(S' + aF') + \mathrm{Exc}(\mu))$. In particular, $F_a = \mu^{-1}_*(aF') + \mathrm{Exc}(\mu)$ contains $E$ with coefficient one.

\begin{lemma}\label{lemma:twistedboundary} Let $(f : X \to C, S + F_a)$ be the relative log canonical model of a Weierstrass model and suppose that $f : X \to C$ has twisted central fiber. Then $F_a = E$. \end{lemma} 

\begin{proof} The boundary divisor of the log canonical model is given by $\mu^{-1}_*(S' + aF') + \mathrm{Exc}(\mu)$ where $\mu : X \dashrightarrow Y$ is the natural birational map. Then $F_a = \mu^{-1}_*(aF') + \mathrm{Exc}(\mu)$ is a divisor supported on the fiber of $f$ and contains $E$ with coefficient one. Since the fiber is twisted then $F_a = E$. 
\end{proof} 

The terminology for a twisted fiber comes from the fact that these fibers are exactly those that appear in the coarse space of a flat family $\calX \to \calC$ of \emph{stable} elliptic curves over a \emph{orbifold} base curve $\calC$. Equivalently, a twisted fiber is obtained by taking the quotient of a family of stable curves over the spectrum of a DVR by a subgroup of the automorphism group of the central fiber. This notion is introduced in \cite{av} for the purpose of obtaining fibered surfaces from twisted stable maps. In \cite[Proposition 4.12]{tsm} it is proved that any twisted fiber pair $(f: X \to C, S + E)$ as in the conclusion of Lemma \ref{lemma:twistedboundary} is obtained as the coarse space of family of stable curves over an orbifold curve. Moreover, twisted models exist and are unique.

\begin{lemma}\label{lemma:twisted} Let $(g : Y \to C, S')$ be a Weierstrass elliptic surface over the spectrum of a DVR and suppose that there exists a birational model of $Y$ with an intermediate fiber. Then the following birational models are isomorphic:
\begin{enumerate}
    \item the twisted model $(f_1 : X_1 \to C, S_1 + E_1)$,
    \item the log canonical model of the intermediate model $(f : X \to C, S + A + E)$ with coefficient one,
    \item the log canonical model of the Weierstrass fiber $(g : Y \to C, S' + F')$ with coefficient one. 
\end{enumerate}
Moreover, in this case there is a morphism $X \to X_1$ contracting the component $A$ to a point. 
\end{lemma} 

\begin{proof} The contraction of $\mu: X \to Y$ to its Weierstrass model provides a log canonical partial resolution of $(Y, S' + F')$ with boundary divisor $\mu^{-1}_*(S' + F') = S + A + E$ and so $(2)$ and $(3)$ agree by definition of log canonical model. Now the pair $(X, S + A + E)$ has log canonical singularities and so we may run an mmp and use abundance to compute its relative log canonical model $\mu_0 : X \to X_0$. Note $\mu_0$ is a morphism since $X$ is a surface. Then $(X_0, \mu_{0_*}(S + A + E))$ is a relative log canonical model with fiber marked with coefficient one and so it must be the twisted model by \cite[Proposition 4.12]{tsm}. 
\end{proof}

\begin{lemma} Let $(f : X \to C, S)$ be a standard intermediate model. Then the relative log canonical model of $f: (X, S + aA + E) \to C$ is the contraction $\mu : X \to X'$ of $E$ to a point with Weierstrass fiber $A' = \mu_*A$ for any $0 \le a \le \mathrm{lct}(Y,S',F')$ where $(Y \to C, S')$ is the corresponding standard Weierstrass model. \end{lemma} 

\begin{proof} By construction a standard intermediate model maps onto the corresponding standard Weierstrass model -- call this map $\mu : X \to Y$. Then $\mu_*A = F'$ and $\mu_*E = 0$. In particular, $\mu: (X, S + aA + E) \to (Y, S' + aF')$ is a log resolution of $(Y, S' + aF')$ for any $a$. If $a \le \mathrm{lct}(Y,S',F')$ then $(Y, S' + aF')$ is log canonical and $\mu$ is the relative log canonical model of $(X, S + aA + E) \to C$. As $(Y, S')$ is log canonical, then $\mathrm{lct}(Y,S', F') \geq 0$.  \end{proof}

\begin{prop}\label{prop:lct} Let $(g : Y \to C, S')$ be a standard Weierstrass model with central fiber $F'$. There exists a number $b_0$ such that $\mathrm{lct}(Y,S', F') < b_0 \le 1$ and the relative log canonical model of $(Y,S' + aF') \to C$ is 
\begin{enumerate}
\item a standard intermediate fiber for $\mathrm{lct}(Y, S', F') < a < b_0$;
\item a twisted fiber for $b_0 \le a \le 1$. 
\end{enumerate}
\end{prop}
\begin{proof} Standard intermediate models for a standard Weierstrass model are computed to exist (by taking a log resolution and blowing down extra components) in \cite{calculations}. Furthermore, there it is shown that a standard intermediate model $(X \to C, S)$ has a contraction $p : X \to X'$ contracting the $A$ component onto a twisted model with central fiber $E' = p_*E$. Now $(X',S + E')$ has log canonical singularities and $p$ is a partial log resolution so $(X',E')$ is the relative log canonical model of $(X, S + \mathrm{Exc}(p) + E) = (X, S + A + E)$. 

On the other hand, $\mu : X \to Y$ is a log resolution of the pair $(Y, S' + aF')$ with boundary $\mu^{-1}_*(S' + aF') + \mathrm{Exc}(\mu) = S + aA + E$. Thus when $a = 1$, the relative log canonical model of $(Y, S' + F')$, which is equal to the relative log canonical model of the log resolution $(X, S + A + E)$, is the twisted fiber. 

Furthermore, $A$ is an exceptional divisor of the log resolution $p : X \to X'$ and so the intersection number $A.(K_X + S + A + E) \le 0$ and similarly $E$ is exceptional for $\mu : X \to Y$ and $E.(K_X + S + a_0A + E) = 0$ for $a_0 = \mathrm{lct}(Y, S', F')$. Thus by linearity of intersection numbers, there is a $b_0$ such that $a_0 \le b_0 \le 1$ and 
\begin{align*}
A.(K_X + S + aA + E) & > 0\\
E.(K_X + S + aA + E) & > 0
\end{align*}
for any $a_0 < a < b_0$.  \end{proof} 

\begin{prop} Let $(g: Y \to C, S')$ be a non-standard Weierstrass model for which there exists an intermediate model. Then there is a number $0 \le b_0 \le 1$ such that the relative log canonical model of $(Y, S' + aF')$ is 
\begin{enumerate}
\item a standard intermediate fiber for $0 \le a < b_0$;
\item a twisted fiber for $b_0 \le a \le 1$. 
\end{enumerate}
\end{prop}

\begin{proof} By Lemma \ref{lemma:twisted}, the log canonical model is the twisted model for $a = 1$. Now consider the contraction $\mu : X \to Y$ from the intermediate to the Weierstrass model. The pair $(X, S + aA + E)$ is a log canonical resolution of $(Y, S')$ so we may compute the relative log canonical model $(Y,S')$ by computing that of $(X, S + aA + E)$.

$Y$ is Gorenstein since it is cut out by a single Weierstrass equation and so $K_Y$ is Cartier. Furthermore, since $g$ is a genus one fibration, $K_Y$ must be supported on fiber components. We may write 
$$
\mu^*(K_Y) = K_X + \alpha E
$$
where $\alpha > 1$ since the singularities of $(Y,0)$ are \emph{not} log canonical. It follows that $$(K_X + S + E).E = (1-\alpha) E^2 > 0$$ so that the first step of the log MMP does not contract $E$. As $a$ increases, this intersection number also increases and so $E$ is never contracted in the first step of the log MMP.

If $(K_X + S + E).A = 1 + (1 - \alpha)A.E \le 0$, then $A$ is contracted by either the log MMP or the log canonical linear series for all $a \geq 0$ as increasing $a$ decreases this intersection number. Thus the log canonical model is the twisted model for all $a$ and $b_0 = 0$. 

Otherwise if $(K_X + S + E).A > 0$, then the $K_X + S + E$ is already ample and so the intermediate fiber is the log canonical model for $a = 0$. By linearity of intersection numbers there is a unique $b_0$ such that $(K_X + S + b_0A + E).A = 0$ and this $b_0$ has the required property.  \end{proof}

To summarize, given a standard Weierstrass model $(g : Y \to C, S' + aF')$ over the spectrum of a DVR, there is a standard intermediate model $g : X \to C$ which maps to the Weierstrass model by contracting the component $E$, and maps to the twisted model by contracting the component $A$. Thus the intermediate fiber can be seen as interpolating between the relative log canonical model being Weierstrass and twisted as the coefficient $a$ varies (see Figure \ref{figure:2}). For a non-standard Weierstrass model, there is a similar picture except the Weierstrass model is never log canonical and so there is only a single transition from intermediate to twisted. 

\begin{figure}[h!]
	\includegraphics[scale=1]{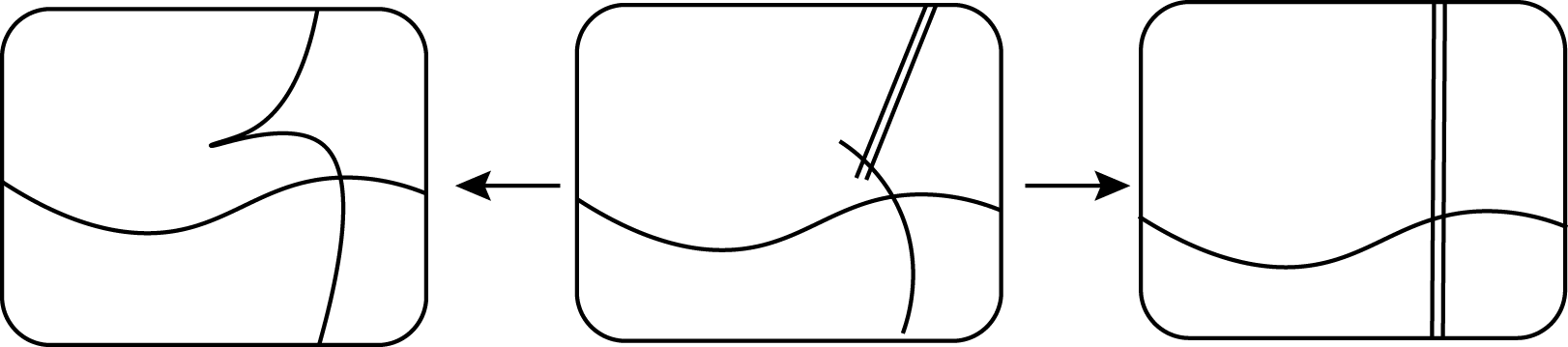}
	\caption{Here we illustrate the relative log canonical models and morphisms between them. From left to right: \textbf{standard Weierstrass} model ($ 0 \leq a \leq a_0$) -- a single reduced and irreducible component meeting the section, \textbf{standard intermediate} model ($a_0 < a < b_0 \le 1$) -- a nodal union of a reduced component meeting the section and a nonreduced component, and \textbf{twisted} model ($b_0 \le a \le 1$) -- a single possibly nonreduced component meeting the section in a singular point of the surface.}\label{figure:2}
\end{figure} 

\begin{remark} Note since $E$ is a log canonical center of the intermediate fiber pair $(X, S + aA + E)$, then $(E, (S + aA + E)|_E)$ is itself a log canonical pair. In particular, $E$ must be at worst nodal. Since $E$ is irreducible then either the intermediate fiber is reduced and $E$ is a stable elliptic curve, or $E$ supports a nonreduced arithmetic genus one component so $E$ is a smooth rational curve. \end{remark} 

The calculations in \cite{calculations} allow us to make precise the coefficients $a_0 = \mathrm{lct}(Y, S', F')$ and $b_0$ where the transitions from the various fiber models occur for the minimal Weierstrass models (see Table \ref{table:kodaira}). We summarize the calculations here and direct the reader to \cite{calculations} for more details. 

\begin{table}[h!]
  \begin{center}
    \caption{Intersection pairings in a standard intermediate fiber}
    \label{tab:table1}
    \begin{tabular}{|l|c|c|c|c|} 
      \hline
      \textbf{Singular fiber} & $A^2$ & $E^2$ &$A.E$ & Mult. of $E$ in $f^{-1}(p)$ \\
      \hline
      $\mathrm{I}_n^*$ & $-2$ & $-1/2$ &1 &2 \\
      \hline
      $\mathrm{II}$ & $-6$ &$-1/6$ & 1 &6\\
      \hline
      $\mathrm{III}$ & $-4$ & $-1/4$ & 1 &4\\
      \hline
      $\mathrm{IV}$ & $-3$ & $-1/3$ & 1 &3\\
      \hline
      $\mathrm{II}^*$& $-6/5$ & $-1/30$ & $1/5$ &6\\
       \hline
      $\mathrm{III}^*$& $-4/3$ & $-1/12$ & $1/3$ &4\\
       \hline
      $\mathrm{IV}^*$& $-3/2$ & $-1/6$ & $1/2$ &3\\
      \hline
    \end{tabular}
  \end{center}
\end{table}

\begin{theorem}\label{thm:thm1} Let $(g: Y \to C, S' + aF)$ be a standard Weierstrass model over the spectrum of a DVR, and let $(f: X \to C, S + F_a)$ be the relative log canonical model. Suppose the special fiber $F'$ of $g$ is either either 
\begin{enumerate*}[label = (\alph*)] \item one of the Kodaira singular fiber types, or \item $g$ is isotrivial with constant $j$-invariant $\infty$ and $F'$ is an $N_0$ or $N_1$ fiber. \end{enumerate*}

\begin{enumerate} 
\item If $F$ is a type $\mathrm{I}_n$ or $\mathrm{N}_0$ fiber, then the relative log canonical model is the Weierstrass model for all $0 \le a \le 1$.
\item For any other fiber type, there is an $a_0$ such that the relative log canonical model is
\begin{enumerate}[label = (\roman*)]
\item the \emph{Weierstrass} model for any $0 \le a \le a_0$,
\item a \emph{twisted fiber} consisting of a single non-reduced component supported on a smooth rational curve when $a = 1$, and
\item a \emph{standard intermediate fiber} with $E$ a smooth rational curve for any $a_0 < a < 1$.
\end{enumerate}
\end{enumerate}

The constant $a_0$ is as follows for the other fiber types:
$$
a_0 = \left\{ \begin{array}{lr} 5/6 & \mathrm{II} \\ 3/4 & \mathrm{III} \\ 2/3 & \mathrm{IV} \\  1/2 & \mathrm{N}_1 \end{array}\right. \\
\  a_0 = \left\{ \begin{array}{lr} 1/6 & \mathrm{II}^* \\ 
1/4  & \mathrm{III}^* \\
1/3 & \mathrm{IV}^* \\
1/2 & \mathrm{I}_n^* \end{array}\right. 
$$
\end{theorem}

\begin{remark} The difference between Theorem \ref{thm:thm1} and the corresponding theorem in \cite{calculations}, is that here we are marking the $E$ component of the intermediate fiber with coefficient one rather than marking it by $a$. By the above discussion, this is equivalent to taking the log canonical model of the Weierstrass pair rather than taking the log canonical models of the minimal resolution as in \cite{calculations}. With this in mind, the result above for types $\mathrm{II}, \mathrm{III}, \mathrm{IV}$ and $\mathrm{N}_1$ fibers are unchanged as for these types of fibers, the coefficient of $E$ was already one in \cite{calculations} and the results for $\mathrm{I}_n$ and $\mathrm{N}_0$ are unchanged as these fibers are already log canonical models regardless of coefficient. The reason for this change in convention is to avoid an unwanted flip (see Appendix \ref{sec:appendix}). \end{remark}

\begin{conv}\label{conv:lcmodel} Let $(f : X \to C, S + F_a)$ be an elliptic fibration over the spectrum of a DVR with section $S$, central fiber $F$, and boundary divisor $F_a$ supported on $F$. From now on we will say this pair is a relative log canonical if it is the relative log canonical model of a Weierstrass model. That is, either 
\begin{enumerate} 
\item $F$ is a Weierstrass fiber and $F_a = aF$ for $a \le \mathrm{lct}(X,S,F)$, 
\item $F$ is a standard intermediate fiber with $F_a = aA + E$ and $\mathrm{lct}(Y,S', F') < a < 1$, or 
\item $F$ is a twisted fiber with $F_a = E$ and $a = 1$. 
\end{enumerate}

More generally, we will say $(f : X \to C, S + F)$ over a smooth curve is a \textbf{relative log canonical model} or \textbf{relatively stable} if it is the relative log canonical model of its Weierstrass model so that the restriction of $(X, S + F)$ to the local ring of each point in $C$ is a relative log canonical model of Weierstrass, intermediate or twisted type. We will call it \textbf{standard (resp. minimal)} if each of the fibers are log canonical models of standard (resp. minimal) Weierstrass models. 
\end{conv}

\subsection{Canonical bundle formula}\label{sec:canonicalbundle} In \cite{calculations}, we computed a formula for the canonical bundle of relative log canonical model. 

\begin{theorem}\cite[Theorem 1.2]{calculations}\label{canonical}  Let $(f: X \to C, S + F_\calA)$ be a relative log canonical model where $f: X \to C$ is a minimal irreducible elliptic surface with section $S$, and let $F_{\calA} = F_{a_i}$ is a sum of marked fibers as in \ref{conv:lcmodel} with $0 \le a_i \le 1$. Then
$$
\omega_X = f^*(\omega_C \otimes \LL) \otimes \calO_X(\Delta).
$$
where $\Delta$ is effective and supported on fibers of type $\mathrm{II}, \mathrm{III}$, and $\mathrm{IV}$ contained in $\Supp(F)$. The contribution of a type $\mathrm{II}$, $\mathrm{III}$ or $\mathrm{IV}$ fiber to $\Delta$ is given by $\alpha E$ where $E$ supports the unique nonreduced component of the fiber and
$$
\alpha = \left\{ \begin{array}{lr} 4 & \mathrm{II} \\ 2 & \mathrm{III} \\ 1 & \mathrm{IV} \end{array} \right.
$$
\end{theorem}

It is important to emphasize here that only type $\mathrm{II}$, $\mathrm{III}$ or $\mathrm{IV}$ fibers that are \emph{not} in Weierstrass form affect the canonical bundle. If all of the type $\mathrm{II}$, $\mathrm{III}$ and $\mathrm{IV}$ fibers of $f : X \to C$ are Weierstrass, then the usual canonical bundle formula $\omega_X = f^*(\omega_C \otimes \LL)$ holds.

\subsection{Pseudoelliptic contractions}\label{sec:pseudocontraction}

In \cite{ln}, La Nave studied compactifications of the moduli space of Weierstrass fibrations by stable elliptic surface pairs $(f: Y \to C, S)$ -- i.e. where $\calA= 0$. There it was shown that the section of some irreducible components of a reducible elliptic surface may be contracted by the log MMP, inspiring the following.
		
\begin{definition}\label{def:pseudo}
A \textbf{pseudoelliptic surface} is a surface $X$ obtained by contracting the section of an irreducible elliptic surface pair $(f:Y \to C,S)$. For any fiber of $f$, we call its pushforward via $\mu : Y \to X$ a \textbf{pseudofiber} of $X$. We call $(f : Y \to C, S)$ the \textbf{associated elliptic surface} to $X$. If $(f : Y \to C, S + F'_\calA)$ is an $\calA$-weighted relative log canonical model then we call $(X, F_\calA)$ a \textbf{pseudoelliptic pair} where $F_\calA = \mu_*F'_\calA$. 
\end{definition}

In the next section we will discuss when a pseudoelliptic surface forms. That is, for which $\calA$ does the minimal model program necessitate that the section of a relative log canonical model $(f : Y \to C, S + F_\calA)$ contracts to form a pseudoelliptic surface? 

In this section we are tasked with understanding when the log canonical contraction of pseudoelliptic pair $(X, F_\calA)$ corresponding to a relatively stable elliptic surface contracts $X$ to a lower dimensional variety.

\begin{prop}\label{properlyelliptic}\cite[Proposition 7.1]{calculations} Let $f: Y \to C$ be an irreducible properly elliptic surface with section $S$. Then $K_Y + S$ is big.
\end{prop}

\begin{cor} If $X$ is a log canonical pseudoelliptic surface such that the associated elliptic surface $\mu : Y \to X$ is properly elliptic, then $K_X$ is big. \end{cor} 

\begin{proof} $\mu : Y \to X$ is a partial log resolution so $K_Y + \mathrm{Exc}(\mu) = K_Y + S$ is big if and only if $K_X$ is big by Corollary \ref{bigresolution}.
\end{proof}

\begin{definition} 
The \textbf{fundamental line bundle $\LL$} of a pseudoelliptic surface $X$ is the fundamental line bundle (see Definition \ref{linebundle}) for $(f : Y \to C,S)$ the corresponding elliptic surface.
\end{definition} 

\begin{prop}\label{rational}\cite[Proposition 7.4]{calculations} Let $(X, F_\calA)$ be an $\calA$-weighted slc pseudoelliptic surface pair corresponding to an elliptic surface $(f:Y \to C,S)$ over a rational curve $C \cong \mathbb{P}^1$.  Denote by $\mu : Y \to X$ the contraction of the section. Suppose $\deg \LL = 1$ and $0 \le \calA \le 1$ such that $K_X + F_\calA$ is a nef and $\Q$-Cartier. Then either 

\begin{enumerate}[label = \roman*)]
\item $K_X + F_\calA$ is big and the log canonical model is an elliptic or pseudoelliptic surface;
\item $K_X + F_\calA \sim_\Q \mu_*\Sigma$ where $\Sigma$ is a multisection of $Y$ and the log canonical contraction maps $X$ onto a rational curve; or
\item $K_X + F_\calA \sim_{\Q} 0$ and the log canonical map contracts $X$ to a point. 
\end{enumerate} 

The cases above correspond to $K_X + F_\calA$ having Iitaka dimension $2,1$ and $0$ respectively. 

\end{prop}

\begin{remark}\label{rem:contract} The proof of \cite[Proposition 7.4]{calculations} actually gives a method for determining which situation of $(i), (ii),$ and $(iii)$ we are in. Indeed since $K_X + F_\calA$ is nef, it is big if and only if $(K_X + F_\calA)^2 > 0$. Furthermore, $K_X + F_\calA \sim_\Q 0$ if and only if $t = 0$ where
$$
K_Y + tS + \tilde{F}_\calA = \mu^*(K_X + F_\calA).
$$
So if $K_X + F_\calA$ is not big,  it suffices to compute whether $t > 0$ or $t = 0$ to decide if the log canonical map contracts the pseudoelliptic to a curve or to a point.  \end{remark}

\begin{prop}\label{stableattaching} Let $(f: X \to C, S + F_\calB)$ be an irreducible elliptic surface over a rational curve $C \cong \mb{P}^1$ such that $(X, S + F_\calB)$ is a stable pair and $\deg \LL = 1$. Suppose $\calB = (1, b_2, \ldots, b_s)$, $0 < \calA \le \calB$ such that $a_1 = b_1 = 1$, and $F_1$ is a type $\mathrm{I}_n$ fiber. Then 
$K_X + S + F_\calA$
is big. 
\end{prop} 

\begin{proof} By the canonical bundle formula, $K_X = -G + \Delta$ for $G$ a general fiber and $\Delta$ effective. All fibers are linearly equivalent as $C$ is rational, and type $\mathrm{I}_n$ fibers are reduced so that $F_1 \sim_\Q G$. Thus $K_X + F_1 = \Delta$ is effective and 
$$
K_X + S+ F_\calB = \Delta + S + \sum_{i = 2}^s F_{b_i},$$
$$K_X + S+ F_\calA = \Delta + S + \sum_{i = 2}^s F_{a_i}
$$
with $0 < a_i \le b_i$ with $$F_a = \left\{\begin{array}{l} aF \\ aA + E \\ E\end{array}\right.$$ depending on whether $F$ is a Weierstrass, intermediate or twisted fiber. Furthermore $K_X + S + F_\calB$ ample. Since $a_i > 0$, for $m$ large enough we can write $m(K_X + S + F_\calA) - (K_X + S + F_\calB) = D$ where $D$ is effective. Therefore $K_X + S + F_\calA$ is big by Kodaira's lemma.  \end{proof}

\begin{prop}\label{K3}\cite[Proposition 7.3]{calculations} Let $(f: X \to C, S + F_\calA)$ be an irreducible elliptic surface over a rational curve $C \cong \mb{P}^1$ such that $(X, S+ F_\calA)$ is a relative log canonical model and suppose that $\deg \LL = 2$. 
\begin{enumerate}[label = (\alph*)]
\item If $\calA > 0$, then $K_X + S + F_\calA$ is big and the log canonical model is either the relative log canonical model, or the pseudoelliptic obtained by contracting the section of the relative log canonical model. 
\item If $\calA = 0$, then the minimal model program results in a pseudoelliptic surface with a log canonical contraction that contracts this surface to a point. 
\end{enumerate} 
\end{prop}

\begin{prop}\label{prop:middlepseudo} Let $(X, G_1 + G_2)$ be an slc pseudoelliptic surface pair with pseudofibers $G_1$ and $G_2$ marked with coefficient one. Then $K_X + G_1 + G_2$ is big. \end{prop}

\begin{proof} Consider the blowup $\mu: Y \to X$, where $(f:Y' \to \PP^1, S' + G_1'' + G_2'')$ is the corresponding elliptic surface.  Taking the relative log canonical model, we obtain a pair $(f:Y \to \PP^1, S + G_1' + G_2')$, where by construction $K_{Y} + S + G_1' + G_2'$ is relatively ample. Note that $(K_{Y} + S + G_1' + G_2').S = 0$ by Proposition \ref{prop:adjunction} and $K_{Y} + S + G_1' + G_2'$ has positive degree on all other curve classes as it is $f'$-ample. Therefore $K_Y + S + G_1' + G_2'$ is actually nef, and thus semiample by Proposition \ref{prop:abundance}. Therefore the only curve contracted by $|m(K_Y + S + G_1' + G_2')|$ is the section $S$ and the log canonical model of $(X, G_1 + G_2)$ is the corresponding pseudoelliptic surface of $(f: Y \to \PP^1, S + G_1'  + G_2')$. Therefore, $(X, G_1 + G_2)$ is log general type and $K_X + G_1 + G_2$ must be big. \end{proof}

		\section{Weighted stable elliptic surfaces}\label{sec:weightedstable}
		
		In this section we will construct a compactification of the moduli space of log canonical models $(f : X \to C, S + F_\calA)$ of $\calA$-weighted Weierstrass elliptic surface pairs by allowing our surface pairs to degenerate to \emph{semi-log canonical (slc)} pairs (see Definition \ref{def:slc}). As such our surfaces can acquire non-normal singularities and break up into multiple components. 
		
		The first definition we give, inspired by the minimal model program, yields a finite type and separated algebraic stack (see Theorem \ref{thm:stack}) with possibly too many components.  In Definition \ref{def:broken}, we will give a more refined definition of the objects that appear on the boundary of the compactified moduli stack when one runs stable reduction (see Theorem \ref{thm:boundary}).

	\begin{definition}\label{def:es} An \textbf{$\calA$-weighted slc elliptic surface with section} $(f: X \to C, S + F_\calA)$,  (see Figure \ref{figure:example1}) is an slc surface pair $(X, S + F_{\calA})$ and a proper surjective morphism with connected fibers $f: X \to C$ to a projective nodal curve such that:
			
			\begin{enumerate}[label = \alph*)]
				
				\item[(a)] $S$ is a section with generic points contained in the smooth locus of $f$, and $F_\calA$ is an $(\mc{A} \sqcup 1)$-weighted sum of reduced divisors contracted by $f$;
				
				\item[(b)] every component of $Z \subset X$ is either an elliptic surface with fibration $f|_Z$ and section $S|_Z$, or a surface contracted to a point by $f$;
				
				\item[(c)] for each elliptic component $Z$, the restriction $(F_\calA)|_Z$ makes the pair $(f|_Z : Z \to C, S|_Z + (F_\calA)|_Z)$ into a $\calA$-weighted relative log canonical model such that all the marked fibers lie over smooth points of $C$. 

			\end{enumerate}
				
\noindent We say that $(f: X \to C, S + F_\calA)$ is an \textbf{$\calA$-stable elliptic surface} if the $\Q$-Cartier divisor $K_X + S + F_\calA$ is ample, that is, if $(X, S + F_\calA)$ is a stable pair.  

\end{definition} 
	
		\begin{figure}[h!]
			\includegraphics[scale=1]{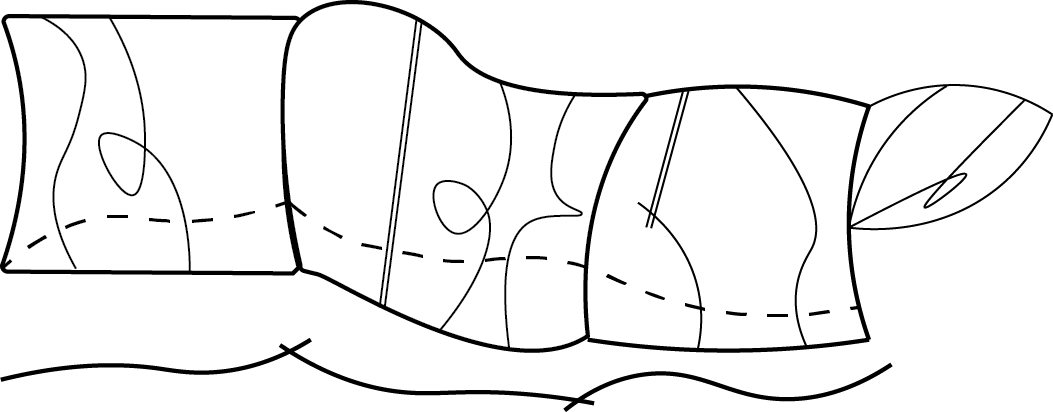}
		\caption{An $\calA$-weighted slc elliptic surface.}\label{figure:example1}
		\end{figure}

We will elaborate on parts $(b)$ and $(c)$ of the above Definition \ref{def:es}. The components in condition $(b)$ contracted to a point by $f$ were the \emph{pseudoelliptic components} (see Definition \ref{def:pseudo}). We will study them further in Section \ref{sec:pseudos}. The condition $(c)$ ensures that the restriction of $F_\calA$ to any elliptic component consists of $a$-weighted twisted, intermediate or Weierstrass fibers $F_a$ marked as in Definition \ref{conv:lcmodel}.

\subsection{Moduli functor for elliptic surfaces}\label{sec:modulifunctor} Following \cite{kp}, we introduce the following notion of a pseudofunctor (following Definition 5.2 of \cite{kp}) of stable elliptic surfaces:

		\begin{definition}\label{def:ellfunctor} Fix $v \in \Q_{>0}$. A pseudofunctor $\calE : \mathfrak{Sch}_k \to \mathfrak{Grp}$ from the category of $k$-schemes to groupoids is a \textbf{moduli pseudofunctor for $\calA$-stable elliptic surfaces of volume $v$} if for any normal variety  $T$, 
			\begin{equation*}
			\calE(T) = \left\{\left(\xymatrix{X \ar[rr]^f \ar[rd]_g & &  C \ar[ld]^h \\ & T & }, S + F_\calA \right) \left|
			\parbox{20em}{
				\begin{enumerate}
				\item $g : (X, S + F_\calA) \to T$ is a flat family of stable pairs of dimension $2$ and volume $v$ as in Definition \ref{def:stablefamily};
				\item $h$ is a flat family of connected nodal curves; 
				\item $f$ is a morphism over $T$; and
				\item for each $t \in T$, the fiber $f_t : (X_t, S_t + (F_\calA)_t) \to C_t$ is an $\calA$-weighted slc elliptic surface with section and marked fibers.
				\end{enumerate}}		
			\right. \right\}.
			\end{equation*}
Let $\calE^\circ$ be the subfunctor consisting of families with $(f_t : X_t \to C_t,S_t)$ a \emph{minimal relative log canonical model} with section over a smooth curve as in Definition \ref{def:standardes}. The \textbf{main component} $\calE^m$ will denote the closure $\overline{\calE^\circ}$ in $\calE$. \end{definition}

\begin{remark} Despite the terminology, it is not true in general that $\calE^m$ is irreducible. Rather, it has components labeled by the configurations of singular fibers on the irreducible elliptic surfaces. \end{remark} 

		\begin{theorem}\label{thm:stack} There exists a moduli pseudofunctor of $\calA$-stable elliptic surfaces of volume $v$ such that the main component $\calE^m$ is a  separated Deligne-Mumford stack of finite type.  \end{theorem} 
		
		\begin{proof} In \cite{kp}, a suitable pseudofunctor $\calM_{v, I,n}$ for stable pairs $(X,D)$ with volume $v$, coefficient set $I$ and index $n$ is defined. Here $n$ is a fixed integer such that $n(K_X + D)$ is required to be Cartier. Furthermore, $\calM_{v,I,n}$ is a finite type Deligne-Mumford stack with projective coarse space (see Proposition 5.11 and Corollary 6.3 in \cite{kp}). Take $I$ to be the additively closed set generated by the weight vector $\calA$. By boundedness for surface pairs (see Theorem 9.2. in \cite{boundedness}), there exists an index $n$ such that $n(K_X + S+ F_\calA)$ is a very ample Cartier divisor for all $\calA$-stable elliptic surfaces of volume $v$. 
		
		Consider the stack of stable pairs $\calM_{v,I,n}$ and denote $\calM := \calM_{v,I,n}$ for convenience. Let $\calX \to \calM$ be the the universal family. Furthermore, let $\mathfrak{M}_g$ be the algebraic stack of prestable curves with universal family $\mathfrak{C}_g \to \mathfrak{M}_g$. Consider the Hom-stack 
		$$
		\mathscr{H}om_{\calM \times \mathfrak{M}_g}(\calX \times \mathfrak{M}_g, \calM \times \mathfrak{C}_g).
		$$
This is a quasi-separated algebraic stack locally of finite presentation with affine stabilizers by Theorem 1.2 in \cite{hr}. Now we consider the pseudofunctor given by 
$$
\calE_{v,\calA,n} : B \mapsto \left\{ \xymatrix{(X, S + F_\calA) \ar[rr]^f \ar[rd] & & C \ar[ld] \\ & B &} \right\}
$$
where $(X, S + F_\calA) \to B$ is a flat family of stable pairs in the sense of \cite{kp}, $C \to B$ is a flat family of pre-stable curves, and $(f_b : X_b \to C_b, S_b + (F_\calA)_b)$ is an $\calA$-stable elliptic surface with volume $v$ for each $b \in B$. 

It is clear that $\calE_{v,\calA,n}$ is a substack of the Hom-stack $\mathscr{H}om_{\calM \times \mathfrak{M}_g}(\calX \times \mathfrak{M}_g, \calM \times \mathfrak{C}_g)$. The substack $\calE_{v,\calA,n}^\circ$ parametrizing irreducible minimal log canonical  models of elliptic surfaces over base curves is an algebraic substack of the Hom-stack, as flatness, irreducibility and smoothness are algebraic conditions. Thus the closure $\calE^m_{v,\calA,n}$ in the Hom-stack is a quasi-separated algebraic stack locally of finite presentation with affine stabilizers, and is a pseudofunctor for $\calA$-stable elliptic surfaces of volume $v$. 

To prove that $\calE^m_{v,\calA,n}$ is separated, let $B$ be a smooth curve and let
$$
\xymatrix{(X^0, S^0 + F_\calA^0) \ar[r]^(.75){f^0} & C^{0} \ar[r] & B^0} = B \setminus p
$$ be a flat family of $\calA$-stable elliptic surfaces over the complement of a point $p \in B$. Suppose 
\begin{align*}
&\xymatrix{(X, S + F_\calA) \ar[r]^(.75){f} & C \ar[r] & B}\\
&\xymatrix{(X', S' + F'_\calA) \ar[r]^(.75){f'} & C' \ar[r] & B}
\end{align*} are two extensions to $B$. 

Then $(X, S + F_\calA) \to B$ and $(X', S' + F'_\calA) \to B$ are two families of stable pairs over $B$ with isomorphic restrictions to $B^0$. Since log canonical models are unique, $(X', S' + F'_\calA) = (X, S + F_\calA)$ over $B$. Furthermore, the compositions $S \to C$ and $S' \to C'$ are isomorphisms so $C \cong C'$ over $B$. Therefore, we have $f, f' : X \to C \to B$ with $f|_{X^0} = f'|_{X^0}$. Since $X \to B$ is flat, $X^0$ is dense in $X$, therefore $f = f'$ since $C$ is separated. Thus an extension to $B$ is unique and so $\calE^m_{v,\calA,n}$ is separated.

Finally, we show that the stack is Deligne-Mumford, by showing that the objects have finitely many automorphisms. An automorphism of $(X, S + F_\calA) \to C$ is an automorphism $\sigma$ of the elliptic surface pair $(X, S + F_\calA)$, as well as an automorphism $\tau$ of $C$ such that the autormophisms commute. Since the autormophism $\sigma$ fixes the fibers $F_\calA$, the compatibility of the automorphisms implies that $\tau$ actually fixes the marked points $D_\calA$ on the base curve $C$ (see Definition \ref{def:DA}). We will show in Corollary \ref{cor:stablecurve} that the base curve is actually a weighted stable pointed curve in the sense of Hassett, and thus has finitely automorphisms. Moreover, there are finitely automorphisms of the stable surface pair (see e.g. \cite[11.12]{iitaka}).  \end{proof}

As it is not clear how to define families of stable pairs over a general base (see Remark \ref{rmk:moduli}), from now on we restrict to only considering families over a normal base.

\begin{definition}\label{def:eva} Define
$$
\calE_{v,\calA} := \left(\calE^m_{v,\calA,n}\right)^{\nu}
$$
to be the normalization of the stack constructed in Theorem \ref{thm:stack} (see Appendix \ref{appendix} for a discussion on normalizations) and $\calU_{v,\calA} \to \calE_{v,\calA}$ the pullback of the universal family.
\end{definition}

 $\calE_{v,\calA}$ is a separated algebraic stack locally of finite type with affine stabilizers. By Proposition \ref{uniquenorm}, the stack $\calE_{v,\calA}$ is independent of $n$ for $n$ large enough, and more generally independent of the choice of pseudofunctor $\calE$ as in Definition \ref{def:ellfunctor}. 

\subsection{Broken elliptic surfaces}\label{sec:broken} 

In this section we refine the definition of an $\calA$-weighted stable elliptic surface pair to more accurately reflect the type of surfaces that will appear as a result of stable reduction. Our strategy for this, inspired by \cite{ln}, is to compute a limit in the twisted stable maps moduli space \cite{av, av2, tsm}, replace this family with its $\calA$-weighted relative log canonical model, and then run the minimal model program to produce a limit of stable pairs. 

To this end, let $(f : X \to C, S + F_\calA)$ be an $\calA$-weighted slc elliptic surface. We want to perform a sequence of extremal and log canonical contractions over $C$ to make $K_X + S + F_\calA$ an $f$-ample divisor. 

Let $\nu : C' \to C$ be the normalization and let $X'$ be the pullback:
$$
\xymatrix{X' \ar[r]^{\varphi} \ar[d]_{f'} & X \ar[d]^f \\ C' \ar[r]^{\nu} & C}.
$$
Then $\varphi^*(K_X + S + F) = K_{X'} + G + S' + F'$ is $f'$-ample if and only if $K_X + S + F$ is $f$-ample. Here $\varphi^*S = S'$ is a section of $f'$ and $F' = \varphi^*F$. The divisor $G$ is the reduced divisor above the points of $C'$ lying over the nodes of $C$. In particular, to compute the relative canonical model over $C$ starting with a log smooth model, it suffices to assume $C$ is smooth and $f : X \to C$ is an irreducible elliptic surface and so the computation of relative log canonical models reduces to that in Section \ref{sec:ellipticbackground}. \\

We now move on to the question of what sorts of pseudoelliptic components appear and how are they attached? There are two \emph{types} of pseudoelliptic components that will appear as irreducible components of a stable limit of elliptic surfaces. 
	
\begin{definition}\label{def:rootedtree} Let $(T,0)$ be a rooted tree with root vertex $0 \in V(T)$. We make $V(T)$ into a poset by declaring that $\alpha \le \beta$ if vertex $\alpha$ lies on the unique minimal length path from vertex $\beta$ to the root $0$. We denote by $T[i]$ the set of vertices of distance $i$ from the root so that $T[0] = \{0\}$. Finally, if $\alpha \in T[i]$, we denote by $\alpha[1]$ the set of vertices $\beta \in T[i + 1]$ with $\alpha \le \beta$. \end{definition}

\begin{definition}\label{def:pseudoI} Let $(T, 0)$ be a rooted tree. A \textbf{pseudoelliptic tree $(Y, F_\calA)$} with dual graph $(T,0)$ is an slc pair consisting of the union of pseudoelliptic components $Y_\alpha$ with dual graph $T$ constructed inductively: every component $Y_\beta$ for $\beta \in \alpha[1]$ is attached to $Y_\alpha$ by gluing a twisted pseudofiber $G_\beta$ of $Y_\beta$ to the arithmetic genus one component $E_\alpha$ of an intermediate pseudofiber with reduced component $A_\alpha$ of $Y_\alpha$. The $\calA$-weighted marked fibers $F_\calA$ satisfy 
\begin{equation}\label{eq:A}
\mathrm{Coeff}(A_\alpha, F_\calA) = \sum_{\beta \in \alpha[1]} \sum_{D \in \Supp(F_\calA|_{Y_\beta})} \mathrm{Coeff}(D, F_\calA).
\end{equation}
A component $(Y_\alpha, F_\calA|_{Y_\alpha})$ is a \textbf{Type I} pseudoelliptic (See Figure \ref{figure:pseudo1}).
\end{definition}

	\begin{figure}[h!]
			\includegraphics[scale=.75]{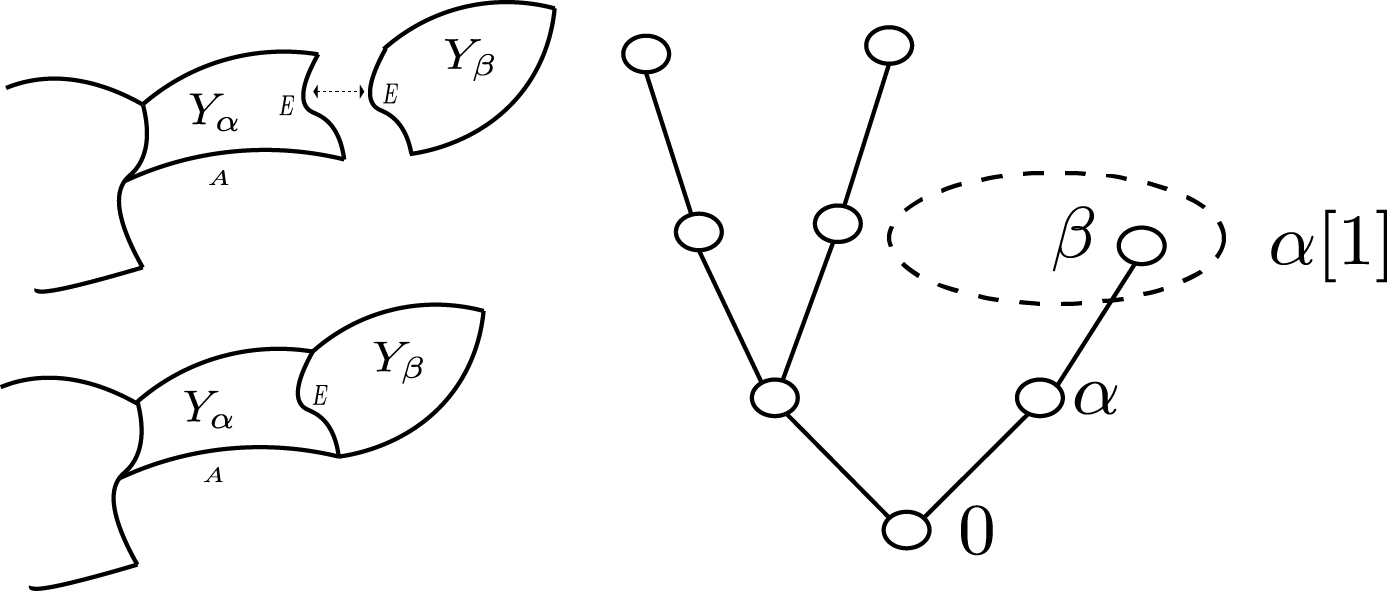}
		\caption{A pseudoelliptic tree of is constructed inductively by attaching a \emph{Type I} pseudoelliptic surface $Y_\beta$ to $Y_\alpha$ for each $\beta \in \alpha[1]$ as pictured. The component $A$ on $Y_\alpha$ is marked by the sum of the weights of the markings on $Y_\beta$. }\label{figure:pseudo1}
		\end{figure}

	\begin{definition}\label{def:pseudoII} A pseudoelliptic surface of \textbf{Type II} (see Figure \ref{figure:pseudo2}) is formed by the log canonical contraction of a section of an elliptic component attached along twisted or stable fibers. \end{definition}

		\begin{figure}[h!]
			\includegraphics[scale=1]{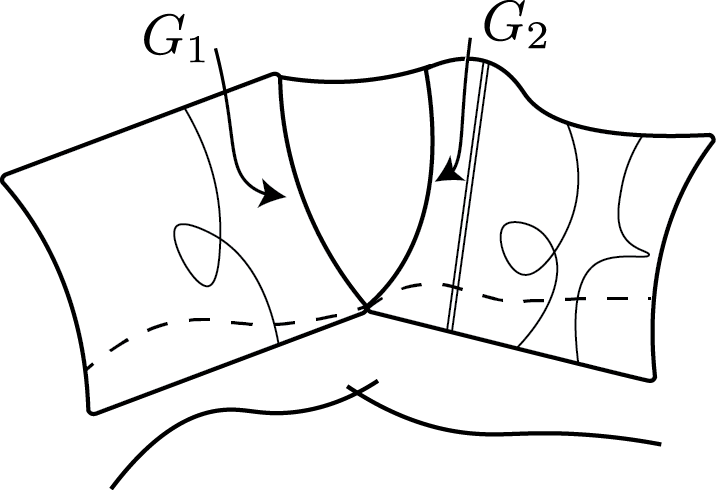}
		\caption{A pseudoelliptic surface of \emph{Type II} attached along twisted fibers $G_1$ and $G_2$.}\label{figure:pseudo2}
		\end{figure}

		One important fact is that the section $S$ is often contracted even for $\calA$-weighted elliptic surfaces with small but nonzero weights. In fact, we will see (Section \ref{sec:pseudos}) that contracting the section of a component to form a pseudoelliptic corresponds to stabilizing the base curve as an $\calA$-stable curve in the sense of Hassett (see Section \ref{sec:hassett}).

We can now define the particular $\calA$-weighted stable elliptic surfaces that will appear on the boundary of the main components of the moduli space (see Figure \ref{fig:brokensurfacedef}).

\begin{definition}\label{def:broken} An $\calA$-\textbf{broken elliptic surface} is an $\calA$-weighted slc elliptic surface pair\\ $(f : X \to C, S + F_\calA)$ such that (see Figure \ref{fig:brokensurfacedef})

\begin{enumerate}[label = (\alph*)]
\item each component of $X$ contracted by $f$ is a type I or type II pseudoelliptic surface with marked pseudofibers;
\item the elliptic components and type II pseudoelliptics are attached along twisted fibers; 
\item the type I pseudoelliptics appear in pseudoelliptic trees attached by gluing a twisted pseudofiber $G_0$ on the root to an arithmetic genus one component $E$ of an intermediate (pseudo)fiber of an elliptic (type II pseudoelliptic) component;
\item all marked intermediate (pseudo)fibers are minimal. 
\end{enumerate}

\noindent We say $(f : X \to C, S + F_\calA)$ is an \textbf{$\calA$-broken stable elliptic surface} if $(X, S + F_\calA)$ is a stable pair. 

\end{definition} 

	\begin{figure}[h!]
			\includegraphics[scale=1]{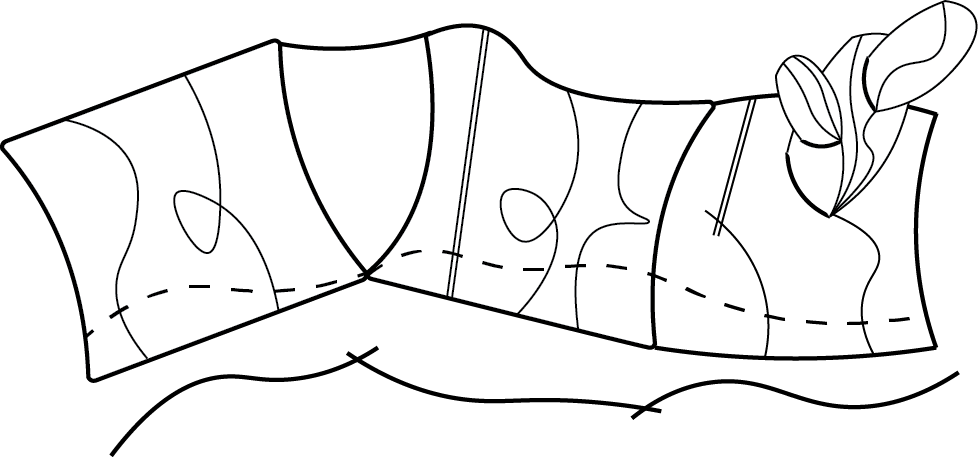}
		\caption{An $\calA$-weighted broken elliptic surface.}\label{fig:brokensurfacedef}
		\end{figure}
		
\begin{remark} Note this definition allows for non-minimal Weierstrass cusps and also non-minimal intermediate fibers contained in the double locus. \end{remark}

\begin{remark} \label{rem:section} For each pseudoelliptic component $X_0 \subset X$ with associated elliptic surface $f_0 : Y_0 \to C_0$ and morphism $\mu_0 : Y_0 \to X_0$ contracting the section, there is a unique slc elliptic surface $f' : X' \to C'$ with $Y_0 \subset X'$ and $f'|_{Y_0} = f_0$. There is a morphism $\mu : X' \to X$ contracting the section of $Y_0$. Thus one can think of a broken elliptic surface as one obtained by contracting the sections and corresponding components in the base curve of some irreducible components of an $\calA$-weighted slc elliptic pair $(f : X' \to C', S' + F'_\calA)$ where \emph{every component} is a relative log canonical model of an elliptic surface and in particular only has twisted, intermediate or Weierstrass fibers marked appropriately.   \end{remark}

\subsection{Formation of pseudoelliptic components}\label{sec:pseudos}

In this subsection, we record various statements about the formation of pseudoelliptic components. 

The following describes how the log canonical divisor class intersects the section (see also Proposition 6.5 in \cite{calculations} and Proposition 4.3.2 in \cite{ln}). This determines when the section of a component contracts to form a pseudoelliptic surface. \\

Given an $\calA$-broken elliptic surface $(f : X \to C, S + F_\calA)$, we obtain an $\calA$-weighted pointed curve $(C, D_\calA)$ as follows using the fact that $S$ is a log canonical center (see \cite[Def 2.34]{singmmp}).

\begin{definition}\label{def:DA} Let $(f : X \to C, S + F_\calA)$ be a $\calA$-broken elliptic surface. $D_\calA$ is the unique divisor on $C$ such that
$$
\sigma^*(K_X + S + F_\calA) = K_C + D_\calA
$$
and $(C,D_\calA)$ is an slc pair where $\sigma : C \to X$ is the map identifying $C$ with the section $S$. 

\end{definition} 

\begin{remark}\label{rem:DA} One can compute that $D_\calA = \sum a_ip_i$ where $p_i = f(F_{a_i})$ is the image of the $i^{th}$ marked fiber and $a_i$ its coefficient. \end{remark}

We form the dual graph of $C$ by assigning a vertex to each irreducible component $C_\alpha \subset C$ and an edge for each node. Let $v_\alpha$ be the valence of $C_\alpha$ in the dual graph and $g(C_\alpha)$ the geometric genus of $C_\alpha$.

\begin{prop}\cite[Proposition 5.3]{tsm}\label{prop:adjunction} Let $(f : X \to C, S + F_\calA)$ be an $\calA$-broken elliptic surface with section $S$. Let $(C, D_\calA)$ be the $\calA$-weighted pointed curve and $C \subset C_\alpha$ an irreducible component. Then for the component $S_\alpha$ of the section lying above $C_\alpha$, we have
\begin{align*}
(K_X + S + F).S_\alpha &= 2g(C_\alpha) - 2 + v_\alpha + \deg\left(D_\calA|_{C_\alpha}\right) \\
&= \deg\left(\omega_C(D_\calA)|_{C_\alpha}\right).
\end{align*}
\end{prop} 

\begin{proof} The case where $\calA = 1$ is precisely Proposition 5.3 of \cite{tsm} (see also Proposition 6.5 in \cite{calculations} and Proposition 4.3.2 in \cite{ln}). This more general case follows from the adjunction formula, as the section passes through the smooth locus of the surface in a neighborhood of any fiber that is not marked with coefficient $a_i = 1$. \end{proof} 

\begin{cor}\label{cor:stablecurve}\cite[Corollary 6.7 \& 6.8]{calculations} Let $(f : X \to C, S + F_\calA)$ be an $\calA$-broken elliptic surface such that each component is elliptically fibered. Then $(K_X + S + F_\calA).S_\alpha > 0$ for every component $S_\alpha$ of $S$ if and only if $(C, D_\calA)$ is an $\calA$-pointed stable curve in the sense of Hassett. In this case, the relative log canonical model over $C$ is stable. \end{cor}

\begin{cor}\label{cor:p1}\cite[Corollary 6.9]{calculations} The log minimal model program contracts the section of an elliptic component $X_\alpha \to C_\alpha$ of $(f : X \to C, S + F_\calA)$  to produce a pseudoelliptic if and only if either:
\begin{enumerate}[label = (\alph*)]
\item $C \cong \mb{P}^1$ and $\sum a_i \le 2$, or
\item $C$ is a genus one curve and $a_i = 0$ for all $i$. 
\end{enumerate} 
\end{cor}

\section{Invariance of log plurigenera for broken elliptic surfaces}\label{sec:vanishing} 

In \cite{tsm}, we investigated the stable pairs compactification of the space of \emph{twisted elliptic surfaces} using the theory of twisted stable maps. A twisted elliptic surface is an irreducible $\calA$-stable elliptic surface $(f : X \to C, S + F_\calA)$ for the constant weight vector $\calA = (1,\ldots, 1)$ satisfying the property that the support of every non-stable fiber is contained in $\Supp(F_\calA)$. In particular, the compactification of the space of twisted elliptic surfaces is a component of the space $\calE_{v,\calA}$ we denote $\calF_v(1,1)$. 

The main result of \cite{tsm} regarding the space $\calF_v(1,1)$ is the characterization of the boundary components as consisting of broken elliptic surfaces (see Theorem 1.4 \cite{tsm}). Our goal is to generalize this result to $\calA$-stable elliptic surfaces for arbitrary weights and use it to construct various morphisms between the moduli spaces for different weights analagous to the reduction morphisms of Hassett spaces (see Theorem \ref{thm:hassett}). 

To show the existence of such morphisms on the level of moduli spaces and universal families, we demonstrate that the pushforwards of the pluri-log canonical sheaves of a family are compatible with base change so that the construction of log canonical models is functorial in families. The main technical step is the following vanishing theorem for the pluri-log canonical divisor which we prove in this section: 

\begin{theorem}\label{thm:vanishing} Let $(f : X \to C, S + F_\calB)$ be a $\calB$-broken stable elliptic surface with section $S$ and marked fibers $F_\calB$. Let $0 \le \calA \le \calB$ such that $K_X + S + F_\calA$ is nef and $\Q$-Cartier. 

\begin{enumerate}[label = (\alph*)] 
\item If $\calA$ is not identically zero.  Then
$$
H^i\Big(X, \calO_X\big(m(K_X + S + F_\calA)\big)\Big) = 0
$$
for $i > 0$ and $m \geq 2$ divisible enough. 
\item If $\calA$ is identically zero, suppose further that \begin{enumerate*}[label = \roman*)] \item $p_g(C) \neq 1$, and \item if $p_g(C) = 0$ then $\deg \LL \neq 2$. \end{enumerate*}
Then 
$$
H^i\Big(X, \calO_X\big(m(K_X + S + F_\calA)\big)\Big) = 0
$$
for $i > 0$ and $m \geq 2$ divisible enough. 
\end{enumerate}
\end{theorem} 

\begin{remark} $ $ 
\begin{enumerate} 

\item \label{rem:A(t)} Let $\calA_t = t\calB + (1-t)\calA$. Since the nef cone is the closure of the ample cone, the divisor $K_X + S + F_{\calA_t}$ is also ample for $t > 0$. That is, $K_X + S + F_\calA$ is the first time we drop from ample to nef along the segment connecting $\calB$ to $\calA$.

\item We consider the case $p_g(C) = 1$ and $\calA = 0$ and the case $p_g(C) = 0, \calA = 0$ and $\deg \LL = 2$ in Theorem \ref{thm:basechange}.\end{enumerate} \end{remark} 

\begin{proof} 

We will prove both cases at once, pointing out where in the argument the hypothesis in case $(b)$ are necessary if $\calA = 0$. For convenience we sometimes denote the $\Q$-line bundle $L^{[m]} := \calO_X(m(K_X + S + F_\calA))$. The proof will proceed through several steps. \\

\begin{steps} 

\proofstep First we carefully break $X$ up into several components. \\

Let $Y \subset X$ be a union of irreducible components and let $X'$ be a union of the complementary irreducible components. Then there is an exact sequence
$$
0 \to L^{[m]}|_{X'}(-M) \to L^{[m]} \to L^{[m]}|_Y \to 0
$$
where $M = \sum_{j = 1}^s M_j$ is the sum of fiber components along which $X'$ and $Y$ are attached to obtain $X$ (see the proof of \cite[Corollary 10.34]{kk}). Since $\calO_{X'}(K_X|_{X'}) = \calO_{X'}(K_{X'} + M)$ and $\calO_Y(K_X|_Y) = \calO_Y(K_Y + M)$, we see that
\begin{align*}
L^{[m]}|_{Y} &= \calO_Y\Big(m\big(K_Y + S|_Y + F_{\calA}|_Y + M\big)\Big)\\
L^{[m]}|_{X'}(-M) &= \calO_{X'}\Big(m\big(K_{X'} + S|_{X'} + F_{\calA}|_{X'} + \frac{m-1}{m}M\big)\Big).
\end{align*}

By the long exact sequence of cohomology, it suffices to prove vanishing for the divisor $L^{[m]}|_{X'}(-M)$ on $X'$ and $L^{[m]}|_Y$ on $Y$. To do this, we need to guarantee some positivity for $L^{[m]}|_{X'}(-M)$, namely that it is nef. This is not immediate due to the twisting by $-M$, and therefore we need to pick $X'$ and $Y$ judiciously to ensure that twisting by $-M$ still yields a nef divisor. Note on the other hand that $L^{[m]}|_Y$ is automatically nef.  \\

Let $Y$ be a pseudoelliptic tree (see Definition \ref{def:pseudoI}) indexed by the rooted tree $(T,0)$ with root component $Y_0$. Suppose that $Y$ is attached to $X'$ by gluing a twisted pseudofiber of $Y_0$ to the arithmetic genus $1$ component of an intermediate fiber on $X'$. In this case $M$ is an irreducible curve. Let $A$ denote the rational component of the intermediate fiber of $X'$. Suppose finally that $\mathrm{Coeff}(A, F_\calA) < \mathrm{Coeff}(A, F_\calB)$. 

\begin{lemma}\label{lemma:separatingtypeI} In the situation above, $L^{[m]}|_{X'}(-M)$ and $L^{[m]}|_Y$ are nef and $\Q$-Cartier. \end{lemma} 

\begin{proof} $L^{[m]}|_Y$ is nef and $\Q$-Cartier as it is the restriction of a nef and $\Q$-Cartier divisor. On the other hand, we need to check that 
$$
L^{[m]}|_{X'}(-M) = \calO_{X'}\left(m(K_{X'} + S|_{X'} + F_{\calA}|_{X'} + \frac{m-1}{m}M)\right)
$$
is nef and $\Q$-Cartier on $X'$. For $\Q$-Cartier, it suffices to note that $X'$ has quotient singularities in a neighborhood of $M$ (see Section 6.2 of \cite{tsm}). To see that it is nef, note that we only need to check 
\begin{align*}
\left(K_{X'} + S|_{X'} + F_{\calA}|_{X'} + \frac{m-1}{m}M\right).M &\geq 0 \\
\left(K_{X'} + S|_{X'} + F_{\calA}|_{X'} + \frac{m-1}{m}M\right).A &\geq 0
\end{align*}
since $K_{X'} + S|_{X'} + F_{\calA}|_{X'} + M$ is nef and reducing the coefficient of $M$ does not affect the degree on the other components of the marked divisor. Furthermore, the intersections we are computing are all on the single component of $X'$ containing $A$, so we may suppose $X'$ is irreducible.

The first inequality is clear -- recall that $M^2 < 0$, so reducing its coefficient \emph{increases} the intersection with $M$. For the second inequality, we take a log resolution $\mu : X_0 \to X'$ if necessary, so that we can assume that $A$ lies on an elliptic component $f_0 : X_0 \to C_0$ with section $S_0$. Using the fact that the $\calB$-weighted divisor $K_X + S + F_\calB$ is ample, we see that $K_{X_0} + S_0 + F_{\calB}|_{X_0} + M$ is $f_0$-ample. Furthermore $A$ is disjoint from the other marked fibers and $A^2 < 0$, so that decreasing the coefficient of $A$ \emph{increases} the degree on $A$. That is, 
$$
(K_{X_0} + S_0 + F_{\calA}|_{X_0} + M).A > 0
$$
so for large enough $m$, 
$$
\left(K_{X_0} + S_0 + F_{\calA}|_{X_0} + \frac{m - 1}{m}M\right).A > 0.
$$
In particular, $K_{X_0} + S_0 + F_{\calA}|_{X_0} + \frac{m - 1}{m}M$ is $f_0$-nef. Thus, after possibly contracting the section if necessary, we obtain a log minimal model $\Big(X', \mu_*\big(S_0 + F_{\calA}|_{X_0} + \frac{m-1}{m}M\big)\Big)$. In particular, $K_{X'} + (S + F_\calA)|_{X'} + \frac{m-1}{m} M$ is nef. \end{proof}

Now we check that the condition $\mathrm{Coeff}(A, F_\calA)<\mathrm{Coeff}(A, F_\calB)$ is satisfied whenever $Y$ is a pseudoelliptic tree which contains at least one marked divisor whose coefficient is lowered. Indeed, if $Y_\alpha$ is a component and $A_\alpha$ is the reduced component of an intermediate fiber where another pseudoelliptic $Y_\beta$ with $\beta \geq \alpha$ is attached, then 
$$
\mathrm{Coeff}(A_\alpha, F_{\calA}) = \sum_{D \subset \mathrm{Supp}(F_{\calA}|_{Y_\beta})} \mathrm{Coeff}(D, F_\calA)
$$
is a sum of the coefficients of marked fibers on $Y_\beta$. In particular, if $A$ as above is the reduced component of an intermediate fiber on $X'$ where the root component $Y_0$ of $Y$ is attached, then $\mathrm{Coeff}(A, F_\calA) < \mathrm{Coeff}(A, F_{\calB})$ since there is some $D$ on some $Y_\beta$ with $\mathrm{Coeff}(D, F_\calA) < \mathrm{Coeff}(D, F_\calB)$.

Now by induction on the number of pseudoelliptic trees where we have reduced coefficients, we use the long exact sequence on cohomology associated to 
$$
0 \to L^{[m]}|_{X'}(-M) \to L^{[m]} \to L^{[m]}|_Y \to 0
$$
and reduce to proving vanishing for the following two cases: 

\begin{enumerate}
\item \label{case:1} $(X, S + F_{\calA})$ is an slc $\calA$-broken elliptic surface such that $F_{\calA}|_Y = F_{\calB}|_Y$ for any pseudoelliptic tree, or
\item \label{case:2} $(X, F_{\calA})$ is an slc pseudoelliptic tree.
\end{enumerate}
We will denote this pair $(X,\Delta)$ and take care to note which case we are in if necessary.  \\

\proofstep We consider Case $\ref{case:1}$. Here we show that we may assume that $K_X + S + F_\calA$ is big on every component of $X$. Indeed $K_X + S + F_\calA$ is \emph{ample} on every pseudoelliptic of Type I by assumption. By Proposition \ref{prop:middlepseudo}, it is big on every pseudoelliptic of Type II (see Definition \ref{def:pseudoII}) and every elliptic component with $\deg \LL > 0$. 

We are left to consider a component $X_1 \cong E \times C_1$ isomorphic to a product with section $S_1$.  By Proposition \ref{prop:adjunction}, if $(K_X + S + F_\calA)|_{X_1}$ is nef but not big, then $C_1$ is rational and $(K_X + S + F_\calA).S_1 = 0$. In this case, the log canonical morphism factors through a morphism $\mu : X \to Z$ which contracts the component $X_1$ onto $E$ and is an isomorphism away from $X_1$. 

Now $(Z, \mu_*(S + F_\calA))$ is an $\calA$-broken elliptic surface and 
$$
\mu_*\calO_X\big(m(K_X + S + F_\calA)\big) = \calO_Z\Big(m\big(K_Z +\mu_*(S + F_\calA)\big)\Big).
$$
Therefore we want to show $R^i\mu_*L^{[m]} = 0$ for $i > 0$ so that 
$$
H^j\Big(X, \calO_X\big(m(K_X + S + F_\calA)\big)\Big) = H^j\Bigg(Z, \calO_{Z}\Big(m\big(K_{Z} + \mu_*(S + F_\calA)\big)\Big)\Bigg). 
$$
By the Theorem on Formal Functions, it suffices to show that
$$
H^i(X_n, L^{[m]}|_{X_n}) = 0
$$
for all $i > 0$ and $n$, where $X_n$ is the $n^{th}$ formal neighborhood of $X_1$ in $X$. The fibration $X_1 \to C_1$ extends to a fibration $X_n \to C_n$ with all fibers isomorphic to $E$, where $C_n$ is isomorphic to the $n^{th}$ formal neighborhood of the component $C_1$ in $C$. That is, $C_n$ is a rational curve with two embedded points of length $n$, and is locally isomorphic to $k[x,y]/(xy,y^n)$ around these points. Furthermore, $L|_{X_n} \cong \calO_{X_n}(S_n)$, where $S_n$ is a formal neighborhood of the section. 

\begin{lemma} Let $f_n : X_n \to C_n$ be an elliptic fibration with all fibers isomorphic to $E$ over a rational curve $C_n$ with finitely many embedded points locally isomorphic to $k[x,y]/(xy, y^n)$. Let $S_n$ be a section. Then $H^i\big(X_n, \calO_{X_n}(mS_n)\big) = 0$ for any $m,n \geq 1$ and $i > 0$. \end{lemma} 

\begin{proof} A direct computation on  $E$ shows that
$$
H^i(E, mP) = 0
$$
for $i > 0$ and $m \geq 1$ where $P = (S_n)|_E$ is a point. Therefore $R^i(f_n)_*\big(\calO_{X_n}(mS_n)\big) = 0$ for $i > 0$. Similarly, 
$$
h^0(E, mP) = m
$$
so $R_{m,n} := (f_n)_*\big(\calO_{X_n}(mS_n)\big)$ is a rank $m$ vector bundle. 

For $m,n = 1$, the pushforward $({f_1})_*\big(\calO_X(S))$ is a line bundle on $C_1 \cong \mathbb{P}^1$ with a section coming from pushing forward the section $\calO_{X_1} \to \calO_{X_1}(S_1)$. Therefore $H^i(C_1, R_{1,1}) = 0$ for $i > 0$. Pushing forward the exact sequence 
$$
0 \to \calO_{X_1}\big((m - 1)S_1\big) \to \calO_{X_1}(mS_1) \to \calO_{S_1}(m{S_1}|_{S_1}) \to 0
$$
and noticing that ${S_1}|_{S_1} = 0$, we get
$$
0 \to R_{m-1,1} \to R_{m,1} \to \calO_{C_1} \to 0. 
$$
Since $H^i(\mb{P}^1,\calO_{\mb{P}^1}) = 0$ for $i > 0$, then $H^i(C_1, R_{m,1}) = 0$ for $i > 0$ by induction on $m$.

Now consider the ideal sequence
$$
0 \to I_n \to \calO_{C_n} \to \calO_{C_{n-1}} \to 0
$$
where $I_n$ is torsion supported on finitely many points. Applying $(-)\otimes_{C_n} R_{m,n}$ and using base change for the Cartesian square
$$
\xymatrix{E \times C_{n - 1} \ar[r]^j \ar[d]_{f_{n-1}} & E \times C_n \ar[d]^{f_n} \\ C_{n - 1} \ar[r]_i & C_n}
$$
gives an exact sequence 
$$
0 \to K_{m,n} \to R_{m,n} \to R_{m,n-1} \to 0
$$
where $K_{m,n}$ is torsion supported on finitely many points. Now by induction on $n$ and the previous vanishing for $R_{m,1}$, we obtain $H^i(C_n, R_{m,n}) = 0$ for all $i > 0$. The required vanishing then follows from the Leray spectral sequence.
\end{proof}

This shows it suffices to prove vanishing in Case \ref{case:1} for the $\calA$-broken elliptic surface pair $(Z, \mu_*(S + F_\calA))$ after contracting the component $X_1$. Applying this inductively, we can assume that in Case \ref{case:1}, the divisor $K_X + S + F_\calA$ is big on every component. \\

\proofstep Next we reduce to the case when $K_X + S + F_\calA$ has positive degree on every component of the section. Let $(X_0 \to C_0, S_0)$ be an elliptically fibered component such that $(K_X + S+ F_\calA).S_0 = 0$. 

Let $\mu : X \to Z$ be the morphism contracting $S_0$. Then $\big(Z, \mu_*(S + F_\calA)\big)$ is an $\calA$-broken elliptic surface pair and 
$$
\mu_*\calO_X\big(m(K_X + S + F_\calA)\big) = \calO_Z\Big(m\big(K_Z + \mu_*(S + F_\calA)\big)\Big)
$$
by Proposition \ref{logcanonical}. We want to show that
$$
R^i\mu_*\calO_X\big(m(K_X + S + F_\calA)\big) = 0
$$
for $i > 0$. This follows by Proposition \ref{prop:gr}, since the exceptional locus of $\mu$ is a rational curve $S_0 \cong \mb{P}^1$ with $S_0^2 < 0$ and $(K_X + S + F_\calA).S_0 = 0$.  Here we have used the hypothesis that if $\calA = 0$, then the genus of the base curve is not $1$ so that $S_0$ is necessarily a rational curve. \\

\proofstep We complete the proof in Case \ref{case:1}, under the assumption that $K_X + S + F_\calA$ is big on every irreducible component of $X$, and has positive degree on every component of the section. 

\begin{prop}\label{vanishing>0} Let $(f:X \to C, S + F_\calB)$ be a $\calB$-broken stable elliptic surface.  Let $L^{[m]}$ denote the divisor $m(K_X + S + F_\calA)$ for $m \geq 2$, where $0 \le \calA \le \calB$.  Suppose that $K_X + S + F_\calA$ is nef, $\Q$-Cartier and big on every irreducible component of $X$ and that $(K_X + S + F_\calA).S_0 > 0$ for every component $S_0$ of $S$. Suppose that $F_\calA|_Y = F_\calB|_Y$ for every pseudoelliptic tree $Y \subset X$. Finally suppose either \begin{enumerate*}[label = (\alph*)] \item $p_g(C) \neq 1$, or \item $\calA$ is not identically zero. \end{enumerate*} Then $H^i(X, L^{[m]}) = 0$ for all $i > 0$. \end{prop}

\begin{proof} We will apply Fujino's Theorem \ref{fujinovanishing} to $L^{[m]}$. We have that
$$
L^{[m]}\big(-(K_X + S + F_\calA)\big) = \calO_X\big((m - 1)(K_X + S + F_\calA)\big)
$$
is big and nef on every irreducible component of $X$ by assumption. Therefore, to apply the theorem, it suffices to prove that $K_X + S + F_\calA$ is big on every slc center of $(X, S + F_\calA)$. This is clear for zero dimensional slc centers.

The one dimensional slc centers of $(X, S + F_\calA)$ are \begin{enumerate*}[label = (\alph*)] \item the components of the section $S$, \item the twisted fibers $F_j$, \item $E$ components of marked intermediate fibers, \item and the components of the double locus $D$, \end{enumerate*} Now $K_X + S + F_\calA$ is big on every component of the section by assumption, and
\begin{align*}
(K_X + S + F_\calA).F_j &= 1/d
\end{align*}
where $F_i$ supports a possibly nonreduced fiber of multiplicity $d$. Here we have used the fact that a twisted fiber is irreducible so that $F_\calA.F_i = 0$. 

Next we need to consider the $E$ components of marked intermediate fibers. Since $K_X + S + F_\calA$ is nef, we have 
$$
(K_X + S + F_\calA).E \geq 0.
$$
If this is positive then the restriction $(K_X + S+ F_\calA)|_E$ is big. If this intersection is $0$ then the log canonical linear series factors through the contraction of $E$ to a minimal Weierstrass cusp. Let $\mu : X \to Z$ be this contraction. Then $(Z, \mu_*(S + F_\calA))$ is log canonical and
$$
\mu_*\calO_X(m(K_X + S + F_\calA)) = \calO_Z(m(K_Z + \mu_*(S + F_\calA)))
$$
by Proposition \ref{logcanonical}. We want to show that
$$
R^i\mu_*\calO_X(m(K_X + S+ F_\calA)) = 0
$$
for $i > 0$. This follows by Proposition \ref{prop:gr} since $E$ is a rational curve as the marked intermediates are all minimal, $E^2 = 0$ and $(K_X + S+ F_\calA).E = 0$. Therefore it suffices to compute cohomology vanishing for the pair $(Z, \mu_*(S + F_\calA))$ and by induction we can assume that $K_X + S + F_\calA$ is big on all $E$ components of marked intermediate fibers. 

This leaves case $(d)$, the double locus $D$, which consists of three types of irreducible components:

\begin{enumerate}[label = (\roman*)] 

\item For a stable or twisted (pseudo)fiber $F$ along which an elliptic or type II pseudoelliptic is glued to the rest of $X$, we have
$$
(K_X + S + F_\calA).F = 1/d > 0;
$$ 

\item For every isotrivial component $Z$ with $j = \infty$, there is the self intersection locus $B$. If $Z$ is a pseudoelliptic component, then the morphism $Z' \to Z$ contracting the section of the associated elliptic component is an isomorphism in a neighborhood of $B$ so we may suppose that $Z$ is elliptic. In this case $B$ is a section of $Z$ disjoint from $S$ and
$$
(K_X + S + F_\calA).B > 0.
$$

\item For every pseudoelliptic tree $Y$, there is the component $M$ along which the root component $Y_0$ is attached to the rest of $X$. By the assumption $$(K_X + S + F_\calA)|_Y = (K_X + S + F_\calB)|_Y$$ is ample on $Y$. In particular, $(K_X + S + F_\calA)|_Y.M > 0$.

\end{enumerate} 

Therefore $K_X + S + F_\calA$ is big and nef on each slc stratum of $(X, F_\calA)$. Applying Theorem \ref{fujinovanishing} we have the required vanishing 
$$
\begin{array}{lr}
H^i\Big(X, \calO_X\big(m(K_X + F_\calA)\big)\Big) = 0, & i > 0.
\end{array}$$ \end{proof}

\proofstep Now we consider Case \ref{case:2} of a pseudoelliptic tree $Y$ indexed by a rooted tree $(T,0)$. If $(Y, F_\calA)$ is already a stable pair, then we are done. Otherwise, there is some $Y_\alpha$ where the coefficients have been reduced. This implies the coefficients have been reduced on $Y_\beta$ for any $\beta \le \alpha$ as well. 

Suppose $Y_\alpha$ is a leaf of the tree and that $Y'$ is the union of $Y_\beta$ for $\beta \neq \alpha$, i.e. the pseudoelliptic tree with dual graph $(T \setminus \alpha, 0)$. Suppose $Y_\alpha$ is attached to $Y'$ along $M$ a component of an intermediate fiber on $Y'$ with genus 0 component $A$. Since the coefficients of $Y_\alpha$ have been reduced, then $L^{[m]}|_{Y'}(-M)$ and $L^{[m]}|_{Y_\alpha}$ are nef and $\Q$-Cartier by Lemma \ref{lemma:separatingtypeI}. By the attaching sequence, it suffices to show that 
$$
H^i(Y_\alpha, L^{[m]}|_{Y_\alpha}) = H^i(Y', L^{[m]}|_{Y'}(-M)) = 0. 
$$

By induction on the number of leaves of $T$, it suffices to prove that
$$
H^i(Y, L^{[m]}|_Y(-M)) = 0
$$
where $(Y, S + F_\calA)$ is a pseudoelliptic tree, $M$ is a sum of the supports of finitely many arithmetic genus $1$ components of intermediate pseudofibers of $Y$, and either
\begin{enumerate}
\item $Y$ is irreducible, or
\item for each leaf $\alpha \in T$, $F_\calA|_{Y_\alpha} = F_\calB|_{Y_\alpha}$.
\end{enumerate}

That is, we have separated of all of the leaves on which coefficients have been decreased. Therefore, we have reduced to proving vanishing on the leaves themselves, as well as on a pseudoelliptic tree for which the coefficients of all emanating leaves have \emph{not} been decreased. \\

\proofstep Let $(Y, F_\calA)$ be a pseudoelliptic tree with dual graph $(T,0)$ and suppose that $F_\calA|_{Y_\beta} = F_\calB|_{Y_\beta}$ for each leaf $\beta$, that is, we are in case $(2)$ of Step 5 above. If $F_\beta|_Y = F_\alpha|_Y$ then $(K_X + S + F_\alpha)|_Y$ is ample so were done. Thus suppose that there exists a component $Y_\alpha$ with $F_\calA|_{Y_\alpha} < F_\calB|_{Y_\alpha}$. We may take $\alpha$ to be maximal so that $F_\calA|_{Y_\beta} = F_\calB|_{Y_\beta}$ for all $\beta > \alpha$. 

Let $\beta \in \alpha[1]$ (Definition \ref{def:rootedtree}) and $T_{\geq \beta} = \{\gamma \in V(T) : \gamma \geq \beta\}$ the subtree of $T$ with root $\beta$ (see Figure \ref{figure:pseudotree}). Then $Y_{\geq \beta} = \bigcup_{\gamma \in T_{\geq \beta}} Y_\gamma$ is a pseudoelliptic subtree of $Y$ with root component $Y_\beta$. Denoting by $Y'$ the union of components of $Y$ not in $Y_{\geq \beta}$, then $Y_{\geq \beta}$ is attached to $Y'$ by gluing a twisted pseudofiber $M$ on $Y_{\geq \beta}$ to the arithmetic genus $1$ component of an intermediate pseudofiber $M \cup A$ on $Y_\alpha \subset Y'$. 

\begin{figure}[h!]
\includegraphics[scale=.2]{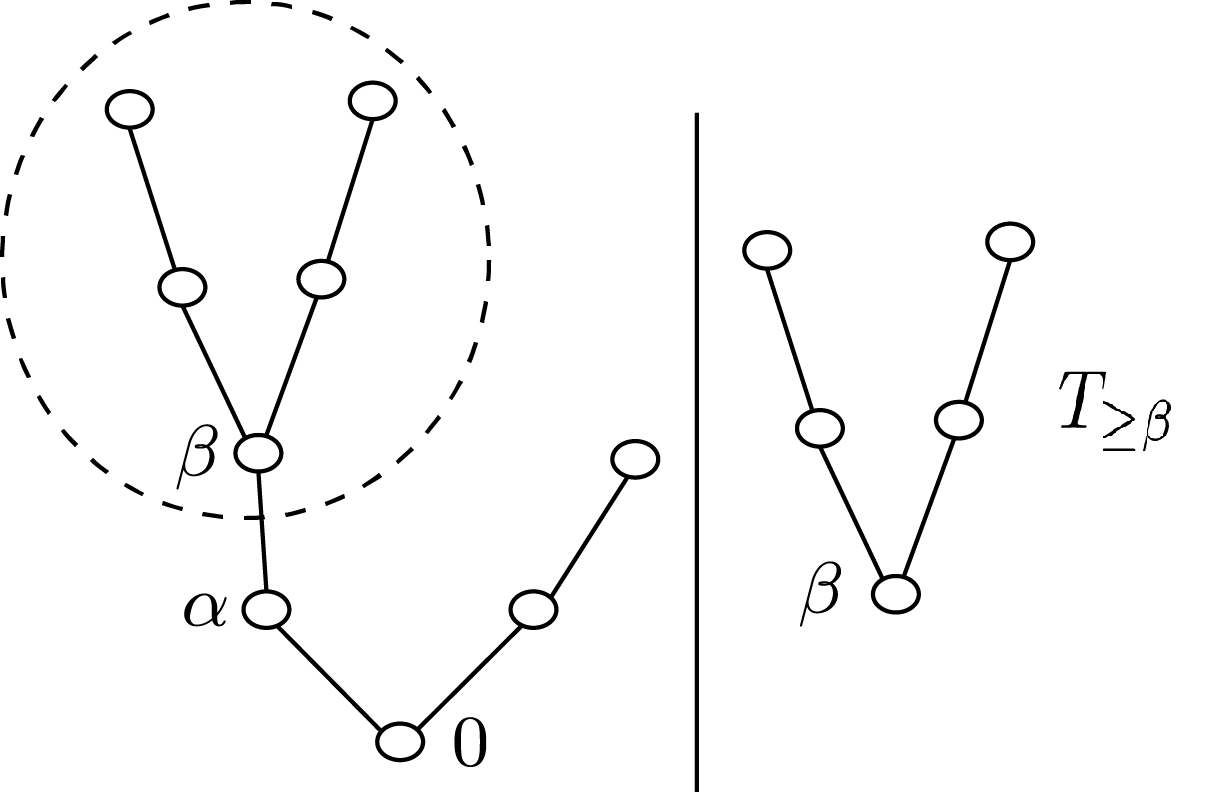}
		\caption{The rooted subtree $(T_{\geq \beta}, \beta)$ corresponds to the pseudoelliptic tree obtained by separating $Y_\beta$ from $Y_\alpha$.}\label{figure:pseudotree}
\end{figure}

We consider the gluing sequence
$$
0 \to L^{[m]}|_{Y_{\geq \beta}}(-M) \to L^{[m]} \to L^{[m]}|_{Y'} \to 0. 
$$

\begin{lemma}\label{L.M>0} $L^{[m]}|_{Y_{\geq \beta}}(-M)$ is ample on $Y_{\geq \beta}$ and $L^{[m]}|_{Y'}$ is nef with positive degree on $M$. \end{lemma} 

\begin{proof} 
Since no coefficients have been reduced on $Y_{\geq \beta}$, then $L|_{Y_{\geq \beta}}$ is ample so $L^{[m]}|_{Y_{\geq \beta}}(-M)$ is ample for $m$ large enough. By assumption $\mathrm{Coeff}(A,F_\calA) = \mathrm{Coeff}(A,F_\calB)$ so that in particular $L.M > 0$ since $(K_X + S + F_\calA)|_{Y'}$ is ample.
\end{proof} 

Thus, we have that $H^i(Y_{\geq \beta},L^{[m]}|_{Y_{\geq \beta}}(-M)) = 0$ for $i > 0$ so $H^i(Y, L^{[m]}) = H^i(Y', L^{[m]}|_{Y'})$. Therefore it suffices to prove vanishing for $L^{[m]}|_{Y'} = \calO_Y(m(K_Y + M + F_\calA|_{Y'}))$, where $(Y', M + F_\calA|_{Y'})$ is a pseudoelliptic tree with a leaf $Y_\alpha$ such that coefficients on $Y_\alpha$ have been reduced. By induction on the number of leaves, we may suppose $(Y', F_\calA + M)$ is a pseudoelliptic tree that $F_\calA|_{Y_\alpha} < F_{\calB}|_{Y_\alpha}$ for every leaf $\alpha$ and $M$ is a sum of reduced arithmetic genus $1$ components of intermediate pseudofibers on the leaf components. Furthermore, by the above Lemma, $(K_Y + F_\calA + M).M_0 > 0$ for each component $M_0$ of $M$. 

Since $F_\calA|_{Y_\alpha} < F_{\calB}|_{Y_\alpha}$ for every leaf $\alpha$, we can apply Step 5 to the pseudoelliptic tree $(Y, F_\calA + M)$. That is, we can separate the irreducible components of $Y$. This reduces to proving
$$H^i(Y, L^{[m]}|_{Y}(-M')) = 0,$$
for $i > 0$ where $Y$ is an irreducible pseudoelliptic surface, $M'$ is a union of the supports of arithmetic genus $1$ components of intermediate fibers, and we can write
$$
L^{[m]}|_{Y}(-M') = \calO_Y\left(m\left(K_Y + F_\calA|_{Y} + G + M + \frac{m - 1}{m}M'\right)\right),
$$
where $M$ is a union of components of intermediate fibers with $L.M > 0$ and $G$ is a twisted fiber. Denoting $\Delta = F_\calA|_{Y} + G + M + \frac{m - 1}{m}M'$, we are then left to consider an irreducible pseudoelliptic pair $(Y,\Delta)$. \\

\proofstep  Let $(Y,\Delta)$ be an irreducible pseudoelliptic pair as in Step 5 case (1) or the conclusion of Step 6 above and suppose $K_Y + \Delta$ is big. Now we may take the partial log semi-resolution $\mu : X \to Y$ by the associated elliptic surface $(X \to C, S)$ over a necessarily rational curve.

We may write
$$
K_X + S + F = \mu^*(K_Y + \Delta) + tS 
$$
for $0 \le t \le 1$ where $F = \mu_*^{-1}\Delta$ is a union of (not necessarily reduced) fiber components. By Proposition \ref{logcanonical} we have 
$$
\mu_*\calO_X\big(m(K_X + S+ F)\big) = \calO_Y\big(m(K_Y + \Delta)\big). 
$$ 

Proceeding as in Proposition \ref{vanishing>0}, we aim to apply Fujino's Theorem \ref{fujinovanishing}. That is, we need to check that $K_Y + \Delta$ is big on each of the slc strata of $(Y, \Delta)$. The divisor $K_Y + \Delta$ is big on $Y$ by assumption, and it is trivially big on the zero dimensional log canonical centers. This leaves the one dimensional log canonical centers of $(Y, \Delta)$. These are exactly the images of the log canonical centers of $(X, S + F_\calA)$, noting that the image of $S$ is a point so we need not consider it. Now $(K_Y + \Delta).M > 0$ for any log canonical center supported on an intermediate pseudofiber where a pseudoelliptic tree was attached by Lemma \ref{L.M>0}. Using the projection formula we may proceed to check the other log canonical centers as in the proof of Proposition \ref{vanishing>0} (where as in \emph{loc. cit.} we might need to first contract necessarily rational $E$ components of marked intermediate pseudofibers). Therefore $$H^i\Big(Y,\calO_Y\big(m(K_Y + \Delta)\big)\Big) = 0,$$ for all $i > 0$ by Theorem \ref{fujinovanishing}. 

In particular, this finishes the proof for the following cases, where we know that $K_Y + \Delta$ is big: \begin{itemize} \item $Y$ is an irreducible pseudoelliptic with $\deg \LL \geq 3$ (Proposition \ref{properlyelliptic}), \item $Y$ is an irreducible pseudoelliptic with with $\deg \LL = 2$ with $\Delta \not \sim_{\Q} 0$ (Proposition \ref{K3}). \end{itemize} 

We are left to deal with irreducible pseudoelliptics $(Y, \Delta)$ with $K_Y + \Delta$ \emph{not} big. \\

\proofstep Let $(Y,\Delta)$ be a pseudoelliptic with $\deg \LL = 1$ and Iitaka dimension $\kappa(K_Y + \Delta) = 1$. By Proposition \ref{rational}, $K_Y + \Delta \sim_\Q \mu_*\Sigma$, where $\mu : Z \to Y$ is the contraction of the section of the associated elliptic surface $f : Z \to C$ and $\Sigma$ is a rational multisection of $f$ disjoint from $S$. Since $\Sigma$ is in the locus where $\mu$ is an isomorphism, it suffices to prove that $H^i(Z, \calO_Z(m\Sigma)) = 0. 
$

By Lemma \ref{semipositivity} below, $f_*\calO_Z(m\Sigma)$ is a semipositive vcector bundle on $\PP^1$. In particular,
$$
H^i(\PP^1, f_*\calO_Z(m\Sigma)) = 0.
$$
Furthermore, $G.\Sigma > 0$ for any irreducible fiber $G$, since $\Sigma$ is an effective multisection. In particular, $m\Sigma - (K_Z + S + F_\calA)$ is $f$-nef and $f$-big over each slc stratum of $Y$ for $m \gg 1$. Therefore 
$
R^if_*\calO_Z(m\Sigma) = 0
$
by Fujino's  Theorem \ref{fujinovanishing} for $i > 0$ and $m \gg 0$, and so $H^i(Z, \calO_Z(m\Sigma)) = 0$ by the Leray spectral sequence.

\begin{lemma}\label{semipositivity} Let $f : Y \to C$ be a fibration from an irreducible slc surface to a reduced curve and let $\Sigma$ be a multisection with $|\Sigma|$ basepoint free. Then $f_*\calO_Y(m\Sigma)$ is a semipositive vector bundle on $C$ for $m \gg 0$. 
\end{lemma}

\begin{proof} Note that for a finite morphism $\varphi: B \to C$ and a vector bundle $\mathcal{E}$ on $C$, the vector bundle $\mathcal{E}$ is semipositive on $C$ if and only if $\varphi^*\mathcal{E}$ is semipositive on $B$. 

Since $R^if_*\calO_Y(m\Sigma) = 0$ for $i > 0$ and $m \gg 0$, we may apply cohomology and base change to conclude that $f_*\calO_Y(m\Sigma)$ is a vector bundle and its formation commutes with basechange. Let $\nu : \widetilde{C} \to C$ be the normalization. Consider the base change
$$
\xymatrix{\widetilde{Y} \ar[r]^\mu \ar[d]_{g} & Y \ar[d]^f \\ \widetilde{C} \ar[r]^\nu & C} 
$$
Since $\nu$ is finite, it suffices to prove that $\nu^*f_*\calO_Y(m\Sigma) \cong g_*\calO_Y(m \mu^*\Sigma)$ is semipositive. Since $\mu^*\Sigma$ is a multisection of $g$, we may assume without loss of generality that $C$ is smooth. 

Since $|m\Sigma|$ is basepoint free, there exists a general member $D$ that is a reduced divisor, i.e. $D = \sum_{i = 1}^s C_i$
where each $C_i$ is a distinct multisection. There exists a finite base change
$$
\xymatrix{Y' \ar[r]^\phi \ar[d]_g & Y \ar[d]^f \\ C' \ar[r]^\varphi & C}
$$
such that $\phi^*C_i \sim_\Q \sum_j S_{ij}$ is a sum of distinct sections. Therefore we may assume that $$m\Sigma \sim_\Q D = \sum_{i = 1}^s C_i,$$  is a finite sum of distinct sections. 

Now we use induction on the number of sections. Let $D_k = \sum_{i = 1}^k C_i$. Then for $k = 1$, $D_1$ is a single section and $f_*\calO_Y(D_1)$ is a line bundle with a section induced by pushing forward the section $\calO_Y \to \calO_Y(D_1)$. Therefore $f_*\calO_Y(D_1)$ is semipositive. 

Now let $k \geq 2$. We consider the exact sequence
$$
0 \to \calO_Y(D_{k - 1}) \to \calO_Y(D_k) \to \calO_{C_k}(D_{k - 1}|_{C_k}) \to 0
$$
induced by adding $C_k$. Since the sections are all distinct, then $D_{k - 1}.C_k \geq 0$ so that $\calO_{C_k}(D_{k - 1}|_{C_k})$ is a semipositive line bundle. Pushing forward, and noting that \\ $R^1f_*\calO_Y(D_{k - 1}) = 0$, then the sequence
$$
0 \to f_*\calO_Y(D_{k - 1}) \to f_*\calO_Y(D_k) \to f_*\calO_{C_k}(D_{k - 1}|_{C_k}) \to 0
$$
is exact. The first term is semipositive by the inductive hypothesis, and the last term is semipositive since $f$ is an isomorphism on $C_k$. Therefore the middle term is semipositive.  \end{proof} 

\proofstep Finally, we are left with the case of an irreducible pseudoelliptic $(X, \Delta)$ with Iitaka dimension $\kappa(K_X + \Delta) = 0$, which occurs for $\deg \LL = 1$ and $\deg \LL = 2$ as in Propositions \ref{rational} and \ref{K3}. We have $K_X + \Delta \sim_{\Q} 0$ so for $m$ large and divisible enough, the sheaf $\calO_X\big(m(K_X + \Delta)\big) = \calO_X$. Thus we need to compute cohomology of the structure sheaf on an irreducible pseudoelliptic surface with associated elliptic surface $(f: Y \to \mb{P}^1, S)$ and contraction $p : Y \to X$. 

By Proposition \ref{prop:gr}, $R^ip_*\calO_Y = 0$ for $i > 0$ and $p_*\calO_Y = \calO_X$ so $H^i(X, \calO_X) = H^i(Y, \calO_Y)$. Similarly, if $Y$ has any nonreduced twisted fiber (in fact it can have at most one such fiber otherwise the section would not contract to form a pseudoelliptic by Proposition \ref{prop:adjunction}), let $\mu : Y' \to Y$ be the partial resolution blowing up the twisted fibers to their intermediate models. Then $\mu : Y' \to Y$ contracts the genus $0$ component $A$ of each such intermediate fiber. Again by Proposition \ref{prop:gr}, $R^i\mu_*\calO_{Y'} = 0$ for $i > 0$ so we may suppose without loss of generality that $Y$ has no nonreduced twisted fibers. In particular, we may assume the section $S$ of $Y$ passes through the smooth locus of $f : Y \to \mb{P}^1$, that is, $(f : Y \to \mb{P}^1, S)$ is standard. 

\begin{lemma}\label{lemma:Hi(O)} Let $(f : Y \to \mb{P}^1, S)$ be a standard elliptic surface. Then $H^1(Y, \calO_Y) = 0$ and $H^2(Y,\calO_Y) = \deg \LL - 1$.
\end{lemma} 

\begin{proof} For a standard elliptic surface, we have $R^1f_*\calO_Y = \LL^{-1}$ (see \cite[II.3.6]{mir3}). Since $\LL$ is effective, then $H^0(\mb{P}^1, \LL^{-1}) = 0$ and $H^1(\mb{P}^1, \calO_{\mb{P}^1}) = 0$ so the Leray spectral sequence for $f$ implies that $H^1(Y, \calO_Y) = 0$ and $H^2(Y, \calO_Y) = H^1(\mb{P}^1, \LL^{-1}) = H^0(\mb{P}^1, \calO_{\mb{P}^1}(\deg \LL - 2))$ by Serre duality. 
\end{proof}

Now we apply the lemma to the case at hand. If $\calA$ is not zero, then by Propositions \ref{rational} and \ref{K3}, we must have $\deg \LL = 1$ and so $H^i(Y, \calO_Y) = 0$. If $\calA = 0$, then we are assuming that if the \emph{original} broken elliptic surface from Step 1 is fibered over a rational curve, then the total degree of $\LL \neq 2$ so any (pseudo)elliptic component with $\deg \LL = 2$ must be attached. In particular the log canonical must be big on it by Proposition \ref{K3}. \\\end{steps}

This concludes the proof of Theorem \ref{thm:vanishing}.  \end{proof}

\begin{remark}\label{rem:k3} Note that the only place we required that if $\calA = 0$ and $p_g(C) = 0$ then $\deg \LL \neq 2$ is in Lemma \ref{lemma:Hi(O)}. However, this is a completely trivial edge case as $L^{[m]} = \calO_X$, so while $H^2(X, L^{[m]}) \neq 0$, the formation of $L^{[m]}$ still commutes with base change in families. 
\end{remark}

\begin{theorem}\label{thm:basechange}(Invariance of log plurigenera)  Let $\pi: (X \to C, S + F_\calB) \to B$ be a family of $\calB$-stable broken elliptic surfaces over a reduced base $B$. Let $0 \le \calA \le \calB$ such that $K_{X/B} + S + F_\calA$ is a $\pi$-nef and $\Q$-Cartier divisor. Then $\pi_*\calO_X\big(m(K_{X/B} + S + F_\calA)\big)$ is a vector bundle on $B$ whose formation is compatible with base change $B' \to B$ for $m \geq2$ divisible enough.  \end{theorem}

\begin{proof} If either \begin{enumerate*}[label = (\alph*)] \item $\calA$ is not identically zero, or \item $\calA = 0$ but $p_g(C) \neq 1$ and if $p_g(C) = 0$ then $\deg \LL \neq 2$, \end{enumerate*} then we may apply Theorem \ref{vanishing>0} to see that $H^i(X_b, m(K_{X/B} + S + F_\calA)|_{X_b}) = 0$ for $i > 0$ and for all closed points $b \in B$ so the result follows by the proper base change theorem. \\

Suppose $p_g(C_b) = 1$ and $\calA = (0, \ldots, 0)$ is identically zero. We may suppose that $\calB = (a, 0, \ldots, 0)$ has exactly one nonzero entry by applying the above result and first decreasing all but one coefficient to $0$. Then $(C_b, ap)$ is a one pointed stable genus $1$ curve with $a < 1$. In particular, it is irreducible. Therefore $X_b$ contains a single elliptically fibered component $X_0 \to C_b$ with a marked divisor $(F_a)_b$ lying over $p$. There are three cases to consider: 
\begin{enumerate}[label = (\roman*)] \item $X_0$ is properly elliptic and $F_a = aF$ is a Weierstrass fiber, \item $X_0$ is properly elliptic and there is a pseudoelliptic tree $(Y_b, (F_a)_b|_{Y_b})$ attached to an intermediate fiber $E \cup A$ above $p$ and $(F_a)_b|_{X_0} = aA$, or \item $\deg \LL = 0$ and $X_b = X_0 = C_b \times E_b$ is a product. \end{enumerate} 
In either of case $(i)$ and $(ii)$ there may be unmarked type I or II pseudoelliptics attached elsewhere to $X_0$. \\

Let us denote $L^{[m]} := \calO_X(m(K_{X/B} + S))$. The linear series $|L^{[m]}_b|$ is semi-ample by Proposition \ref{prop:abundance} and $L^{[m]}_b.S_b = 0$. Thus the linear series factors through the contraction of $S_b$ which gives a morphism $\mu : X_b \to Z_b$. In case $(i)$ and $(ii)$ this maps onto an slc broken pseudoelliptic surface with an elliptic singularity at $\mu(S_b)$ and $\mu_*L^{[m]}_b = \calO_{Z_b}(m(K_{Z_b}))$. 

In case $(i)$, the pair $(X_0, S_b)$ is log general type by Proposition \ref{properlyelliptic} and for every other component $W \subset X_b$, the line bundle $L^{[m]}_b|_W = \calO_{X_b}(m(K_{X_b} + S_b + (F_a)_b))|_W$ is still ample on $W$. It follows that $K_{Z_b}$ is big and nef on each slc stratum of $(Z_b, 0)$ so $H^1(Z_b, \mu_*L_b^{[m]}) = 0$ by Theorem \ref{fujinovanishing}. 

In case $(ii)$ we consider the attaching sequence
$$
0 \to \mu_*L^{[m]}_b|_{Z_b'}(-M) \to \mu_*L^{[m]}_b \to L^{[m]}_b|_{Y_b} \to 0
$$
where $M = \mu_*E$ is the curve along which $Y_b$ is attached to $Z_0 = \mu(X_0)$ and $Z_b'$ is the union of components of $Z_b$ not contained in $Y_b$. Now $L_b|_{Y_b} = K_{Y_b} + E$ and $(Y_b, E)$ is a broken pseudoelliptic tree and we can apply Theorem \ref{thm:vanishing} to conclude $H^1(Y_b, L^{[m]}_b|_{Y_b}) = 0$. On the other hand, $K_{Z_b}|_{Z_b'} = K_{Z_b'} + M$ so 
$$
\mu_*L_b^{[m]}|_{Z_b'}(-M) = \calO_{Z_b'}\left(m(K_{Z_b'} + \frac{m - 1}{m} M)\right).
$$
As in case $(i)$, the divisor $K_{Z_b'} + \frac{m-1}{m} M$ is big and nef on every slc stratum of $(Z_b', \frac{m-1}{m}M)$ so $$H^1(Z_b', \mu_*L_b^{[m]}|_{Z_b'}(-M)) = 0$$ by Theorem \ref{fujinovanishing} and we conclude that $H^1(Z_b, \mu_*L_b^{[m]}) = 0$. 

In either case $(i)$ or $(ii)$, it follows that $H^1(X_b, L^{[m]}_b) = H^0(Z_b, R^1\mu_*L^{[m]}_b)$. Now $$(K_{X_b} + S_b)|_{S_b} \sim_\Q 0$$ so $L^{[m]}_b|_{S_b} = \calO_{S_b}$ for $m$ divisible enough. On the other hand $S_b$ is an irreducible nodal arithmetic genus $1$ curve so by the theorem on formal functions, $R^1\mu_*L^{[m]}_b$ is a skyscraper sheaf supported on $\mu(S_b)$ with $1$ dimensional fiber and $h^1(X_b, L^{[m]}_b) = 1$. \\

In case $(iii)$, consider the trivial fibration $f : X_b \to C_b$ with section $S_b$ and $g(C_b) = 1$. Then $K_{X_b} = 0$ and $L^{[m]}_b = \calO_{X_b}(mS_b)$. Furthermore, $R^1f_*\calO_{X_b}(mS_b) = 0$ and $f_*\calO_{X_b}(mS_b) = \calO_{C_b}^{\oplus m}$ for $m \geq 1$ by \cite[II.3.5 and II.4.3]{mir3}. It follows that $h^1(X_b, L^{[m]}_b) = h^1(C_b, \calO_{C_b}^{\oplus m}) = m$. \\

In each case $h^1\big(X_b, \calO_X(m(K_{X/B} + S))|_{X_b}\big)$ is constant and since the base is reduced, it follows from cohomology and base change over a reduced base scheme (see e.g. \cite[Theorem 1.2]{osserman}) that formation of $$\pi_*\calO_X\big(m(K_{X/B} + S))$$ is compatible with base change.  \\

Finally, when $\calA = 0$, $p_g(C) = 0$ and $\deg \LL = 2$, we have that $\calO_{X_b}(m(K_{X_b} + S_b)) = \calO_{X_b}$ (Remark \ref{rem:k3}) and $h^0(X_b,\calO_{X_b}) = 1$ is constant for $b \in B$ since $X_b$ is connected and reduced so formation of $\pi_*L^{[m]}$ commutes with base change.\end{proof}

\begin{remark} Note that for the first part of Theorem \ref{thm:basechange}, we do \emph{not} need to assume that $B$ is reduced. Indeed, whenever we can apply the vanishing theorem \ref{vanishing>0}, a strong form of proper base change ensures that the formation of $\pi_*\calO_X(m(K_{X/B} + S + F_\calA))$ commutes with arbitrary base change for any base $B$. It is only in the second case when the higher cohomology \emph{does not} vanish that we need to assume $B$ is reduced to apply cohomology and base change. This will not matter in the sequel as we restrict to normal base schemes.  \end{remark} 

The above Theorem \ref{thm:basechange} allows us to compute the $\calA$-stable model of a $\calB$-stable family by working fiber by fiber. This is used in the next section to explicitly describe the steps of the log MMP to compute the stable limits of a $1$-parameter family. Then in Section \ref{sec:reduction}, we use Theorem \ref{thm:basechange} to show that performing the steps of the log MMP on a family of elliptic surfaces is functorial. This leads to the existence of reduction morphisms between moduli spaces of elliptic surfaces for weights $0 \leq \calA \le \calB$ as above.

\section{Stable reduction}\label{sec:stablereduction}

The goal of this section is to prove a stable reduction theorem 
for $\calA$-broken elliptic surfaces in the spirit of La Nave \cite{ln}. As a result we obtain properness of the moduli spaces $\calE_{v,\calA}$ and give a description of the surfaces that appear in the boundary. 

Our strategy for stable reduction is to first compute stable limits of a family of irreducible elliptic surfaces with large coefficients. To this end, in \cite{tsm} we use the theory of \emph{twisted stable maps} to compute stable limits in the case when all singular fibers are marked with coefficient $b_i = 1$. We then run the minimal model program while reducing the coefficients to compute the stable limit for weights $\calA$ using the classification of log canonical models of elliptic surfaces as well Theorem \ref{thm:basechange} and the results of Appendix \ref{sec:appendix}.

\subsection{Wall and chamber structure}\label{sec:walls} 

Let $\calD \subset (\Q \cap [0,1])^n$ be the set of \emph{admissible weights}: weight vectors $\calA$ such that $K_X + S + F_\calA$ is pseudoeffective. A \emph{wall and chamber} decomposition of $\calD$ is a finite collection $\calW$ of hypersurfaces (the \emph{walls}), and the \emph{chambers} are the connected components of the complement of $\calW$ in $\calD$.

First we describe a wall and chamber decomposition of $\calD$ defined by where the log canonical model of an $\calA$-slc elliptic surface changes as $\calA$ varies. The collection of walls $\calW$ corresponds to the steps in the MMP required to produce a stable limit of a family of elliptic surfaces over a smooth curve. 

\begin{definition} \label{def:walls} The collection $\calW$ consists of the following types of walls:
\begin{itemize}
\item[I] A wall of \emph{Type $\mathrm{W_{\mathrm{I}}}$} is a wall arising from the log canonical transformations seen in Section \ref{sec:local} -- that is, the walls where the fibers of the relative log canonical model transition from twisted, to intermediate, to Weierstrass fibers.
\item[II] A wall of \emph{Type $\mathrm{W_{\mathrm{II}}}$} is a wall at which the log canonical morphism induced by the log canonical contracts the section of some components. By Corollary \ref{cor:stablecurve} these are the same as the walls for Hassett space $\overline{\calM}_{g,\calA}$.
\item[III] A wall of \emph{Type $\mathrm{W_{\mathrm{III}}}$} is a wall where the morphism induced by the log canonical contracts a rational pseudoelliptic component. These are determined by Proposition \ref{rational} and Remark \ref{rem:contract}
\end{itemize}

There are also \emph{boundary walls} given by $a_i = 0,1$ at the boundary of $\calD$. These can be of any of the types above.   \end{definition}

\begin{remark} By the results of Appendix \ref{sec:appendix} by G. Inchiostro, namely Theorem \ref{Teo:appendix:giovanni}, there are no other birational transformations that occur when computing the stable limit of a family of $\calA$-weighted stable elliptic surfaces.
\end{remark}

\begin{theorem} \label{thm:finitewalls} The non-boundary walls of each type are described as follows: 
\begin{enumerate}[label = (\alph*)]
\item Type $\mathrm{W_{\mathrm{I}}}$ walls are defined by the equations
$$
a_i =\frac{1}{6}, \frac{1}{4}, \frac{1}{3}, \frac{1}{2}, \frac{2}{3}, \frac{3}{4}, \frac{5}{6}.
$$ 

\item Type $\mathrm{W_{\mathrm{II}}}$ walls are defined by equations
$$
\sum_{j = 1}^k a_{i_j} = 1.
$$
where $\{i_1, \ldots, i_k\} \subset \{1, \ldots, n\}$. When the base curve is rational there is another $\mathrm{W_{\mathrm{II}}}$ wall at
$$
\sum_{i = 1}^r a_i = 2.
$$
\item Type $\mathrm{W_{\mathrm{III}}}$ walls where a rational pseudoelliptic component contracts to a point are given by
$$
\sum_{j = 1}^k a_i = c
$$
where $\{i_1, \ldots, i_k\} \subset \{1, \ldots, n\}$ and $c = \frac{1}{6}, \frac{1}{4}, \frac{1}{3}, \frac{1}{2}, \frac{2}{3}, \frac{3}{4}, \frac{5}{6}$ are the log canonical thresholds of minimal Weierstrass fibers. \\

\item Finitely many Type $\mathrm{W_{\mathrm{III}}}$ walls where an isotrivial rational pseudoelliptic component contracts onto the $E$ component of a pseudoelliptic surface it is attached to. 

\end{enumerate}

\noindent In particular, there are only finitely many walls and chambers. 
\end{theorem}

\begin{proof} Part $(a)$ follows from the results of Section \ref{sec:local} since these are exactly the coefficients at which minimal Weierstrass cusps transition from Weierstrass models to intermediate models. 

Part $(b)$ follows from Proposition \ref{prop:adjunction} since $(K_X + S + F_\calA).S > 0$ if and only if the base curve is a weighted stable pointed curve. When $\sum a_{i_j} = 1$, the section of any component fibered over a rational curve, which is attached to the other components of the surface along one attaching fiber, and contains marked fibers $i_1, \ldots, i_k$ gets contracted. When the base curve is $\PP^1$ and $\sum a_i = 2$, the section of \emph{every} elliptic surface gets contracted so that all $\calA$-slc elliptic surfaces have \emph{only} pseudoelliptic components. 

For type $\mathrm{W_{\mathrm{III}}}$ walls [(c) and (d)], note that by the results of Section \ref{sec:pseudocontraction}, if $K_X + S + F_\calA$ is \emph{not big} on a pseudoelliptic component $Y$, then the component is rational. Suppose $Y$ is attached to a component $E$ of an intermediate (pseudo)fiber $A \cup E$ on $X'$. 

By Proposition \ref{rational}, if $(K_X + S+ F_\calA)|_Y$ is not big then either the log canonical linear series contracts $Y$ to a point or contracts $Y$ along a morphism $Y \to E$. 

In particular, $Y$ contracts to a point if and only if $E$ contracts to a point. We can do the computation by restricting $(K_X + S + F_\calA)$ to $X'$ first. In this case we have an intermediate marked (pseudo)fiber $F_a = aA + E$ where 
$$
a = \sum_{i \ : \ F_{a_i} \text{ lies on }Y} a_i
$$
is a sum of markings on the pseudoelliptic $Y$ by Equation \ref{eq:A}. Then by Proposition \ref{prop:lct} the fiber $F_a$ contracts onto its Weierstrass model if and only if $a \le c$ where $c$ is the log canonical threshold of the Weierstrass fiber. By Theorem \ref{thm:thm1} the nonzero log canonical thresholds are exactly the ones written above. 

On the other hand, suppose that $Y$ contracts along a morphism $Y \to E$. By Proposition \ref{stableattaching}, the curve $E$ is not a fiber of type $\mathrm{I}_n$, so it has to support a nonreduced twisted fiber. In particular $E \cong \mb{P}^1$ is a rational curve. Let $\mu : Y' \to Y$ be the associated elliptic surface. Then there is a morphism $Y' \to E$ by composition which is induced by the linear series $\mu^*((K_X + S + F_\calA)|_Y) = K_{Y'} + \alpha S' + \mu_*^{-1}(F_\calA)|_Y$. By Proposition \ref{rational}, the coefficient $\alpha > 0$, and the generic fiber of $Y' \to E$ is a generic multisection $M$ of the elliptic fibration $Y' \to C$ that is disjoint to $S'$. In particular, $M.S' = 0$.  

Let $p : Y' \to Y_0$ be the contraction of the rational components of each intermediate fiber of $Y' \to C$ and let $S_0 = p_*(S'), F_0 = p_*(\mu_*^{-1}(F_\calA)|_Y)$ and $M_0 = p_*(M)$. We claim that $Y' \to E$ factors through $Y' \to Y_0$. That is, $M.A = 0$ for $A$ the genus zero component of each intermediate fiber. By the inequalities in the proof of Lemma \ref{lemma:first:contract:the:section} as well as the fact that $S_0^2 \le 0$ by Lemma \ref{lemma:the:coef:of:S:pulling:back:the:lc:on:a:psudo:positive}, we have
\begin{align*}
0 &= (K_{Y'} + \alpha S' + \mu_*^{-1}(F_\calA)|_Y).M = (K_{Y'} + \mu_*^{-1}(F_\calA)|_Y).M \geq (K_{Y_0} + F_0).M_0 \\
&= m(K_{Y_0} + F_0).S_0 \geq m(K_{Y_0} + \alpha S_0 +  F_0).S_0 \geq m(K_{Y'} + \alpha S' + \mu_*^{-1}(F_\calA)|_Y).S' = 0
\end{align*}
so all the inequalities are equalities. Now $\alpha > 0$ by Proposition \ref{rational} and by the first inequality on the second line, $S_0^2 = 0$ on the twisted model. Now we conclude by the following lemma. \end{proof} 


\begin{lemma} Suppose $(f : X \to C, S)$ is an irreducible relative log canonical model with only twisted fibers. Suppose further that $S^2 = 0$. Then $X$ is the quotient of a trivial fibration $B \times E \to B$. \end{lemma}

\begin{proof} By \cite[Proposition 4.12]{tsm}, the pair $(f : X \to C, S)$ is the coarse space of a family of stable curves over a stable curve denoted by $\calX \to \calC$. Pick a projective curve with a finite cover $B \to \calC$ and consider the pullback $Y \to B$ of $\calX \to \calC$. Then $Y \to B$ is a family of stable curves over $B$ with section $T$ pulled back from $S$. On the other hand, by the projection formula, $S^2 = 0$ implies that $T^2 = 0$. However, a Weierstrass elliptic fibration with $T^2 = - \deg \LL = 0$ is trivial by \cite[III.1.4]{mir3}.  \end{proof} 


\subsection{The birational transformations across each wall}

We wish to describe the birational transformations that a family of $\calA$-broken stable elliptic surfaces undergoes as $\calA$ crosses a wall. Let $(f : X \to C, S + F_\calA) \to B$ be a one parameter family of broken elliptic surfaces with normal generic fiber and special fiber $f' : X' \to C'$. 

\subsubsection{Type $\mathrm{W_{\mathrm{I}}}$}\label{sec:WI} If $F$ is a minimal intermediate fiber, then at the wall at coefficients $a_i = 0, \frac{1}{6}, \frac{1}{4}, \frac{1}{3}, \frac{1}{2}, \frac{2}{3}, \\ \frac{3}{4}$ or $\frac{5}{6}$ (depending on the Kodaira fiber type of $F$), $X$ undergoes a divisorial contraction where $F$ transforms from an intermediate to Weierstrass model. Similarly, at the boundary wall $a_i = 1$, the surface $X$ undergoes a divisorial contraction where $F$ transforms from an intermediate into a twisted fiber. 

\subsubsection{Type $\mathrm{W_{\mathrm{II}}}$}\label{sec:WII} Let $\calA_0$ be a weight on the non-boundary wall defined by $\sum a_{i_j} = 1$ for $\{i_1, \ldots, i_k\}$. Let $\calA_\pm$ be in the adjacent chambers with $\sum a_{i_j} = 1 \pm \epsilon$ for $\epsilon$ very small. $\calA_0$ is on a wall for the Hassett space where a leaf component $C_0'$ of the central fiber $C'$ is contracted. 

La Nave studied this situation in \cite[Section 4.3]{ln}. At $A_0$, the section $S_0'$ of an elliptic component $X_0'$ lying over $C'$ in the central fiber $X' \to C'$ of the $\calA_+$ stable family $X \to C$ must contract by Proposition \ref{prop:adjunction}. This is a log canonical contraction of the pair $(X, S + F_{\calA_0})$, but it is an extremal contraction of the pair $(X, S + F_{\calA_-})$. 

Since the total space $X$ is a threefold and $S_0'$ is a curve, this is a small contraction so we must perform a flip to compute the $\calA_-$ stable model. La Nave computes this flip explicitly using a local toric model around $S_0'$ inside the total space $X$ \cite[Theorem 7.1.2]{ln}. This leads to the formation of a type $I$ pseudoelliptic surface $Z$ in the central fiber attached to the component $E$ of an intermediate (pseudo)fiber $E \cup A$ where $A$ is the flipped curve, as depicted in Figure \ref{figure:typeWII}: 

\begin{figure}[h!]
	\includegraphics[scale=.75]{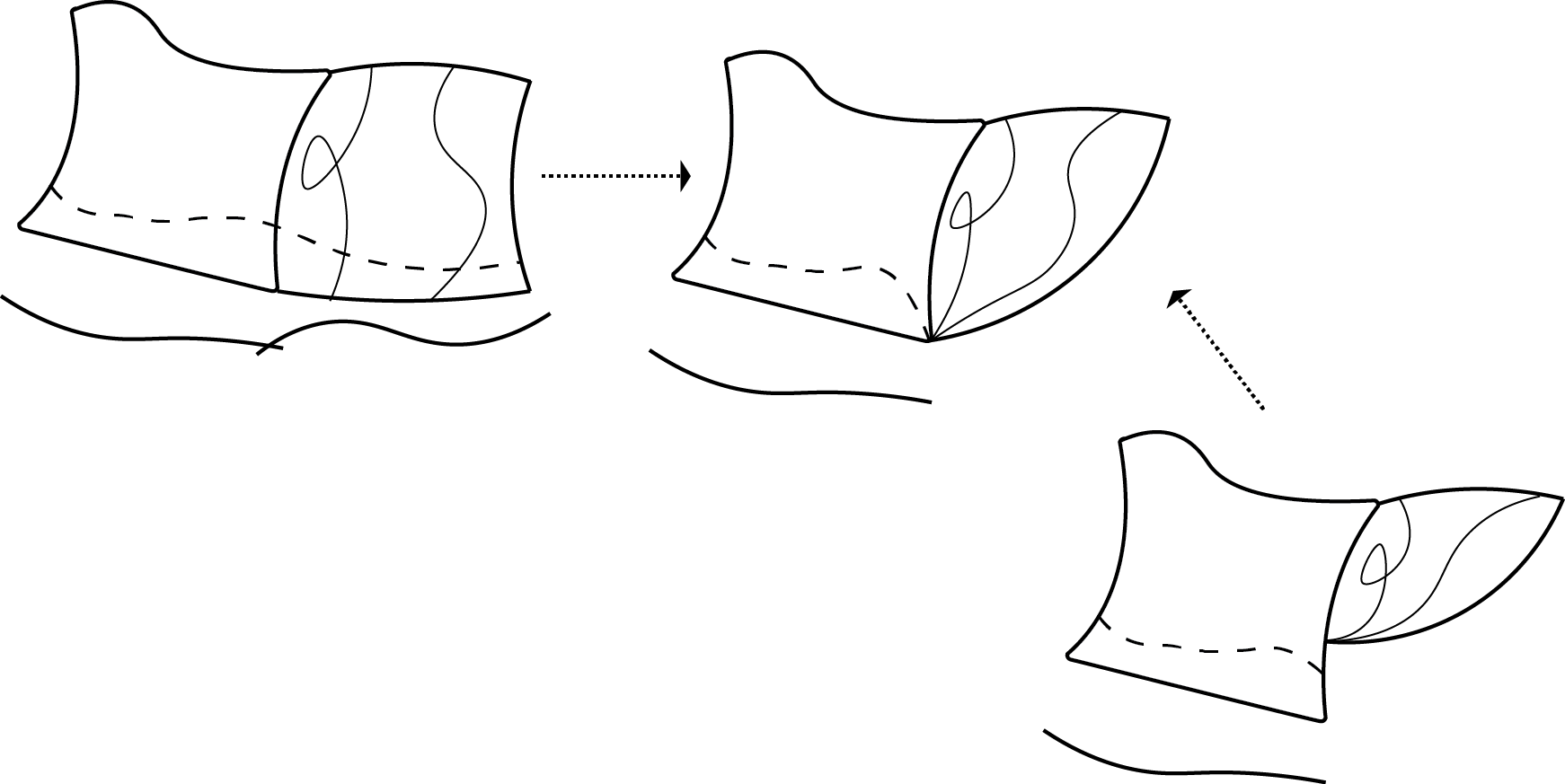}
	\caption{This depicts, from left to right, the central fiber of the $\calA_+$, $\calA_0$ and $\calA_-$ stable families where $\calA_0$ is a type $\mathrm{W_{\mathrm{II}}}$ wall. }\label{figure:typeWII}
\end{figure}

At a boundary type $\mathrm{W_{\mathrm{II}}}$ wall, a rational component $C_0'$ of $C'$ which is not a leaf may contract. The contraction of the corresponding section component $S_0'$ in the central fiber $X'$ of $X \to B$ is a log canonical contraction which forms a type $\mathrm{II}$ pseudoelliptic surface. 

Finally when the genus of the base curve is $0$, we must consider the wall defined by $\sum a_i = 2$. In this case the base curve is contracted to a point and so the section of the total family $X \to C \to B$ is contracted by a divisorial log canonical contraction. This produces a one parameter family of pseudoelliptic surfaces $Z \to B$ with normal generic fiber and special fiber consisting of an $\calA$-broken pseudoelliptic surface. 

\subsubsection{Type $\mathrm{W_{\mathrm{III}}}$}\label{sec:WIII} At $\calA_0$, there is a pseudoelliptic component $Z$ in the central fiber of $X'$ for which $K_X + S + F_{\calA_0}$ is nef but not big. Then the total space $(X, S + F_{\calA_0})$ undergoes a divisorial log canonical contraction $X \to Y$ which contracts $Z$ onto either a point or a curve as determined by Remark \ref{rem:contract}. 

When $Z$ contracts to a point, this results in a cuspidal cubic fiber on the central fiber $Y' \to C'$ of $Y$ at the point that $Z$ contracted to. When $\calA_0$ is not on a boundary wall, then the surface $Y'$ has at worst a rational singularity at this point by Proposition \ref{stableattaching}. At a boundary, the contraction of $Z$ may produce an elliptic singularity at the cusp. 

\subsubsection{Multiple walls} Figures \ref{figure:brokenchain}, \ref{figure:wallcrossing} and \ref{figure:wallcrossing2} illustrate some of the multi-step transformations the central fiber can undergo due to crossing several walls at once.

\begin{figure}[h!]
	\includegraphics[scale=.75]{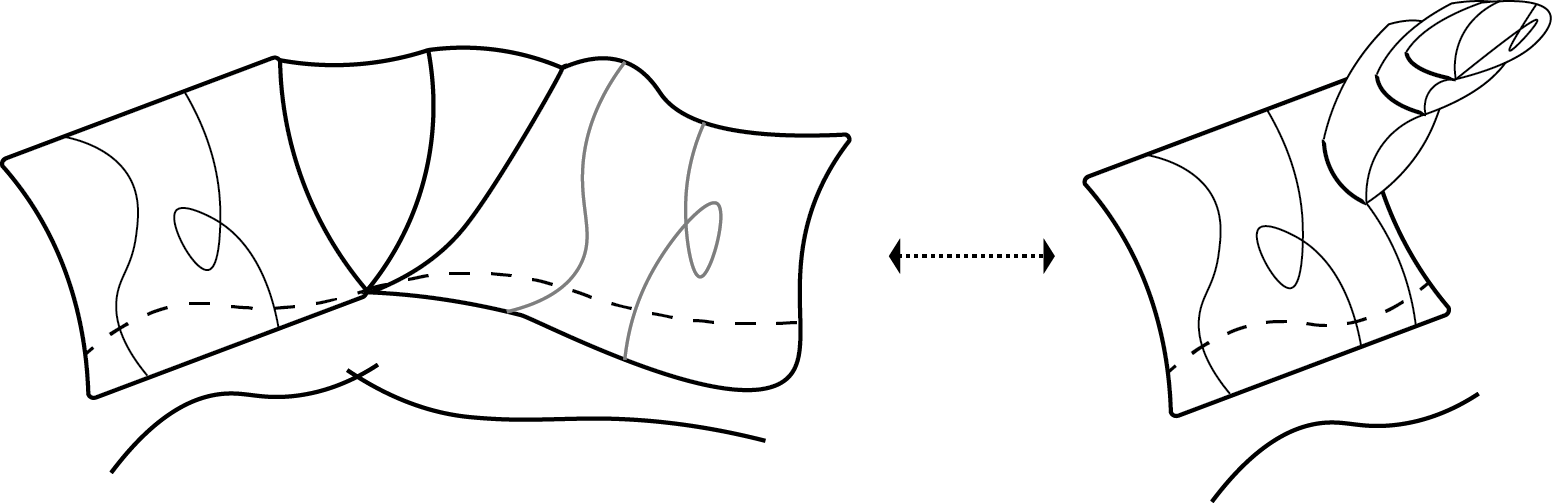}
	\caption{Here a type $\mathrm{W_{\mathrm{II}}}$ wall is crossed which causes the right most component to transform into a type $I$ pseudoelliptic. However, that then causes the type $\mathrm{II}$ pseudoelliptics to also become type $I$ since they have no marked fibers.}\label{figure:brokenchain}
\end{figure} 

\begin{figure}[h!]
	\includegraphics[scale=.75]{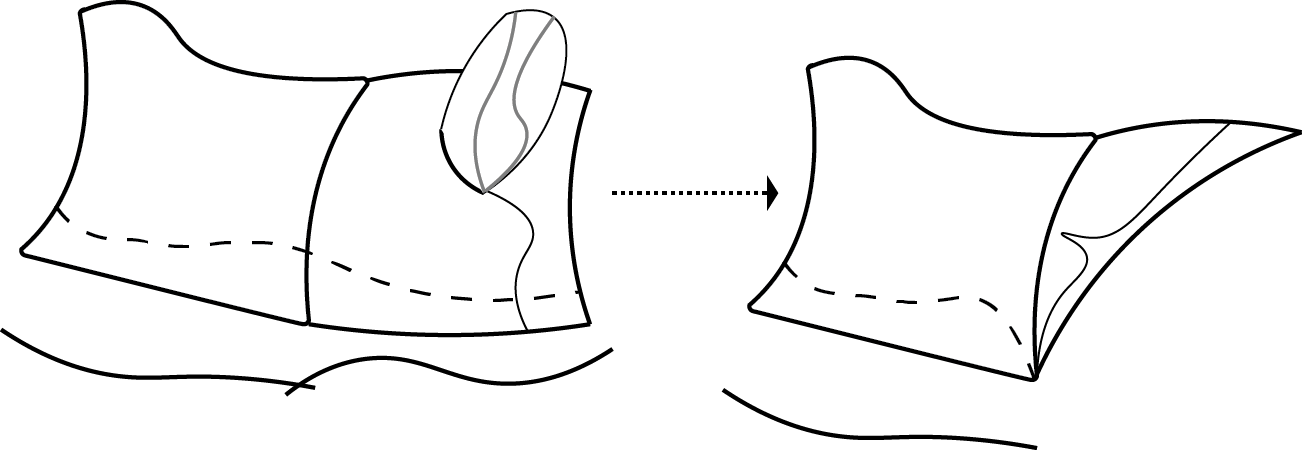}
	\caption{This is a simultaneous $\mathrm{W_{\mathrm{II}}}$ and $\mathrm{W_{\mathrm{III}}}$ wall where the type $I$ pseudoelliptic component contracts onto a point and the right most elliptic component becomes a pseudoelliptic. }\label{figure:wallcrossing}
\end{figure} 

\begin{figure}[h!]
	\includegraphics[scale=.8]{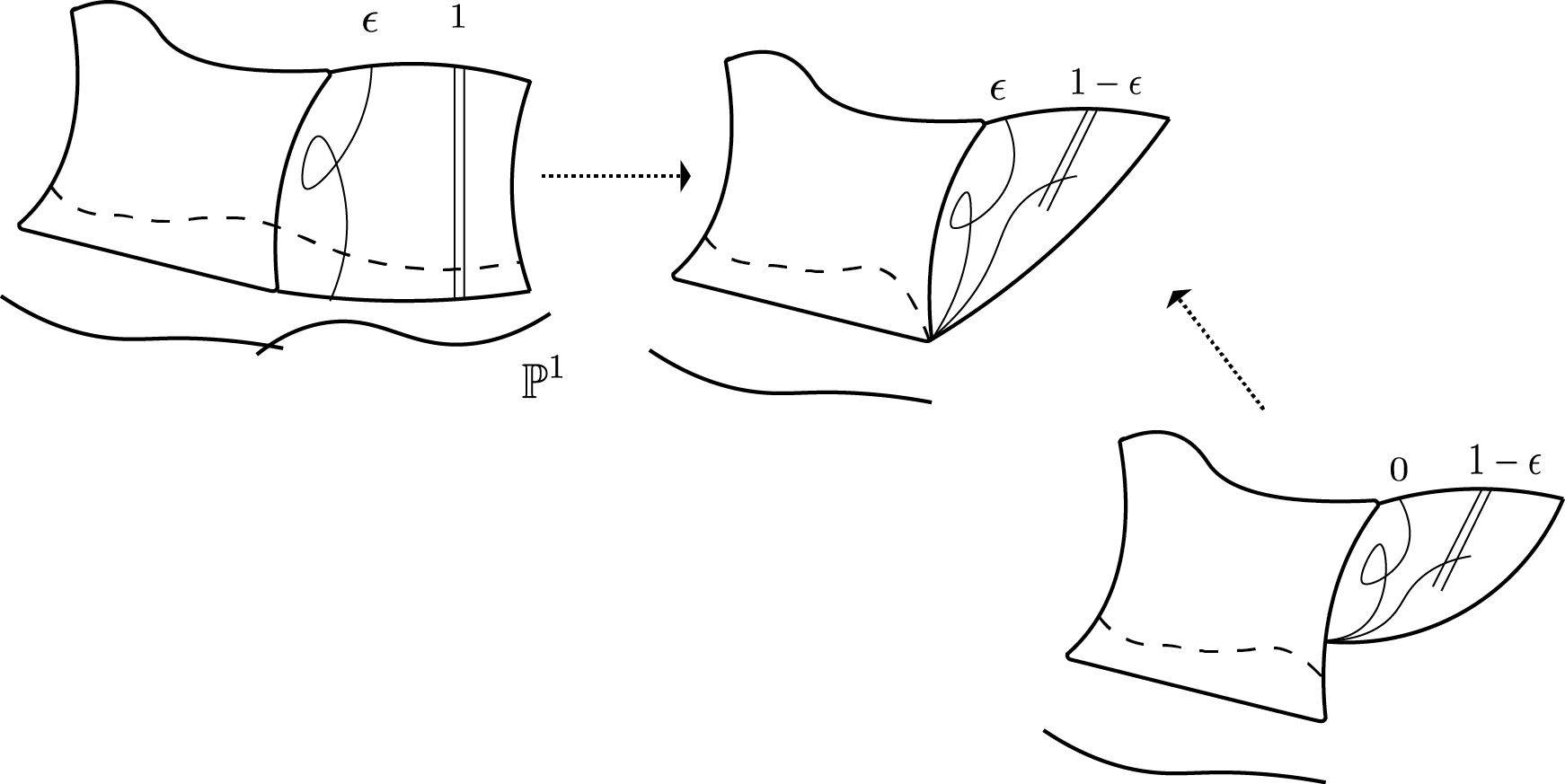}
	\caption{This is a simultaneous $\mathrm{W_{\mathrm{I}}}$ and $\mathrm{W_{\mathrm{II}}}$ where the twisted fiber becomes an intermediate fiber and a type $I$ pseudoelliptic forms. }\label{figure:wallcrossing2}
\end{figure}

\subsection{Explicit stable reduction}

Recall the following definition (see Definition 4.9 \cite{tsm}):

\begin{definition} An $\calA$-broken elliptic surface  $(f: X \to C, S + F_\calA)$ is \textbf{twisted} if $a_i = 1$ for all $i$, there are no pseudoelliptic components, and the support of every non-reduced fiber is contained in $\mathrm{Supp}(F_\calA)$. \end{definition}

In \cite{tsm}, we used the Abramovich-Vistoli moduli space of twisted stable maps \cite{av2} to construct a proper moduli space of twisted elliptic surfaces analogous to the moduli spaces of fibered surfaces considered in \cite{av}. In particular, in \cite[Proposition 4.12]{tsm} we proved that any surface with only twisted fibers is the coarse space of a stacky family of stable curves.  This is the starting point for computing the stable limits in $\calE_{v,\calA}$ for any $\calA$.

Given a family of $\calA$-stable irreducible elliptic surface $(X \to C, S + F_\calA) \to U$ over a punctured curve $U$, the idea is to 
\begin{enumerate} \item increase the coefficients so that $a_i = 1$ for all $i$, and  \item add the supports of any unstable fibers to the boundary divisor. \end{enumerate} 

Then the stable model of this new pair will be a family of twisted elliptic surfaces. By the results of \cite{tsm},  this family extends uniquely after a base change $U' \to U$. Finally, we can run the log MMP to compute the stable model as we decrease coefficients again. This is analogous to the approach used by La Nave \cite{ln} to compute stable limits of stable Weierstrass fibrations, i.e. when $\calA = 0$. 

\begin{theorem} \label{thm:proper} The moduli stack $\calE_{v,\calA}$ is proper. \end{theorem} 

\begin{proof}

Consider a family of normal $\calA$-stable elliptic surfaces $(X^0, S^0 + F_{\calA}^0) \to C^0 \to U$ over $U = B \setminus p$, a smooth curve minus one point. Let $\calB_1 = (1, \ldots, 1)$ be the constant weight $1$ vector and let $G^0 = G^0_{r+1} + \ldots + G^0_{s}$ be the reduced divisor whose support consists of the singular fibers not contained in $\Supp(F_\calA)$. Define $D^0_{\calB} = F^0_{\calB_1} + G^0$ so that $(X^0, S^0 + D^0_{\calB}) \to C^0 \to U$ is a family of pairs with all non-stable fibers marked and all fibers marked with coefficient one. We index the weight vector $\calB = (b_1, \ldots, b_r, b_{r+1}, \ldots, b_s)$ such that $b_i$ for $i = 1, \ldots, r$ are the coefficients of the original marked fibers $F_i$. 

After performing a log resolution, we can take the log canonical model of this pair to obtain a family of slc elliptic surfaces $(X^1, S^1 + D_{\calB}^1) \to C^0 \to U$, such that all fibers are either stable or twisted, and all fibers that are not of type $\mathrm{I}_n$ are contained in either the double locus of $X$ or in $D_{\calB}^1$. By \cite[Corollary 5.10]{tsm}, there is a map $C^0 \to \overline{M}_{1,1}$ making $(X^1, S^1 + D^1_\calB) \to \overline{M}_{1,1}$ an Alexeev stable map from a twisted elliptic surface (see Section 5 and Proposition 5.2 of \cite{tsm}).

By \cite[Proposition 5.2]{tsm}, the moduli space of Alexeev stable maps from a twisted elliptic surface is proper. Therefore, after a finite base change $B' \to B$, this family extends uniquely to a family $(Z_1, S_1 + D_{\calB}) \to C_1 \to B'$ of twisted elliptic surfaces over $B'$ with a well defined $j$-invariant map $C_1 \to \overline{M}_{1,1}$. Furthermore the central fiber consists of only elliptic components fibered over a possibly reducible nodal curve.  

Now consider the line segment $\calA(t) := t\calB + (1-t)\calA_0$ for $t \in [0,1]$ where \linebreak $\calA_\delta = (a_1, \ldots, a_r, 0, \ldots, 0)$. By Theorem \ref{thm:finitewalls}, there are finitely many $t_0 = 0, t_1, \ldots, t_{n-1}, t_n = 1$ so that $\calA(t_k)$ are on walls.

By invariance of log plurigenera (Theorem \ref{thm:basechange}), we can compute the stable model of $$\pi: (Z_1, S_1 + D_{\calB(t)}) \to C_1 \to B'$$ as we decrease $t$ from $t = 1$ by taking the stable model of each fiber as long as $K_\pi + S_1 + D_{\calB(t)}$ remains $\pi$-nef, and $\Q$-Cartier. First we need that each wall-crossing preserves the structure of a fibered surface:  

\begin{lemma}\label{factoring} Let $(f:X \to C, S + F_\calA)$ be a $\calB$-broken stable elliptic surface. Let $\calA \le \calB$ and denote by $X'$ and $C'$ the stabilizations of $X$ and $C$ with respect to $\calA$ respectively. Then there exists a commutative diagram as follows.

$$
\xymatrix{X \ar[r] \ar@{-->}[d] & C \ar[d] \\ X' \ar[r] & C'}
$$
 
\end{lemma}

\begin{proof} Let $\phi : X \dashrightarrow X'$ be the log canonical birational map induced by $m(K_X + S + F_\calA)$. We can factor $\phi$ into a sequence of type $W_{\mathrm{I}}$, $W_{\mathrm{II}}$ and $W_{\mathrm{III}}$ birational transformations described in Section \ref{sec:stablereduction}. We reduce to checking that for each of these birational transformations, there is a compatible factorization of $X' \to C'$.

\begin{enumerate}[label = \Roman*.]
\item If $\phi$ is a $\wi$ type transformation, that is, a transition between twisted, intermediate and Weierstrass fibers, then $\phi$ is a composition of blowups and blowdowns of fiber components so there is a factorization $X' \to C$. 

\item If $\phi$ is a $\wii$ type transformation, then there is a diagram 
$$
\xymatrix{X_- \ar[rd] & & X_+ \ar[ld] \\ & X_0 & }
$$
where $X_+ \to X_0$ is the contraction of a section component $X_- \to X_0$ is birational on every component and $\phi$ is either $X_+ \to X_0$ or $X_+ \dashrightarrow X_-$ (see Section \ref{sec:flip} for details). By Proposition \ref{prop:adjunction}, the map $X_+ \to X_0$ contracts a section component if and only if that component of the base curve is contracted by the morphism $C \to C'$. Therefore there is a unique factorization $X_0 \to C'$ also inducing a unique map $X_- \to C'$ by composition. 

\item If $\phi$ is a type $\wiii$ transformation, then it contracts components of $X$ which are contracted to a point by $f : X \to C$. Therefore there is a unique factorization $X' \to C$. \end{enumerate}

 By Theorem \ref{Teo:appendix:giovanni}, the above are the only birational transformations that occur. \end{proof}

Now for $0 < \epsilon \ll 1$, there exists a unique family of $\calA(1 - \epsilon)$-weighted stable elliptic surfaces $(Z_{1 - \epsilon}, S_{1 -\epsilon} + D_{\calA(1 - \epsilon)}) \to C_1 \to B'$ obtained by the blowup from twisted to intermediate models of the marked fibers. Then one performs the following whenever $\calA(t)$ crosses a wall:

\begin{itemize}
\item Each time $t$ crosses a type $\mathrm{W_{\mathrm{I}}}$ or $\mathrm{W_{\mathrm{III}}}$ wall $t_k$, the family undergoes a divisorial contraction as described in Sections \ref{sec:WI} and \ref{sec:WIII}. In this case one obtains a $\calA(t_k)$-weighted stable family $(Z_{t_k}, S_{t_k} + D_{\calA(t_k)}) \to C_{t_k} \to B'$;
\item Across a type $\mathrm{W_{\mathrm{II}}}$ wall $t_l$, the map $X_{t} \to X_{t_l}$ is a flipping contraction of a section of a component of the central fiber. As described in Section \ref{sec:WII} there is a unique flip $X_{t'} \to X_{t_l}$ constructed by La Nave in \cite{ln} giving an $\calA(t')$-weighted stable family $$(Z_{t'}, S_{t'} + D_{\calA(t')}) \to C_{t'} \to B'.$$
\end{itemize}

Since there are only finitely many walls crossed, we eventually obtain an $\calA(0) = \calA_0$-weighted stable family $\pi : (Z_{\delta}, S_{0} + D_{\calA_0}) \to C_{0} \to B'$. Forgetting about the auxillary divisors $G$ now marked with $0$, this is in fact a $\calA$-stable family. \end{proof}

\begin{theorem} \label{thm:boundary} The stable limit of a family of irreducible $\calA$-stable elliptic surfaces is an $\calA$-broken stable elliptic surface. In particular, the compact moduli stack $\calE_{v,\calA}$ parametrizes $\calA$-broken stable elliptic surfaces. \end{theorem}

\begin{proof} Every step of the proof of Theorem \ref{thm:proper} produces a central fiber which is a broken elliptic surface. \end{proof}

\begin{cor}\label{cor:chambers} For any $\calA$ and $\calB$ within the same chamber, $\calE_{v,\calA} \cong \calE_{v,\calB}$. \end{cor} 

\begin{proof} The walls of type $\mathrm{W_{\mathrm{I}}}$, $\mathrm{W_{\mathrm{II}}}$ and $\mathrm{W_{\mathrm{III}}}$ describe precisely when the log canonical divisor is nef rather than ample. Within a chamber $K_X + S + F_\calA$ is ample if and only if $K_X + S + F_\calB$ is ample and the log canonical models are the same. It follows that $\calE^m_{v,\calA,n}(T) = \calE^m_{v,\calB,n}(T)$ for any normal base $T$ and so $\calE_{v,\calA} \cong \calE_{v,\calB}$ by Proposition \ref{uniquenorm}. \end{proof}

\section{Reduction morphisms}\label{sec:reduction}

We begin by reviewing the notion of reduction morphisms present in the work of \cite{has}.

\subsection{Hassett's moduli space}\label{sec:hassettreduction}
Recall the moduli spaces $\overline{\calM}_{g, \calA}$, parametrizing genus $g$ curves with $r$ marked points weighted by a weight vector $\calA = (a_1, \dots, a_r)$ were defined in \cite{has}. Hassett studied what happened as one lowers the weight vector $\calA$. Namely, the following theorem guarantees the existence of reduction morphisms on the level of moduli spaces. 

\begin{theorem}\cite[Theorem 4.1]{has}\label{hastheorem} Fix $g$ and $n$ and let $\calA = (a_1, \dots, a_r)$ and $\calB = (b_1, \dots, b_r)$ two collections of admissible weights and suppose that $\calA \leq \calB$. Then there exists a natural birational reduction morphism
$$\overline{\calM}_{g,\calB} \to \overline{\calM}_{g, \calA}.$$ Given an element $(C, p_1, \dots, p_r) \in \overline{\calM}_{g, \calB}$, the image in $\overline{\calM}_{g, \calA}$ is obtained by 
collapsing components of $C$ along which $K_C + a_1p_1 + \dots + a_rp_r$ fails to be ample. \end{theorem}

We will construct analagous reduction morphisms on the moduli spaces $\calE_{v, \calA}$ and their universal families which are compatible with the reduction morphisms of Hassett in the following way. The image curve $(C, D_\calA)$ (recall Definition \ref{def:DA}) is naturally an $\calA$-weighted curve in the sense of Hassett. We obtain a natural forgetful morphism from ${\calE}_{v,\calA} \to \overline{\calM}_{g,\calA}$ for all $0 \leq a \leq 1$ (see Corollary \ref{cor:forget}) and the reduction morphisms (Theorem \ref{thm:reduction}) will commute with Hassett's reduction morphisms above.  

\subsection{Preliminary results}

Let $(f : X \to C, S + F_\calA)$ be an $\calA$-broken elliptic surface. Denote by $D_\calA$ the weighted divisor on $C$ corresponding to the weighted marked fibers of $f: X \to C$ (Definition \ref{def:DA}). Then $(C, D_\calA)$ is a weighted pointed nodal curve in the sense of Hassett.

\begin{lemma}\label{basechangeforC} Let $\xymatrix{(X, S + F_{\calB}) \ar[r]^(.75){f} & C \ar[r]^{q} & B}$ be a flat family of $\calB$-stable elliptic surfaces over a base $B$. Denoting the composition $p = q \circ f$, then the formation of $p_*(f^*\omega_ q(D_\calA)^{[m]})$ commutes with base change for any $\calA \le \calB$ and $m \geq 1$.  \end{lemma}

\begin{proof} By Lemma \ref{pushpull}, $p_*(f^*\omega_q(D_\calB)^{[m]}) = q_*f_*f^*\omega_q(D_\calB)^{[m]} = q_*\omega_q(D_\calB)^{[m]}$, and the latter commutes with base change by Proposition 3.3 of \cite{has}.  \end{proof}

First, we show the base curve of our weighted elliptic surface pairs are weighted stable curves in the sense of Hassett, so we can use these spaces to gain understanding of $\calE_{v, \calA}$. 

\begin{cor}\label{cor:forget} There is a natural forgetful morphism $\calE_{v,\calA} \to \overline{\calM}_{g,\calA}$ given by sending a family of $\calA$-broken stable elliptic surfaces $p : (f : X \to C,  S + F_\calA) \to B$ to the family of $\calA$-weighted stable curves $q: (C, D_\calA) \to B$. \end{cor}

\begin{proof} By Lemma \ref{basechangeforC}, the formation of $p_*(f^*\omega_q(D_\calA)^{[m]}) = q_*\omega_q(D_\calA)^{[m]}$ commutes with base change. Therefore it suffices to check that $(C_b, (D_\calA)_b)$ is an $\calA$-stable curve for each $b \in B$ and this is Corollary \ref{cor:stablecurve}.  \end{proof}

\subsection{Reduction morphisms} 
We are now ready to state and prove our main theorem on reduction morphisms for moduli of elliptic surfaces analagous to \cite[Theorem 4.1]{has}.

\begin{theorem}\label{thm:reduction} Let $\calA$ and $\calB$ be rational tuples such that $\calA \le \calB$. Suppose that that $\calA(t)$ never lies on a Type $W_{\mathrm{II}}$ wall for $t > 0$ (see Remark \ref{rem:A(t)}). Then there exists a reduction morphisms $\rho_{\calA,\calB} : \calE_{v,\calB} \to \calE_{v,\calA}$. If we further suppose that $\lfloor \calA \rfloor = \lfloor \calB \rfloor$, then there exists a compatible $\widetilde{\rho}_{\calA,\calB} : \calU_{v,\calB} \to \calU_{v,\calA}$ making the following diagram commute:
$$
\xymatrix{\calU_{v,\calB} \ar[r]^{\widetilde{\rho}_{\calA,\calB}} \ar[d] & \calU_{v,\calA} \ar[d] \\ \calE_{v,\calB} \ar[r]^{\rho_{\calA,\calB}} & \calE_{v,\calA}}.
$$
All of the above reduction morphisms commute via the forgetful morphism of Corollary \ref{cor:forget} with the reduction morphisms for Hassett space. 
\end{theorem}

\begin{remark} The condition $\lfloor \calA \rfloor = \lfloor \calB \rfloor$ just means that $a_i = 1$ if and only if $b_i = 1$. We consider the case when $a_i < b_i = 1$ in Proposition \ref{prop:1-e}.  
\end{remark}

 \begin{proof} The proof that $\rho_{\calA,\calB}$ is a morphism is modeled off of the proof of \cite[Theorem 4.1]{has}. Let $\calA = (a_1, \dots, a_r)$ and $\calB = (b_1, \dots, b_r)$ be so that $\calA \le \calB$. Denote $\calA(t) := (1 - t)\calA + t \calB$, where $t \in \Q$ and $0 \le t \le 1$.  
 
With notation from the proof of Theorem \ref{thm:stack}, we define a natural transformation of pseudofunctors 
$$
{\calE}^m_{v,\calB,n}(B) \to \calE^m_{v,\calA,n}(B)
$$
for a normal base scheme $B$ that is compatible with base change. This will lead to a morphism of moduli spaces $\rho_{\calA,\calB}: \calE_{v,\calB} \to \calE_{v,\calA}$ by Proposition \ref{normfunctor}. There are necessarily finitely many $t_0 = 0, t_1, \ldots, t_{k - 1}, t_k = 1$ so that $\calA(t_j)$ lie on walls and for any $t \neq t_j$, weights $\calA(t)$  are not on any wall. It is clear that the weight vectors $\calA(t_j) \le \calA(t_{j + 1})$ satisfy the hypothesis of the theorem so it suffices to construct reduction morphisms
$
\rho_{\calA(t_{j}), \calA(t_{j+1})}$
so that
$$
\rho_{\calA, \calB} = \rho_{\calA(t_0), \calA(t_1)} \circ \ldots \circ \rho_{\calA(t_{k-1}), \calA(t_{k})}
$$

Therefore we may assume that $\calA(t)$ does not lie on a wall for any $t \neq 0,1$, and that $\calA$ is either in a chamber or on a wall of type $\mathrm{W_{\mathrm{I}}}$ or $\mathrm{W_{\mathrm{III}}}$. Writing $\calA(t) = (a_1(t), \dots, a_r(t))$ this means explicitly that for all $0 < t < 1$, 
\begin{enumerate}[label = \roman*)]
\item $a_j(t) \neq\frac{1}{6}, \frac{1}{4}, \frac{1}{3}, \frac{1}{2},\frac{2}{3}, \frac{3}{4}, \frac{5}{6}$ (there are no type $\mathrm{W_{\mathrm{I}}}$ walls);
\item there is no subset $\{i_1, .., i_k \} \subset \{1, \dots, r\}$ such that $a_{\mathrm{I}_1}(t) + \dots + a_{i_k}(t) =1$ (there are no type $\mathrm{W_{\mathrm{II}}}$ walls); 
\item $K_X + S + F_{\calA(t)}$ is big on every irreducible component of every $\calB$-broken stable elliptic surface $(X \to C, S + F_\calB)$ (there are no type $\mathrm{W_{\mathrm{III}}}$ walls).
\end{enumerate}

Let $\pi: (f:X \to C, S + F_{\calB}) \to B$ be a family of $\calB$-broken elliptic surfaces over a normal base $B$. By our above assumption, $K_X + S + F_{\calA(t)}$ is ample for $t > 0$ and $K_X + S + F_{\calA(0)} = K_{X/B} + S + F_\calA$ is $\pi$-nef and $\Q$-Cartier. By Proposition \ref{prop:abundance}, $K_{X/B} + S + F_{\calA}$ is $\pi$-semiample.  Then we can write
\begin{align*}
C' &= \Proj_B\left(\bigoplus_m \pi_*f^*\omega_{C/B}(mnD_\calA)\right) \\
X' &= \Proj_B\left(\bigoplus_m \pi_*\calO_X(mn(K_{X/B} + S + F_\calA))\right)
\end{align*}
where $n$ is a large enough integer such that $nD_\calA$ and $n(K_{X/B} + S + F_\calA)$ are Cartier.

There are canonical maps $C \to C'$ and $X \dashrightarrow X'$. It follows from the basechange results Theorem \ref{thm:basechange} and Lemma \ref{basechangeforC} that $X'$ and $C'$ are families of $\calA$-broken stable elliptic surfaces and $\calA$-weighted pointed stable curves respectively. By Lemma \ref{factoring}, there is a map $f' : X' \to C'$ making $\pi' :(f' : X' \to C' , S + F_\calB) \to B$ a a family of $\calA$-broken stable elliptic surfaces over $B$. Since the construction of $\pi_*f^*\omega_{C/B}(mnD_\calA)$ and $\pi_*\calO_X(mn(K_{X/B} + S + F_\calA)$ commute with basechange by Theorem \ref{thm:basechange} and Lemma \ref{basechangeforC}, it follows that the construction of the family of $\calA$-broken stable elliptic surfaces is functorial in $B$. Furthermore, it is clear that the map on closed points, $\calE^m_{v, \calB, n}(k) \to \calE^m_{v,\calA,n}(k)$ is dominant on each component by observing that it is dominant on the locus of irreducible elliptic surfaces. This induces the required morphism
$$
\rho_{\calA,\calB} : \calE_{v,\calB} \to \calE_{v,\calA}.
$$

Next, we show existence of the morphism $\widetilde{\rho}_{\calA,\calB}$ on the level of universal families under the assumption $a_i = 1$ if and only if $b_i = 1$. In this case, there are no type $W_\mathrm{I}$ transformations from twisted to intermediate fibers so the rational map $X \dashrightarrow X'$ is actually a morphism $X \to X'$. The universal family $\calU_{v,\calB} \to \calE_{v,\calB}$ is itself a family of $\mc{B}$-weighted stable elliptic surfaces. Therefore applying the above construction gives a family $\mathcal Y \to \calE_{v,\calB}$ of $\calA$-stable elliptic surfaces with a morphism $\calU_{v,\calB} \to \mathcal Y$ over $\calE_{v,\calB}$. This induces the morphism $\rho_{\calA,\calB}$ so that
$$
\mc{Y} = \rho_{\calA,\calB}^*\ \calU_{v,\calA}.
$$
The composition $\calU_{v,\calB} \to \mc{Y} \to \calU_{v,\calA}$ gives the required $\widetilde{\rho}_{\calA, \calB}$.

The fact that these morphisms commute with the reduction morphisms for Hassett space is immediate since the forgetful map to the base curve is a morphism, and the family of base curves is stabalized by the linear series $\omega_{C/B}(nD_\calA)$ by Proposition \ref{prop:adjunction} and Lemma \ref{basechangeforC}. \end{proof}

\begin{cor}\label{cor:reduction} The reduction morphisms $\rho_{\calA,\calB}$ are surjective. \end{cor} 

\begin{proof} This follows since $\rho_{\calA,\calB}$ is a dominant morphism of proper stacks. \end{proof} 

\section{Flipping walls}\label{sec:flip}

Theorem \ref{thm:reduction}, shows that there are reduction morphisms 
$$
\rho_{\calA,\calB}: \calE_{v,\calB} \to \calE_{v,\calA}
$$
whenever $\calA(t):= (1 - t)\calA +  t\calB$ never crosses a type $\mathrm{W_{\mathrm{II}}}$ wall for $t \in (0,1]$. The key point is that if $\calA(t_0)$ is a type $\mathrm{W_{\mathrm{II}}}$ wall for $t_0 \in (0,1]$ and $t_\pm := t_0 \pm \epsilon$ for $0 < \epsilon \ll 1$, then
$$
K_X + S + F_{\calA(t_-)}
$$
is \emph{not necessarily} $\Q$-Cartier where $(f: X \to C, S + F_{\calA(t_0)})$ is an $\calA(t_0)$-stable elliptic surface. Therefore it no longer makes sense to construct the $\calA(t_-)$-stable model as a $\Proj$ of the section ring. 

Rather, to construct the $\calA(t_-)$-stable model from $(X, S + F_\calA(t_0))$, we need to first perform a log resolution to make the log canonical divisor $\Q$-Cartier before running the steps of the minimal model program. Therefore, across a type $\mathrm{W_{\mathrm{II}}}$ wall, we obtain a morphism resembling a \emph{flip} (see Figure \ref{figure:3}). 

We fix some notation. Let $t_0, t_\pm$ be as above where $\calA_0 := \calA(t_0)$ is on a wall of type $\mathrm{W_{\mathrm{II}}}$ and $\calA_- := \calA(t_-) < \calA_0 < \calA(t_+) =: \calA_+$ so that $\calA_\pm$ are in the interiors of chambers. We will use $(X_0, S_0 + F_{\calA_0})$ and $(X_\pm, S_\pm + F_{\calA_\pm})$ to denote $\calA_0$-stable and $\calA_\pm$-stable elliptic surfaces respectively. 

\begin{theorem}\label{thm:increase} There exist morphisms $\tilde{\epsilon}_{\calA_0}^{-}, \epsilon_{\calA_0}^-$ making the following diagram commute:
$$
\xymatrix{\calU_{v,\calA_-} \ar[r]^{\tilde{\epsilon}_{\calA_0}^-} \ar[d] & \calU_{v,\calA_0} \ar[d] \\ \calE_{v,\calA_-} \ar[r]^{\epsilon_{\calA_0}^-} & \calE_{v,\calA_0}}.
$$
\end{theorem}

\begin{proof}
This proof is analogous to the proof of Theorem \ref{thm:reduction}. Under these assumptions $K_{X_-} + S_- + F_{\calA_0}$ is a semiample $\Q$-Cartier divisor and the $\calA_0$-stable model is 
$$
\Proj\left(\bigoplus_{k \geq 0} H^0(X_-, \calO_X(km_0(K_{X_-} + S_- + F_{\calA_0})))\right)
$$ 
where $m_0$ is the index. Thus it suffices to prove a vanishing result analogous to Theorem \ref{thm:vanishing}. 

\begin{lemma} In this situation, $H^i(\calO_X(m(K_{X_-} + S_- + F_{\calA_0}))) = 0$ for $i > 0$ and $m$ large and divisible. 
\end{lemma}

\begin{proof} 
We consider the irreducible components of $X_-$. There are three types of components: 
\begin{enumerate}
\item[(a)] a pseudoelliptic whose section was contracted at the wall $\calA_0$;
\item[(b)] a component along which a pseudoelliptic from case $(a)$ is attached;
\item[(c)] a component not in either of the above cases.
\end{enumerate}

The pair $(X_-, S_- + F_{\calA_0})$ is slc and the linear series $|K_{X_-} + S_- + F_{\calA_0}|$ is semi-ample by Proposition \ref{prop:abundance}. It induces a morphism $g : X_- \to X_0$ which is necessarily an isomorphism on the components in cases $(a)$ and $(c)$ above. 

Suppose $X'$ is a component in case $(b)$. Then it is attached to a pseudoelliptic $Z$ in case $(a)$ along a fiber component $G$. As explained in La Nave (see Section 4.3 and Theorem 7.1.2 in \cite{ln}), the contraction of the section of a component to form $Z$ at the wall $\calA_0$ may be a log flipping contraction inside of the total space of a one parameter degeneration with central fiber $X_-$. In this case, $Z$ is a type $I$ pseudoelliptic attached along an irreducible pseudofiber $G$ to an intermediate (pseudo)fiber $G \cup A$ on $X'$ (see Figure \ref{figure:3}). The coefficient $\mathrm{Coeff}(A, F_\calA)$ given by the sum of weights of fibers on $Z$ as can be seen from La Nave's local toric model and the morphism $g : X_- \to X_0$ contracts $A$. In particular $\mathrm{Coeff}(A, F_{\calA_0}) = 1$.

\begin{figure}[h!]
	\includegraphics[scale=.75]{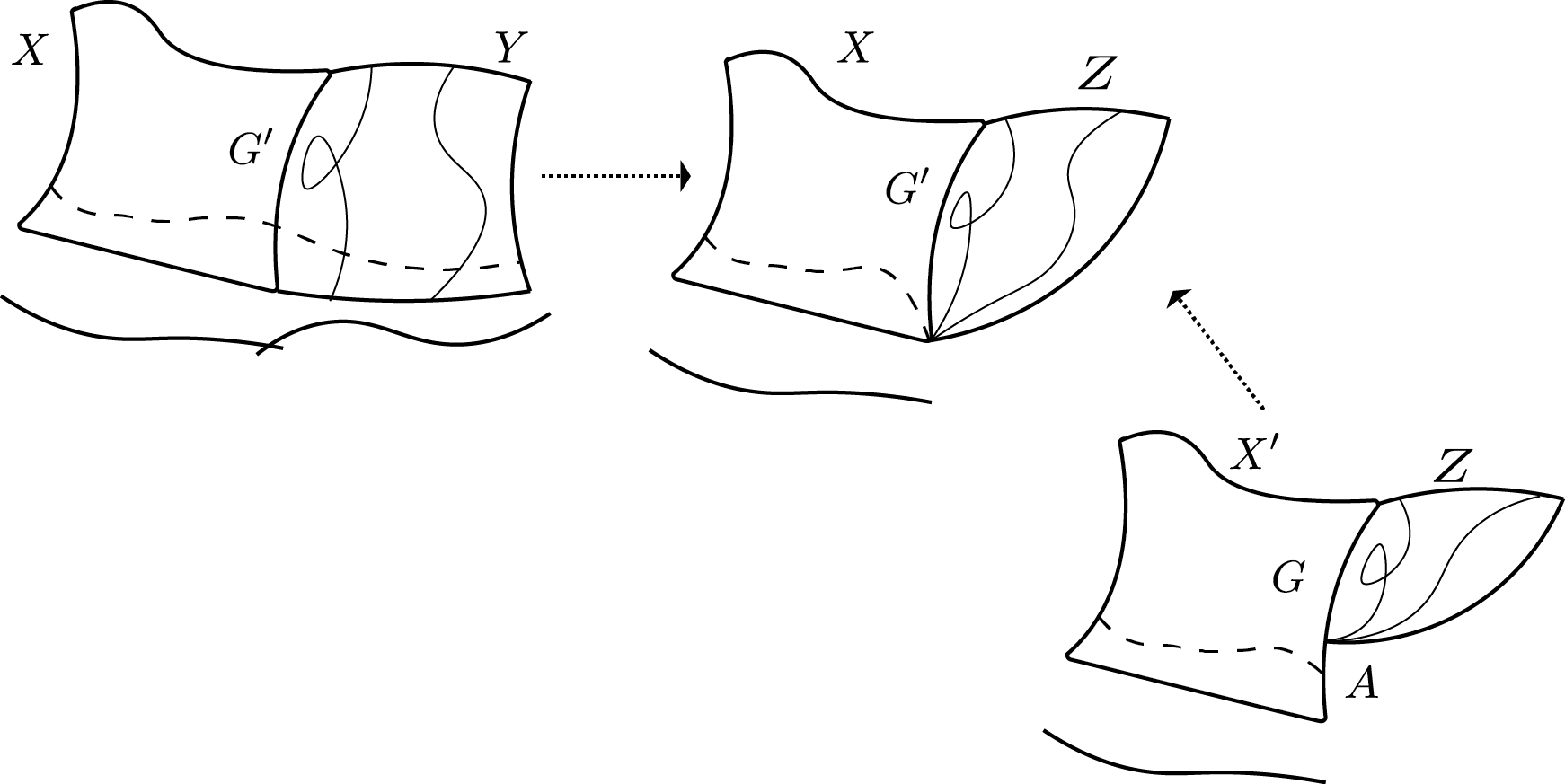}
	\caption{From left to right, the $\calA_+$, $\calA_0$ and $\calA_-$ stable models. The sum of the weights of the marked pesudofibers on $Z$ is equal to the coefficient of $A$ in $F_\calA$.  }\label{figure:3}
\end{figure} 

Thus $g : X_- \to X_0$ is precisely the contraction of these rational curves $A$ produced by La Nave's flips. Denote $S_- + F_{\calA_0} = \Delta$. Then by Proposition \ref{prop:gr}, 
$$
R^1g_*\calO_{X_-}(m(K_{X_-} + \Delta)) = 0.
$$
On the other hand, $g_*\calO_{X_-}(m(K_{X_-} + \Delta)) = \calO_{X_0}(m(K_{X_0} + g_*\Delta))$  by Proposition \ref{logcanonical}, since $g_*^{-1}g_*\Delta + \mathrm{Exc}(g) = \Delta$ as each curve $A$ appears with coefficient $1$. Therefore
$$
H^1(X_-, \calO_{X_-}(m(K_{X_-} + \Delta))) = H^1(X_0,\calO_{X_0}(m(K_{X_0} + g_*\Delta))) = 0
$$
since $(X_0, g_*\Delta) = (X_0, S_0 + F_{\calA_0})$ is the $\calA_0$-stable model.  \end{proof}

Now we can proceed as in the proof of Theorem \ref{thm:reduction}: let $\pi : (X_- \to C, S_- + F_{\calA_-}) \to B$ be an $\calA_-$-weighted stable family of elliptic surfaces over a normal base $B$. Then the construction of
$$
\Proj_{B}\left(\bigoplus_{k} \pi_*\calO_{X_-}(km_0(K_{X_-} + S_- + F_{\calA_-})\right)
$$
commutes with base change and produces a family $\pi_0 : (X_0 \to C, S_0 + F_{\calA_0})$ of $\calA_0$-stable elliptic surfaces and realizes the morphism $\epsilon_{\calA_0}^-$. Applying this construction to $B = \calE_{v, A_-}$ with the universal family yields $\tilde{\epsilon}_{\calA_0}^-$. \end{proof}

\begin{remark} Note that in the above construction, the $\calA_0$-stable family $(X_0 \to C, S_0 + F_{\calA_0})$ associated to the $\calA_-$-stable family $(X_- \to C, S_- + F_{\calA_-})$ has the same base curve $C$. This is because a marked curve is $\calA_0$-stable if and only if it is $\calA_-$-stable where $\calA_0$ is one of the walls for the space of weighted stable curves. That is, the reduction morphism $\calM_{g, \calA_0} \to \calM_{g,\calA_-}$ is a canonical isomorphism. In particular there is a commutative diagram
$$
\xymatrix{\calE_{v,\calA_-} \ar[r]^{\epsilon_{\calA_0}^-} \ar[d] & \calE^m_{v,\calA_0} \ar[d] & \calE_{v,\calA_+} \ar[d] \ar[l]_{\epsilon_{\calA_0}^+} \\ \calM_{g, \calA_-} \ar@{=}[r]& \calM_{g,\calA_0} & \calM_{g,\calA_-} \ar[l]}
$$
showing compatibility with the reduction morphisms on Hassett spaces. 

\end{remark}

Let $\tilde{\epsilon}_{\calA_0}^+ := \widetilde{\rho}_{\calA_+,\calA}$ and $\epsilon_{\calA_0}^+ := \rho_{\calA_+, \calA}$ be the reduction morphisms of the previous section. Then we have a commuting diagram
$$
\xymatrix{
\calU_{v,\calA_-} \ar[rd]_{\tilde{\epsilon}_{\calA_0}^-} \ar[dd] &  & \calU_{v,\calA_+} \ar[ld]^{\tilde{\epsilon}_{\calA_0}^+} \ar[dd] \\
& \calU_{v,\calA_0} \ar[dd] & \\
\calE_{v,\calA_-} \ar[rd]_{\epsilon_{\calA_0}^-} & & \calE_{v,\calA_+} \ar[ld]^{\epsilon_{\calA_0}^+} \\
& \calE_{v,\calA_0} &}. 
$$

We want to compare $\calA_+$-, $\calA_0$-, and $\calA_-$-stable families of elliptic surfaces over the same base $B$. To do this, it is natural to consider the fiber product
$$
\calE_{v,\calA_-} \times_{\calE_{v,\calA_0}} \calE_{v,\calA_+} =: \mathfrak{F}.
$$ 
Pulling back the universal families gives a commutative diagram 
$$
\xymatrix{
\frakU_- \ar[rd] \ar[rdd] & & \frakU_+ \ar[ld] \ar[ldd] \\
& \frakU_0 \ar[d] & \\
& \frakF & }
$$
Then a map $B \to \frakF$ is equivalent to a commutative diagram 
$$
\xymatrix{
X_- \ar[rd] \ar[rdd] & & X_+ \ar[ld] \ar[ldd] \\
& X_0 \ar[d] & \\
& B & }
$$
of compatible families $X_0, X_\pm \to B$ of $\calA_0$-stable (resp. $\calA_\pm$-) stable elliptic surfaces. 

We show that the diagram 
$$
\xymatrix{\frakU_- \ar[rd] & & \frakU_+ \ar[ld] \\ & \frakU_0 & } 
$$
is a \emph{universal flip} in the following sense: 

\begin{prop}\label{prop:flip} For any normal and irreducible base $B$ and map $B \to \frakF$ with generic point mapping to the interior of the moduli space, the induced diagram 
$$
\xymatrix{X_- \ar[rd]_{g_-} & & X_+ \ar[ld]^{g^+} \ar@{-->}[ll]^{\varphi}   \\ & X_0 &}
$$
is a $(K_{X_+} + S_+ + F_{\calA_-})$-flip of the total spaces. 
\end{prop}
\begin{proof}
Let $V \subset B$ be the open locus over which the elliptic surfaces are irreducible and let $Z = B \setminus V$. By assumption $V$ is nonempty and the morphisms $X_- \to X_0$ and $X_+ \to X_0$ are isomorphisms over $V$. Thus the exceptional locus $\mathrm{Exc}(\varphi)$ lies over $Z$. On each fiber over $Z$, the map $X_+ \to X_0$ contracts the section of a pseudoelliptic component and $X_- \to X_0$ contracts a curve in an attaching fiber. Therefore the $\mathrm{Exc}(\varphi)$ is codimension at least $2$. 

We need to show that $-(K_{X_+} + S_+ + F_{\calA_-})$ is $g_+$-ample and $K_{X_-} + \varphi_*(S_+ + F_{\calA_-})$ is $g_+$-ample. Note that $\varphi_*(S_+ + F_{\calA_-}) = S_- + F_{\calA_-}$, where by abuse of notation, we write $F_{\calA}$ for $\calA$-weighted fibers on any of the birational models. Since $g_-$ and $g_+$ are proper, relative ampleness is a fiberwise condition. Thus it suffices to check this after pulling back to a smooth curve $B' \to B$ so without loss of generality, we may take $B$ to be an irreducible smooth curve so that $V = B \setminus \{p\}$. 

In this case $X_+ \to X_0$ is the contraction of the section in a component of the central fiber $(X_+)_p$. It is then proven in \cite[Theorem 7.1.2]{ln} that $X_+ \to X_0$ is a flipping contraction induced by $K_{X_+} + S_+ + F_{\calA_-}$ with log flip $X_- \to X_0$. 
 \end{proof}

\begin{cor}\label{cor:decrease} The morphism $\epsilon_{\calA_0}^-$ is an isomorphism. \end{cor} 

\begin{proof} $\epsilon_{\calA_0}^-$ is a proper bijection and our moduli spaces are normal. \end{proof}

\begin{remark}\label{rem:infinitessimal} Since we normalize the moduli spaces, we make no claims about the infinitessimal structure of $\epsilon_{\calA_0}^-$. Indeed the deformation theories of $\calA_0$ and $\calA_-$ broken elliptic surfaces may be very different. \end{remark} 

\subsection{The ``wall'' at $a = 1$}

In this section we discuss an analogous behavior to the flipping morphism $\epsilon_{\calA_0}^{-} : \calE^m_{v,A_-} \to \calE^m_{v,\calA_0}$ that occurs in the limit as a coefficient $a \to 1$. 

Indeed if we take $X_- = X' \cup Z$ as in the proof of of Theorem \ref{thm:increase} so that $X'$ is an elliptic component, $Z$ is a pseudoelliptic component of type $I$ attached to $X'$ along an intermediate fiber $G \cup A$, then we saw that the morphism $\epsilon_{\calA_0}^{-}$ contracts the fiber component $A$. Locally on $X'$ around the fiber $G \cup A$, this contraction of $A$ is the transition from an intermediate to a twisted fiber Section \ref{sec:local}. In both cases, this contraction occurs when the intermediate fiber components $G$ and $E$ are both marked with coefficient $a = 1$ and in both cases, this induces a morphism on moduli spaces:

\begin{prop} \label{prop:1-e} Let $\calB = (b_1, \ldots, b_r)$ and fix $j$ such that $b_j = 1$. Let $\calA < \calB$ be a weight vector with $a_i = b_i$ for $i \neq j$ and $a_j = 1 - \epsilon$ where $0 < \epsilon \ll 1$. Then there are morphisms $\theta_j : \calE_{v,\calA} \to \calE_{v,\calB}$ and $\tilde{\theta}_j : \calU_{v,\calA} \to \calU_{v,\calB}$ making the following diagram commute:
$$
\xymatrix{\calU_{v,\calA} \ar[r]^{\tilde{\theta}_j} \ar[d] & \calU_{v,\calB} \ar[d] \\ \calE_{v,\calA} \ar[r]^{\theta_j} & \calE_{v,\calB}}
$$
\end{prop}

\begin{proof} Since we are taking $\epsilon \ll 1$, then $K_X + S + F_\calB$ is a nef $\Q$-Cartier divisor on a $\calA$-stable elliptic surface $(X \to C, S + F_\calA)$. Therefore $K_X + S + F_\calB$ is semiample and by the results of Section \ref{sec:local}, there are two possibilities for the Iitaka map $\varphi:= \varphi_{m(K_X + S + F_\calB)} : X \to X'$ depending on the fiber $F_j$ whose coefficient is changing: 
\begin{itemize} 
\item the fiber $F_j$ is a smooth or stable fiber (type $\mathrm{I}_n$) so that the birational model does \emph{not} change when $b_j = 1$ and $\varphi$ is the identity;
\item the fiber $F_j$ is not type $\mathrm{I}_n$ so that it is an intermediate fiber given by a union $A \cup E$ of a reduced component $A$ and a nonreduced component $E$. The Iitaka map $\varphi$ is the contraction of $A$ to produce a twisted fiber. 
\end{itemize} 

In the first case there is nothing to prove. In the second, 
$$
R^1\varphi_*(\calO_X(m(K_X + S + F_\calA))) = 0
$$
by Proposition $\ref{prop:gr}$ and $\varphi_*(\calO_X(m(K_X + S + F_\calA))) = \calO_{X'}(m(K_{X'} + \varphi_*(S + F_\calA)))$ by Proposition \ref{logcanonical}. It follows that $H^1(X,\calO_X(m(K_X + S + F_\calA))) = 0$ by the Leray spectral sequence. 

Now if $\pi : (X \to C, S + F_\calA) \to B$ is a family of $\calA$-stable elliptic surfaces, then as in the construction of reduction morphisms and flipping morphisms, 
$$
\Proj_B\left(\bigoplus_{k} \pi_*\calO_X(km_0(K_X + S + F_\calB))\right)
$$
gives a family $\calB$-stable elliptic surfaces over $B$. This construction is compatible with base change by the above vanishing and induces the required morphism $\theta_j$.

The morphism $\tilde{\theta}_j$ is induced by applying the above to the universal family $\calU_{v,\calA} \to \calE_{v,\calA}$.  \end{proof} 

\begin{cor}\label{isomorphism} In the situation above, the morphism $\theta_j$ is inverse to the reduction morphism $\rho_{\calB,\calA}$. In particular, $\calE_{v,\calA} \cong \calE_{v,\calB}$. \end{cor}

\begin{remark} As in Remark \ref{rem:infinitessimal}, the validity of the above corollary hinges on the fact that we are defining our moduli spaces to be the normalizations of the appropriate pseudofunctors. In general the deformation theories of $\calA$-stable and $\calB$-stable models might differ depending on the choice of functor of stable pairs and we can only hope for $\theta_j$ to be some type of partial normalization. \end{remark}

\appendix
\section{Normalizations of algebraic stacks}
\label{appendix}

In this appendix, we justify the fact that we only work with normal base schemes throughout the paper. Specifically, the goal is to prove that in certain situations, the normalization of an algebraic stack is uniquely determined by its values on normal base schemes (Proposition \ref{uniquenorm}) and that a morphism between normalizations of algebraic stacks can be constructed by specifying it on the category of normal schemes (Proposition \ref{normfunctor}). This material is probably well known but we include it here for lack of a suitable reference.  

If $X$ is a locally Noetherian scheme, the normalization $\nu : X^{\nu} \to X$ is defined as the normalization of $X$ in its total ring of fractions. We denote by $|X|$ (resp. $|\calX|$) the underlying topological space of points of a scheme (resp. algebraic stack). We begin with some facts about normalizations of schemes. 

\begin{lemma} \label{normfact} \cite[Tag 035Q]{stacks-project} Let $X$ be a locally Noetherian scheme;

\begin{enumerate}

\item the normalization $X^{\nu} \to X$ is integral, surjective and induces a bijection on irreducible components;

\item for any normal scheme $Z$ and morphism $Z \to X$ such that each irreducible component of $Z$ dominates an irreducible component of $X$, there exists a unique factorization $Z \to X^{\nu} \to X$. 

\end{enumerate}
\end{lemma}

\begin{lemma}\label{smoothnorm}\cite[Tag 07TD]{stacks-project} Let $X \to Y$ be a smooth morphism of locally Noetherian schemes. Let $Y^{\nu} \to Y$ be the normalization of $Y$. Then $X \times_Y Y^{\nu} \to X$ is the normalization of $X$. \end{lemma} 

This motivates the following definitions: 

\begin{definition} Let $\calX$ be a locally Noetherian algebraic stack. We say that $\calX$ is \textbf{normal} if there is a smooth surjection $U \to \calX$ where $U$ is a normal scheme.  A \textbf{normalization} of $\calX$ is a representable morphism 
$$
\nu : \calX^\nu \to \calX
$$
from an algebraic stack $\calX^{\nu}$ such that for any scheme $U$ and any smooth morphism $U \to \calX$, the pullback $\calX^{\nu} \times_{\calX} U \to U$ is the normalization of $U$. 
\end{definition}

\begin{lemma} Let $\calX$ be a locally Noetherian algebraic stack. Then a normalization $\nu : \calX^{\nu} \to \calX$ exists and it is unique up to unique isomorphism. \end{lemma} 

\begin{proof} The proof closely follows \cite[Tag 07U4]{stacks-project} which proves the claim for algebraic spaces. Indeed let $R \rightrightarrows U$ be a smooth groupoid presentation for $\calX$. Then by Lemma \ref{smoothnorm} one sees that the pullback of $R$ to $U^{\nu}$ under both morphisms is isomorphic to $R^{\nu}$. One can then check as in \textit{loc. cit.} that $R^{\nu} \rightrightarrows U^{\nu}$ is a smooth groupoid and define $\calX^{\nu} = [U^{\nu}/R^{\nu}]$ with morphism to $\calX$ induced by $U^{\nu} \to U$ and $R^{\nu} \to R$. 

Normality is local on the base in the smooth topology \cite[Tag 034F]{stacks-project} so that for any scheme $T$ and smooth morphism $T \to \calX$, we can check normality of $T \times_{\calX} \calX^{\nu}$ by pulling back to the smooth cover $U \to \calX$. Here the result follows from Lemma \ref{smoothnorm}. Finally uniqueness is clear from the construction.  \end{proof} 

\begin{lemma}\label{normstack} Let $\calX$ be a locally Noetherian algebraic stack, then;

\begin{enumerate}

\item $\calX^{\nu}$ is normal;
\item $\calX^{\nu} \to \calX$ is an integral surjection that induces a bijection on irreducible components;
\item for any normal algebraic stack $\calZ$ and a morphism $\calZ \to \calX$ such that every irreducible component of $\calZ$ dominates an irreducible component of $\calX$, there exists a unique factorization $\calZ \to \calX^{\nu} \to \calX$. 

\end{enumerate} \end{lemma} 

\begin{proof} The proof follows the analagous result \cite[Tag 0BB4]{stacks-project} for algebraic spaces. $(1)$ is clear and $(2)$ follows from Lemma \ref{normfact} and descent. 

For $(3)$ let $U \to \calX$ be a smooth surjection and $R = U \times_{\calX} U \rightrightarrows U$. Pulling back to $\calZ$ gives a smooth morphism $\calY:= U \times_{\calX} \calZ \to \calZ$. Let $U' \to \calY$ be a smooth cover of $\calY$ by a scheme and $U'$. The composition $U' \to \calZ$ is a smooth cover with groupoid presentation $R' : U' \times_{\calZ} U' \rightrightarrows U'$ and a commutative square
$$
\xymatrix{R' \ar@<-.5ex>[d] \ar@<.5ex>[d] \ar[r] & R \ar@<-.5ex>[d] \ar@<.5ex>[d] \\ U' \ar[r] & U}.
$$

The conditions on $\calZ \to \calX$ imply that we can apply Lemma \ref{normfact} to obtain unique factorizations $R' \to R^{\nu}$ and $U' \to U^{\nu}$. By uniqueness, these morphisms are compatible with the groupoid data so that we get a unique factorization $\calZ \to \calX^{\nu}$ by descent.\end{proof}

Now we are ready for the main results of this appendix. 

\begin{prop}\label{normfunctor} Let $\calX$ and $\calY$ be locally Noetherian algebraic stacks. Suppose that for each normal scheme $T$, there exist functors
$$
f_T : \calX(T) \to \calY(T) 
$$
compatible with base change and such that the induced morphism on points $|f| : |\calX| \to |\calY|$ is dominant on irreducible components. Then $f_T$ induces a unique representable morphism
$$
f^{\nu} : \calX^{\nu} \to \calY^{\nu}.
$$
\end{prop}

\begin{proof} Let $U \to \calX$ be a smooth surjection from a scheme $U$ and let $U^{\nu} \to U$ be the normalization. Then $U^{\nu} \to \calX$ is an integral surjection that induces a bijection on irreducible components by Lemma \ref{normstack} (2). Let $\xi \in \calX(U^{\nu})$ be the object inducing this morphism. Then we have an object $f_T(\xi) \in \calY(U^{\nu})$ inducing a morphism $U^{\nu} \to \calY$. By assumption this is compatible with the pullbacks to $R^{\nu} = U^{\nu} \times_{\calX^{\nu}} \times U^{\nu}$ and thus induces a morphism $g : \calX^{\nu} \to \calY$. 

The map $|g| : |\calX| \to |\calY|$ factors as
$$
\xymatrix{|\calX^\nu| \ar[r]^{|\nu|} \ar[rd]_{|g|} & |\calX| \ar[d]^{|f|} \\ & |\calY|}.
$$
By Lemma \ref{normstack} (2) and the assumptions on $|f|$, $|g|$ is dominant on irreducible components. Therefore there is a unique factorization $f^{\nu} : \calX^{\nu} \to \calY^{\nu}$ by Lemma \ref{normstack} (3).  \end{proof}

\begin{prop}\label{uniquenorm} Let $\calX$ and $\calY$ be separated locally Noetherian algebraic stacks. Suppose that for each normal scheme $T$, there is an isomorphism $f_T : \calX(T) \cong \calY(T)$ compatible with base change. Then there is an isomorphism $f: \calX^{\nu} \to \calY^{\nu}$. \end{prop}

\begin{proof} First let $\calT$ be a normal algebraic stack. Then there is a smooth cover $U \to \calT$ where $U$ is normal giving a groupoid presentation $R \rightrightarrows U$ of $\calT$. Since normality is local in the smooth topology \cite[Tag 034F]{stacks-project}, $R$ is normal and we have equivalences $\calX(R) \cong \calY(R)$ and $\calX(U) \cong \calY(U)$ compatible with base change by the two morphisms $R \rightrightarrows U$. By descent, this induces an equivalence $f_\calT : \calX(\calT) \cong \calY(\calT)$ compatible with base change by a normal algebraic stack. Denote the inverse by $g_{\calT}$.  

By Proposition \ref{normfunctor} there exist morphisms $f : \calX^{\nu} \to \calY^{\nu}$ and $g : \calY^{\nu} \to \calX^{\nu}$ induced by $f_T$ and its inverse. The map $\calX^{\nu} \to \calX$ is induced by an object $\xi \in \calX(\calX^{\nu})$ and under the equivalence described in the preceding paragraph, $f_{\calX^{\nu}}(\xi) \in \calY(\calX^{\nu})$ corresponds to the composition $\calX^{\nu} \to \calY^{\nu} \to \calY$. Similarly, if $\xi' \in \calY(\calY^{\nu})$ is the object inducing the normalization $\calY^{\nu} \to \calY$, then $g_{\calY^{\nu}}(\xi') \in \calX(\calY^{\nu})$ corresponds to the composition $\calY^{\nu} \to \calX^{\nu} \to \calX$. 

By compatibility of the equivalences with pullbacks, we get that $g^*\xi = g_{\calY^{\nu}}(\xi')$ so that $$\xi' = f_{\calY^{\nu}}g^*\xi = g^*f_{\calX^{\nu}} \xi \in \calY(\calY^{\nu})$$. But the latter is the object corresponding to the composition 
$$
\calY^{\nu} \to \calX^{\nu} \to \calY^{\nu} \to \calY.
$$
Therefore $\nu \circ f \circ g = \nu$, i.e. the morphism $fg : \calY^{\nu} \to \calY^{\nu}$ commutes with the normalization $\calY^{\nu} \to \calY$. 

Since the normalization factors uniquely through $\calY^{red}$, we may suppose that $\calY$ is reduced. Then $\nu$ is an isomorphism over a dense open subset of each irreducible component of $\calY$. Therefore $fg$ must agree with the identity over this dense open subset so $fg = \mathrm{id}_{\calY^{\nu}}$, since $\calY^{\nu}$ is separated. Applying the same argument to $\calX^{\nu}$ yields that $gf = \mathrm{id}_{\calX^{\nu}}$. \end{proof}

\begin{remark} Note that $\calX^{\nu}(T)$ is \emph{not} necessarily equal to $\calX(T)$ for $T$ normal even though $\calX^{\nu}$ is uniquely determined by the values of $\calX(T)$ for $T$ normal. Indeed this fails even for schemes. For example the inclusion of the node of nodal curve has multiple lifts to the normalization. It is an interesting question to determine a functorial way to define the normalization of $\calX$ directly as a category fibered in groupoids over schemes without knowing a priori that $\calX$ is algebraic. 
\end{remark}

\section{Small contractions of $1$-parameter families of elliptic surfaces}
\label{sec:appendix}
The goal of this appendix is to control the birational transformations one has to deal with to find the stable limit of a family of  $\calA$-weighted broken elliptic surfaces over a DVR (see Theorem \ref{Teo:appendix:giovanni}). In particular, we want to show that we can assume the only flip that happens is the flip present in the work of La Nave (see Section \ref{sec:WII}).

We work with a family of $\calA$-broken elliptic surfaces $(f:X \to C,S+F_\calA)$ over a DVR $R$. Assume that the generic fiber of
$X \to \Spec(R)$ is normal. Let $\calB \le \calA$ such that for every $\calB <\calB' \le \calA$, the divisor $K_X+S+F_{\calB'}$ is $\mathbb{Q}$-Cartier and ample, but $K_X+S+F_{\calB}$ is only nef.
To reach the desired conclusion, it suffices to show that when taking the stable model of $(X,S+F_\calB)$, the only codimension two exceptional curves that log abundance contracts are section components (see Theorem \ref{Teo:appendix:giovanni}).

\subsection{Intersection pairings on an elliptic fibration}
We recall from Definition \ref{def:twisted} that an intermediate fiber is the nodal union of a rational component $A$, and a (possibly non-reduced) arithmetic genus one component $E$. Furthermore, the section meets the fiber along the smooth locus of $A$. The rational map that replaces an intermediate fiber with a twisted one, is a regular morphism in a neighbourhood of each intermediate fiber, and such a morphism contracts the rational component. For what follows, we will use the notation $A_\calA$ to denote $\sum a_i A_i$, where $A_i$ are the rational components of an intermediate fiber. Moreover, we will use the notation $E_\calB = \sum  b_i E_i$ to denote the sum of twisted fibers and Weierstrass fibers. We note that the divisor $E_\calB \subset X$ contains \emph{both} the twisted fibers, as well as the twisted components of the intermediate fibers (i.e. the $E$ components). 

First we need to understand how the intersection pairing works on the irreducible components of the special fiber. We can divide those into two types: either irreducible elliptic surfaces, or irreducible pseudoelliptic surfaces. An irreducible divisor on an elliptic (resp. pseudoelliptic) surface is either supported on a fiber component (resp. pseudofiber component), or it is a multisection (resp. pseudomultisection). Thus we need to understand what happens if a negative curve is either of those types of divisors.

We start with a lemma that slightly generalizes \cite[Lemma II.5.6]{mir3}.

\begin{lemma}\label{lemma:Ssquare:negative}Let $(f: X \to C,S+A_{\calA}+E_{\calB})$ be an irreducible broken (slc) elliptic surface. Then:
\begin{enumerate}
\item $S^2 \le 0$, and if $S^2=0$ then the only singular fibers of $f: X \to C$ are twisted fibers. 
\item If in addition to $S^2 =0$, the divisor $K_X+S+A_\calA+E_\calB$ is nef and $(K_X+S+A_\calA+E_\calB).S=0$, then $K_X+S+A_\calA+E_\calB$ is not big. \end{enumerate}
 \end{lemma}
\begin{proof}
(1): First observe that, since the computation is local around $S$, we can replace all the Weierstrass fibers which are cusps, with intermediate fibers. This will not affect $S^2$.
 
 Now observe that we can replace all the intermediate fibers with twisted fibers.
 Indeed, if $p:X\to Y$ is the contraction of the rational component of an intermediate fiber, $A_1$, and $S':=p_*(S)$, then
 $$(S')^2 = (p_*S).S' = S.p^*(S') = S^2 + \alpha\cdot S.A_1,$$  
 where $p^*(S')=S+\alpha A_1$, as $A_1$ is irreducible.
 However, $0=p^*(S').A_1=S.A_1+\alpha \cdot A_1^2$ and $A_1^2 < 0$ since $A_1$ is an exceptional curve.
 Therefore $\alpha >0$, and $S^2 < (S')^2$, and so we can replace the singular fibers with twisted fibers and obtain a new surface $Y$ such that if $S'$ is the proper transform of $S$,
 then $S^2 \le (S')^2$ with equality only if $X=Y$. 
 
Since now the only singular fibers are twisted, the surface $f: Y \to C$ is the coarse space of a twisted elliptic surface $\calY \to \calC$ obtained using the construction of twisted stable maps (see \cite{tsm}). We abuse notation and call the section of $Y$ by $S$, so that we have an elliptic fibration $(f: Y \to C, S)$. By \cite[Proposition 5.3]{tsm} and \cite[Theorem 6.1]{calculations} we have the following $$2g(C) - 2 =(K_Y + S).S \ge \deg(\LL) + 2g(C) - 2 + S^2,$$ and so $S^2 \le -\deg(\LL) \leq 0$, since $\LL$ always has non-negative degree. 

(2): From the proof of (1), we see that to have $S^2=0$, we require that the surface $X$ has only twisted fibers (so in particular $A_\calA = 0$). In this case all fibers are irreducible, and since the generic fiber of $f: X \to C$ has trivial canonical divisor, the canonical divisor of the surface $K_X$ is supported on fibers. 

Suppose that $(K_X + S +E_\calB).S = 0$ (recall that all fibers are twisted so that $A_\calA = 0$). Then we see that $$(K_X + S + E_\calB)^2 = K_X^2 + K_X.S + 2K_X.E_\calB +  E_\calB^2 + E_\calB.S + (K_X + S + E_\calB).S.$$ Recall that $(K_X + S + E_\calB).S = 0$ by hypothesis. Furthermore, $E_\calB^2 = 0$ as $E_\calB$ is supported on fibers, and the same for $K_X^2$ and $K_X.E_\calB$. Finally, $K_X.S + E_\calB.S = (K_X + S + E_\calB).S - S^2 = 0 - 0 = 0$. \end{proof}

In order to conclude that that no other flips occur, we need to check that whenever we contract a multisection (resp. pseudomultisection) we have to contract the whole surface component.

 \begin{lemma}\label{lemma:first:contract:the:section}
Let $(f: X \to C, sS + A_\calA + E_\calB)$ be as above, with $s>0$, and let $M \neq S$ be an irreducible multisection of $f$. Assume that $K_X+S+A_\calA + E_\calB$ is $f$-nef. Then
\begin{enumerate} \item If $(K_X+sS+A_\calA+E_\calB).M \le 0$, then $(K_X+sS+A_\calA+E_\calB).S \le 0$. \item Moreover, if both are 0 and $K_X+S+A_\calA + E_\calB$ is nef, the divisor $K_X+sS+A_\calA+E_\calB$ is not big. \end{enumerate}
\end{lemma}

\begin{remark}
To deal with the pseudoelliptic case, it is convenient not to assume that $\textrm{Coeff}(S)=1$. This is because, if $p:(X,S+F_\calA) \to (Y,F'_\calA:=p_*(F_\calA))$ is the contraction
of a section, then $p^*(K_Y+F'_\calA)=K_X+F_\calA+\alpha S$, and $\alpha$ might not (and in general will not) be one. 
\end{remark}

\begin{proof} 
Let $Y$ be the surface obtained by contracting the rational component of an intermediate fiber $A_1$, and let $p:X \to Y$ the contraction morphism.
 We will denote with $S'$ (resp. $E'_\calB$, $A'_\calA$ and $M'$) the push-forward $p_*(S)$ (resp. $p_*(E_\calB)$, $p_*(A_\calA)$ and $p_*(M)$).  The proof proceeds in two steps.
 \begin{Step_2} We wish to show that
  $(K_X+sS+A_\calA+E_\calB).S \le (K_{Y}+sS'+A'_\calA+E'_\calB).S'$
 \end{Step_2}
The computation can be performed locally around $p(A_1)$, as away from here $p$ is an isomorphism, therefore it suffices to show this for the contraction of a single fiber component $A_1$, and thus $A'_\calA=0$. We will drop the subscript $\calA$.

If $p^*(S')=S+\alpha A$, we have $0=p^*(S').A = (S+\alpha A).A=S.A+\alpha \cdot A^2$ since $A$ is irreducible, thus
$$p^*(S)=S+\alpha A \text{ and }\alpha=-\frac{S.A}{A^2}$$
If $\beta $ is such that $K_X=p^*K_Y+\beta A$, as before we have $K_X.A = (p^*(K_Y) + \beta A).A = \beta \cdot A^2$, so 
$$K_X=p^*(K_Y)+\beta A \text{ and }\beta=\frac{K_X.A}{A^2}$$
Observe now that the following equalities hold:
\begin{enumerate}
\item: $(S')^2=(p_*S).S'=S.p^*(S')=(S+\alpha A).S=S^2-\frac{S.A}{A^2}(S.A)$;
\item: $S'.E_\calB'=(S+\alpha A).E_\calB=-\frac{(S.A)(A.E_\calB)}{A^2}$;
\item: $K_X.S=K_Y.S'+\beta (S.A)=K_Y.S'+\frac{K_X.A}{A^2}(S.A)$
\end{enumerate}
Therefore using (3) we have
$$(K_X+sS+A+E_\calB).S =K_Y.S'+\frac{K_X.A}{A^2}(S.A)+sS^2+A.S$$
and using (1) and (2), since $A'=0$,
$$ (K_Y+sS'+A'_\calA+E'_\calB)S'=K_Y.S'+s(S'^2)+S'.E_\calB' =
 K_{Y}.S'+sS^2+s\frac{(S.A)^2}{-(A)^2}+\frac{(S.A)(A.E_\calB)}{-(A)^2}$$
Thus we need to show that $$K_Y.S'+\frac{K_X.A}{A^2}(S.A)+sS^2+A.S  \le
 K_{Y}.S'+sS^2+s\frac{(S.A)^2}{-(A)^2}+\frac{(S.A)(A.E_\calB)}{-(A)^2}$$
Since $A$ is an exceptional curve $A^2 <0$ and $S.A>0$. Thus we can multiply both sides by
$\frac{-(A)^2}{S.A}$ and we need to show that $$-K_X.A-A^2 \le s(S.A)+E_\calB.A,$$
but this holds since $K_X+S+A+E_\calB$ is $f$-nef by hypothesis.

  \begin{Step_2} We now wish to show that
  $(K_{Y}+A'_\calA+E'_\calB).M' \le (K_X+A_\calA+E_\calB).M$ 
  \end{Step_2}
As before we assume that $A_\calA=A_1$ and $A'_\calA=0$. We also drop the subscript $\calA$.

Proceeding as above we have the following equality, since $A$ is irreducible:
$$(K_{Y}+E_\calB').p_*(M)=p^*(K_{Y}+E_\calB').M=(K_X+\gamma A+E_\calB).M$$
To show the desired result it is enough to show that $\gamma\le1$ (recall that $A$ is an effective $\mathbb{Q}$-divisor).
Observe that $K_Y$ is supported on fiber components since the generic fiber of $f$ has relative canonical divisor 0. Thus, let $E$ be the genus one component of the intermediate fiber containing $A$, and let $E':=p_*(E)$.
Since $E'$ is supported on a fiber, $(K_Y+E'_\calB).E'=0$. But then
$$0=(K_X+\gamma A+E_\calB).E \le (K_X+sS+A+E_\calB).E$$ since $K_X+sS+A+E_\calB$ is $f$-nef. Thus
$\gamma(A . E) \le A.E$ since $S\cap E= \emptyset$, and $\gamma \le 1$

$\text{ }$
\\
\begin{bf}Conclusion:\end{bf}

Now, let $Y$ be the surface obtained from $X$ contracting $A_\calA$, let $p:X \to Y$ be the contraction morphism and let $S':=p_*(S)$, $M':=p_*(M)$ and $E_\calB':=p_*(E_\calB)$.
We want to motivate the following inequalities:
 \begin{align*}
     (K_X+sS+A_\calA+E_\calB).M &\ge (K_X+A_\calA+E_\calB).M  
     \ge (K_Y+E_\calB').M'\\ & = m(K_Y+E_\calB').S' \ge m(K_Y+sS'+E_\calB').S'\ge \\
     & \ge m(K_X+sS+A_\calA+E_\calB).S
 \end{align*}
 The first inequality follows since $M \neq S$ and $M$ is irreducible.
 The second one follows from Step 2.
For the equality, since the generic fiber of $g:Y \to C$ is a stable curve of genus one, the canonical of the generic fiber is trivial, and so $K_Y$ is supported on some fiber components. But all the fibers are irreducible, so there is a $\mathbb{Q}$-divisor $D$ on $C$
such that $K_{Y}+E'_\calB$ is numerically equivalent to $g^*(D)$.
Thus $$(K_{Y}+E'_\calB).M'=\deg(K_{Y}+E'_\calB)_{|M}=\deg(g^*(D))_{|M}=\deg(g_{|M'})\deg(D)= m(K_{Y}+E'_\calB).S',$$ where $m$ is the degree of $g_{|M'}:M' \to C$.

The third inequality follows from Lemma \ref{lemma:Ssquare:negative}. Finally Step 1 gives the desired fourth one.

$\text{ }$

If both $(K_X+sS+A_\calA+E_\calB).M=(K_X+sS+A_\calA+E_\calB).S=0$, all the inequalities above are equalities. In
particular the last one is an equality, and from the proof of Step 1, we see that log-abundance contracts $A_\calA$,
and the stable model of
$(X,sS+A_\calA+E_\calB)$
is the same as that of $(Y,sS'+E_\calB')$. Therefore to 
show that $K_X+sS+A_\calA+E_\calB$ is not big, it is enough to show that
$K_Y+sS'+E_\calB'$ is not big. It is easy to see that $p^*(K_Y+sS'+E_\calB')=K_X+sS+A_\calA+E_\calB$, so $K_Y+sS'+E_\calB'$ is nef. Moreover
since $s>0$, from the third  inequality we get $(S')^2=0$. Then 
 Lemma \ref{lemma:Ssquare:negative}
 applies.
\end{proof}

We now discuss a similar situation in the case in which $X$ is a pseudoelliptic. The result we need will be Corollary \ref{Cor:dont:contract:multi:on:psudo}, which is a consequence of Lemma
\ref{lemma:first:contract:the:section} and the following Lemma.

\begin{lemma}\label{lemma:the:coef:of:S:pulling:back:the:lc:on:a:psudo:positive}
Let $(X,D)$ be a psudoelliptic surface.
Assume that $K_X+D$ is non-negative on each psudo-fiber component, and let $p:X' \to X$ be the morphism which contracts the section $S$. Let $\alpha$ be such that
$p^*(K_X+D)=K_{X'}+D'+\alpha S$ where $D'=p_*^{-1}(D)$. Then $\alpha \ge 0$, and if $\alpha=0$ and $K_X+D$ is nef, then it is not big. 
\end{lemma}
\begin{proof}
 Let $F$ be an irreducible pseudofiber of $X$, let $f:X' \to C$ be the associated elliptic fibration, and let $F':=p_*^{-1}(F)$. Up to replacing $F$,
 we can assume that $F'$ is not a multiple fiber. By the non-negativity assumption of $K_X+D$,
 $$0 \le (K_X+D).F=(p^*(K_{X}+D)).F'=(K_{X'}+D'+\alpha S).F'.$$
 But now $D'$ is supported on some fiber components, and so is $K_{X'}$ since the generic fiber of $f$ is a stable curve of genus one. Then
 $(K_{X'}+D'+\alpha S).F' =\alpha(S'.F')=\alpha$ since $F'$ is an irreducible fiber, which is not a multiple fiber.
 
 If $\alpha=0$, to take the the log canonical model of $(X',D')$ we have to contract the fibers of $f$, and
 so $K_{X'}+D'=p^*(K_X+D)$ is not big. As a result $K_X+D$ is not big as well. 
\end{proof}

\begin{cor}\label{Cor:dont:contract:multi:on:psudo}
Let $(X,D)$ be a pseudoelliptic surface, let $M \subseteq X$ be an irreducible pseudomultisection, and assume that $K_X+D$ is nef.
If $(K_X+D).M=0$ then $K_X+D$ is not big.
\end{cor}
\begin{proof}Let $f:X' \to C$ be the elliptic fibration with section $S$ such that $X$ is obtained by contracting $S$ on $X'$.
 Let $p:X' \to X$ be the contraction morphism, and let $D':=p^{-1}_*(D)$. Finally let $M':=p_*^{-1}(M)$.
 
 From Lemma \ref{lemma:the:coef:of:S:pulling:back:the:lc:on:a:psudo:positive} there is an $\alpha \ge 0$ such that
 $p^*(K_X+D)=K_{X'}+D'+\alpha S$. If $\alpha=0$ the result follows from Lemma \ref{lemma:the:coef:of:S:pulling:back:the:lc:on:a:psudo:positive},
 otherwise assume $\alpha > 0$. Then we have
 $$0 =(K_X+D).M=(K_{X'}+D'+\alpha S).M'.$$ But from Lemma \ref{lemma:first:contract:the:section} the divisor $K_{X'}+D'+\alpha S$ is not big, 
 and so $K_{X}+D$ is not big.
\end{proof}

We want to show that the only flip that happens is the one of La Nave. In particular, with the notation of the beginning of the Appendix, we need to make sure that if on a pseudoelliptic component $Z$ the intermediate component of an intermediate pesudofiber is not $K_X+S+F_\calB$-positive, then to take the stable model of $(X,S+F_\calB)$ we
contract $Z$. This is the content of the following Lemma.
Here is where we use the assumption that the twisted component of an intermediate
fiber has weight one.

\begin{lemma}\label{lemma:dont:contract:A:cp:on:a:psudo}
Let $(X,D+G)$ be a pseudoelliptic pair
(Definition \ref{def:pseudo}), where we write the divisor as a sum of two effective $\mathbb{Q}$-divisors $D$ and $G$. Assume that
$\operatorname{Supp}(G)$ is irreducible, and $\operatorname{Supp}(D)$ has a component $E_0$ such that $\textrm{Coeff}(E_0,D) = 1$.
 Let $F \subset X$ be the rational component of an intermediate pseudofiber, let $E \subset D$ be the associated twisted component, and suppose that $\mathrm{Coeff}(E,D) = 1$.
Assume that for a certain $0\le \beta <1$, the divisor $K_X+D+\beta G$ is nef, but for any other $\beta < \beta' \le 1$ the divisor $K_X+D+\beta' G$ is ample.

If $(K_X+D+\beta G).F=0$ then $K_X+D+\beta G$ is not big.
\end{lemma}
\begin{proof}
 We use the same notation as in Corollary \ref{Cor:dont:contract:multi:on:psudo}: let $G':=p_*^{-1}(G)$, $F':=p_*^{-1}(F)$, $D':=p_*^{-1}(D)$, $E_0':=p_*^{-1}(E_0)$
 and $E':=p_*^{-1}(E)$ .
 Let $\alpha$ be such that $p^*(K_X+D+\beta G)=K_{X'}+D'+\beta G'+\alpha S$. From \cite[Corollary 3.53]{km}, the pairs $(X,D+\beta G)$ and $(X',D'+\beta G'+\alpha S)$ have the same stable model. So if $\alpha=0$ we have the desired result from Lemma
 \ref{lemma:the:coef:of:S:pulling:back:the:lc:on:a:psudo:positive}.
 Thus we can assume $\alpha > 0$.
 
 From Lemma \ref{lemma:Ssquare:negative}
 we can also assume $S^2<0$. Then we have $p^*(G)=G'+aS$ for some $a$, and $0=G.S+aS^2$. Therefore since $G.S >0$ we have $a > 0$. Since the pair $(X,D+G)$ is lc, there is a $b$ such that
 $p^*(K_X+D+G)=K_{X'}+D'+G'+bS$ with $b \le 1$. But then 
 \begin{align*}
p^*(K_X+D+\beta G)&=p^*(K_X+D+G)+p^*((\beta-1)G)\\&=K_{X'}+D'+\beta G'+bS+a(\beta-1)S \\&=K_{X'}+D'+\beta G'+\alpha S   
 \end{align*}
 So $\alpha=b+a(\beta -1)$, but by hypothesis $\beta <1$ and $a>0$, so $a(\beta -1)<0$ and so $b+a(\beta -1)<b$. Since $b \le 1$ we conclude that $ \alpha < 1$.  Combining it with $\alpha > 0$, we get $0 < \alpha <1$.
 
 We proceed by contradiction. If $K_X+D+\beta G$ was big and $(K_X+D+\beta G).F=0$, the stable model of $(X,D+\beta G)$ is a surface
 obtained from
 $X$ contracting $F$ and maybe some other curve. But the stable model of $(X,D+\beta G)$ is the same as the stable model of $(X',D'+\alpha S+\beta G')$. Therefore the stable model of $(X',D'+\alpha S+\beta G')$ would also be a surface, and it would be obtained by
 contracting $S, F'$ and possibly some other curve. Let $S,F',\{B_i\}_{i=1}^n$ be the curves we need to contract.
 From Lemma \ref{lemma:first:contract:the:section}, the curves $B_i$ are fiber components. Furthermore,
 we do not contract a whole
 fiber by the bigness assumption. Then, since all the fibers have at most two irreducible components, $B_i^2<0$, $(F')^2<0$, $F'.B_j =0$ and
 $B_i.B_j=0$ for $i \neq j$.
 
 Let  $H_1,H_2$ be two irreducible and reduced fibers
 of $X'\to C$ which are not cusps, and let $Y$ be the stable model of $$(X',D'+H_1+H_2+\alpha S+\beta G').$$
 Observe that the contraction morphism $q:X' \to Y$ contracts just $F'$ and $\{B_i\}_{i=1}^n$ and does not contract $S$. Let $D_Y:=q_*D'$, $S_Y:=q_*S'$, $E_Y:=q_*(E')$,
 $(E_0)_Y:=q_*(E_0')$
 and $G_Y:=q_*G'$
 But then: $$q^*(K_Y+D_Y+\alpha S_Y+\beta G_Y)=K_{X'}+D'+\alpha S+\beta G'+\sum c_i B_i +c_F F'$$ for certain coefficients $c_i$ and $c_F$. Observe then that
 $$q^*(K_Y+D_Y+\alpha S_Y+\beta G_Y).B_i=(K_{X'}+D'+\alpha S+\beta G').B_i=0$$
  for every $i$ and similarly for $F'$, since $B_i$ and $F'$ are contracted by log abundance. Then the coefficients $c_F$ and $c_i$ are 0.

 Therefore $K_Y+D_Y+\alpha S_Y+\beta G_Y$ is nef, and from \cite[Corollary 3.53]{km}
 the stable model of $(X',D'+\alpha S+\beta G')$ is the same as the stable model of $(Y,D_Y+\alpha S_Y+ \beta G_Y)$.
 So in particular
 to take the stable model of  $(Y,D_Y+\alpha S_Y+ \beta G_Y)$ we have to contract $S_Y$. Since $K_Y+D_Y+\alpha S_Y+\beta G_Y$ is already
 nef,
 to take the stable model of $(Y,D_Y+\alpha S_Y+ \beta G_Y)$ we use log abundance. 

 But by hypothesis, $E_Y$ has coefficient one in $D_Y$. Therefore the divisor $D_Y$ contains two twisted fibers with
 coefficient one, namely $E_Y$ and $(E_0)_Y$. Moreover $S_Y^2<0$ from Lemma \ref{lemma:Ssquare:negative} and from the bigness
 assumption. Therefore
 $$0 \le (K_Y+D_Y+\beta G_Y+S_Y).S_Y<(K_Y+D_Y+\beta G_Y+\alpha S_Y).S_Y,$$
 where the first inequality follows from Proposition \ref{prop:adjunction},
 and the second one follows since $\alpha <1$. Then log abundance does not contract $S_Y$,
 which is a contradiction.
 \end{proof}
 
 This is the final step. We need to make sure that, whenever we contract a pseudoelliptic surface, we do not create new intermediate fibers with the twisted component having coefficients which are not one. This could happen if
 we contract a pseudoelliptic with a marked irreducible pseudofiber to its attaching component. 
 
 \begin{lemma}\label{Lemma:all:pseudofibers:irred}
 Assume that $(X,D)$ is an irreducible pseudoelliptic surface with $K_X+D$ nef but not big. Let $G\neq 0$ be a union of pseudofibers and assume that $(X,D+G)$ is a stable pair.
 Assume finally that all the pseudofibers of $X$ are irreducible (i.e. there are no intermediate pseudofibers). Then the log-canonical model of $(X,D)$ is a point.
 \end{lemma}
 \begin{remark}
 The hypothesis on Lemma \ref{Lemma:all:pseudofibers:irred}
 automatically implies that all but at most one pseudofiber are twisted.  In fact, if there were 2 fibers with coefficient one, the pair $(X,D+G)$ would not be lc.
 
 \end{remark}
 \begin{proof}
 Let $p:X' \to X$ be the contraction of the section $S'$, and let $\alpha$ be such that $$K_{X'}+\alpha S'+p_*^{-1}(D)=p^*(K_X+D)$$ 
 Since $(X,D+G)$ is stable, from \cite[Corollary 3.53]{km}, 
 $X$ is obtained from $X'$ taking the stable model of $(X',S'+p_*^{-1}(D+G))$.
 So in particular $K_{X'}+S'+p_*^{-1}(D+G)$ is big, thus $S^2 <0$ from Lemma \ref{lemma:Ssquare:negative}.
 
 From Lemma
 \ref{lemma:the:coef:of:S:pulling:back:the:lc:on:a:psudo:positive}, $\alpha \ge 0$. Moreover, the lc model of $(X,D)$ is the same as the lc model of $(X',D')$, where
 $D':=\alpha S'+p_*^{-1}(D)$. So to take the stable model of $(X',D')$ we need to contract the section, and to prove the lemma, it suffices to show that we also need to contract the generic fiber. Namely, it suffices to show that
 $\alpha =0$.
 
 Consider an irreducible curve $M \subseteq X$ such that $M \cap p(S')=\emptyset$. Let $M'$
 be its proper transform. Then $M'$ is a multi-section, so
 it gets contracted upon taking the stable model of $(X',D')$.
 Therefore $$0=(K_{X'}+D').M'=(K_{X'}+\alpha S'+p_*^{-1}(D)).M,$$ but $S'$ and $M'$ do not intersect, so $(K_{X'}+p_*^{-1}(D)).M=0$.
 
 Finally let $f:X' \to C$ be the morphism to a curve, and let $d:=\deg (f_{|M})$.
 Since all the fiberd of $f$ are irreducible, $K_{X'}+p_*^{-1}(D)$ is supported on some fibers. So there
 is a divisor $G \subseteq C$ such that $K_{X'}+p_*^{-1}(D)=f^*(G)$. Thus:
 $$0=(K_{X'}+p_*^{-1}(D)).M=(K_{X'}+p_*^{-1}(D)).(dS')$$
 Since $K_{X'}+D'$ is nef, $0 \le (K_{X'}+\alpha S'+p_*^{-1}(D)).S'$, but
 since $S^2 <0$, we must have $\alpha=0$ as required.
 \end{proof}
 
 \begin{remark}
 The previous lemma fits in the picture as follows. With the notation of the beginning of this section, assume that $Z$ is a pseudoelliptic component of the closed fiber of $(X,S+F_\calA)$. Suppose that taking the stable model of $(X,S+F_\calB)$ contracts the pseudoelliptic component $Z$ to a curve $E$ (which will be a component of the double locus of the closed fiber of $X$). Then $Z$ must contain an intermediate fiber with twisted component $E'$ with coefficient one. Then the push-forward of $E'$ under the contraction of $Z$, marks $E$ with coefficient one.
 This guarantees that, after the contraction of $Z$, all the new intermediate fibers that appear have the twisted component marked with one.
 \end{remark}
 
We are finally ready to state our main result, which shows that we can control the birational transformations one has to deal with to find the stable limit of a family of  $\calA$-broken elliptic surfaces over a DVR.

\begin{theorem}\label{Teo:appendix:giovanni}
  Assume that $(f: X \to C,S+F_\mathcal{A}) \to \Spec(R)$ is a family of $\mathcal{A}$-broken elliptic surfaces over a DVR $R$. Write $F_{\mathcal{A}}=F_{\mathcal{B}}+G$ where $G$ is an effective $\mathbb{Q}$-Cartier divisor. Assume finally that $(X,S+F_\mathcal{A})$ is stable, and that $K_X+S+F_{\mathcal{B}}$ is nef.
  
  Then the codimension two exceptional locus arising from taking the stable model of $(X, S + F_\calB)$ will be a union of components of the section of the closed fiber.
  In particular, if $G$ is irreducible and $\epsilon$ is small enough, to take the stable model of $(X,S+F_\calB- \epsilon G)$ we only need to perform some divisorial contractions and a flip of La Nave.
 \end{theorem}
 
 \begin{proof} Let $p$ be the closed point of $\Spec(R)$.
  The components of $X_p$ are the union of pseudoelliptic surfaces and elliptic surfaces. In each of these surfaces, each divisor is either a fiber component, a component of $S_p$,
  or a multisection.
  
  Assume that $Y$ is an elliptic surface, which is an irreducible component of the closed fiber of $X \to \Spec(R)$. Let $C \subseteq Y$ be a curve we have to contract which is an irreducible component of codimension two of the exceptional locus. Then $C$ cannot be the genus one component of an intermediate fiber. Otherwise, contracting $C$ would either contract a divisor that meets the generic fiber (if $C \subseteq \operatorname{Supp}(F_\calB)$, we would contract a genus one component of an intermediate fiber of the generic fiber), or we would contract an irreducible component of the closed fiber (if $C$ is in the double locus of the closed fiber). In either case, $C$ would not be an irreducible component of codimension two of the exceptional locus.
  
  Therefore $C$ cannot be a fiber component. Indeed, $G_p$ is a union of fibers, pseudofibers, and intermediate components of intermediate fibers.
  So  $C.G \le 0$, and $(K_X+S+F_\mathcal{B}+G).C>0$ as  $(X,S+F_\mathcal{A})$ is stable. Thus $(K_X+S+F_\mathcal{B}).C>0$. Moreover
  from Lemma \ref{lemma:first:contract:the:section}, if $C$ is a multisection so that $C \neq S$, then the whole component $Y$ gets contracted. Thus the only curves we are left with are section components, and so any irreducible component of codimension two of the exceptional locus contained in an elliptic component of the closed fiber of $X$ must be a section component.
  
  Assume finally that $Y$ is pseudoelliptic. Then $C$ can either be a pesudomultisection, or a pseudofiber component. Proceeding as above, $C$ cannot be the genus one component of an intermediate pseudofiber. However, if we contract $C$ we also need to contract $Y$ from
  Corollary \ref{Cor:dont:contract:multi:on:psudo} and Lemma \ref{lemma:dont:contract:A:cp:on:a:psudo}. Thus any irreducible component of codimension two of the exceptional locus is not contained in a pseudoelliptic component.

We now prove that if we further reduce the weights on $F_{\mathcal{B}}$ to get
$F_{\mathcal{D}}$, to take the stable
model of $(X,S+F_{\mathcal{D}})$ the only flip we need to perform running the MMP is the one of La Nave.

Up to reducing a single weight at a time, we can assume that $F_\mathcal{A} - F_\mathcal{B}$, on the generic fiber, is irreducible, so let $G:=
\operatorname{Supp}(F_\mathcal{A} - F_\mathcal{B})$.
Then observe that $G \cap X_p$ is the union of a \emph{single} fiber component on an elliptic component, and some psudo-fiber components.
Then if $Y$ is a divisor of $X$ that gets contracted taking the stable model of $(X,S+F_{\mathcal{B}})$,
it gets contracted through a step of the MMP taking the stable model of $(X,S+F_{\mathcal{D}})$. So to take the stable model
of $(X,S+F_{\mathcal{D}})$, we start by performing these divisorial contractions.
By considering $(K_X+S+F_\mathcal{B})_{|Z}^2$ and $(K_X+S+F_\mathcal{B}-\epsilon G)_{|Z}^2$
for every irreducible
component $Z$ of $X_p$, and for every twisted component of an intermediate fiber of $X_\eta$, if $\epsilon$
is small enough, no other divisorial contraction is required. Next we consider the flips.

The section component that gets contracted to take the stable model of $(X,S+F_\mathcal{B})$, gets flipped through the flip of La Nave. Observe
that such a section component, if it does exist, is unique from the irreducibility of $G$.
We now obtain a new threefold $(X',S+F_\mathcal{B}-\epsilon G)$. We want to show that
if $\epsilon$
is small enough, the new pair is stable. We need to find the log-non-positive curves.

First, observe that for every $0<\epsilon$ such that $F_\mathcal{B}-\epsilon G$ is effective, the
pair $(X',S+F_\mathcal{B}-\epsilon G)$ is lc. Therefore also $(X',S+F_\mathcal{B})$ is lc. Moreover, by the uniqueness of the stable model,
the stable model of $(X',S+F_\mathcal{B})$ is $(X,S+F_\mathcal{B})$. Namely, to produce the stable model
of $(X',S+F_{\mathcal{B}})$ we need to contract the flipped curves.
Therefore it is easy to see that $K_{X'}+S+F_\mathcal{B}$ is nef.

The  the flipped curves are positive curves for $(X',S+F_\mathcal{B}-\epsilon G)$, by the definition of flip.
If $Z$ is a psudo-elliptic component of $X'$ that intersects $G$, let $Z'$ be its associated elliptic surface and let $p:Z' \to Z$ be the contraction of the
section $S_Z$. For every $\epsilon$, let
$L_Z(\epsilon):=(K_{X'}+S+F_{\mathcal{B}}-\epsilon G)_{|Z}$ and let $L_{Z'}(\epsilon):=p^*(L_Z(\epsilon))$.
Since $Z$ does not get contracted taking the stable model of $(X',S+F_\mathcal{B})$, we see that $L_{Z'}(0) \equiv p_*^{-1}(L_Z(0))+\alpha S_Z$
and $\alpha >0$.
Then one can see that if $\epsilon$ is small enough, the pair $(Z,(F_{\mathcal{B}}-\epsilon G)_{|Z})$ is the stable model of
$(Z',(\alpha - \epsilon)S_Z+p_*^{-1}((F_{\mathcal{B}}-\epsilon G)_{|Z}))$. In particular, there are no non-positive curves in any psudoelliptic component.

On an elliptic component, choosing $\epsilon$ small enough, any section component is a positive curve since no contraction happens in Hassett's space.
Any fiber component away from the flipped curve remains positive, and the twisted component of the new intermediate fiber,
introduced by the flip of La Nave, remains positive since we have no contractions on the psudoelliptic components. The flipped curve
is positive by the definition of a flip, and from Lemma \ref{lemma:first:contract:the:section}, no multisection
is a negative curve. Then for $\epsilon$ small enough there are no negative curves, so $(X',S+F_\mathcal{B}-\epsilon G)$ is ample. \end{proof}

\bibliographystyle{alpha}\bibliography{master}

\begin{thebibliography}{AFKM02}

\bibitem[AB16]{tsm}
Kenneth Ascher and Dori Bejleri.
\newblock Moduli of elliptic surfaces from twisted stable maps.
\newblock {\em arXiv:1612:00792}, 2016.

\bibitem[AB17]{calculations}
Kenneth Ascher and Dori Bejleri.
\newblock Log canonical models of elliptic surfaces.
\newblock {\em Adv. Math.}, 320:210--243, 2017.

\bibitem[AB18]{rational}
K.~{Ascher} and D.~{Bejleri}.
\newblock {Stable pair compactifications of the moduli space of degree one del
  pezzo surfaces via elliptic fibrations, I}.
\newblock {\em arXiv:1802.00805}, February 2018.

\bibitem[AFKM02]{afkm}
D.~Abramovich, L.-Y. Fong, J.~Koll\'ar, and J.~McKernan.
\newblock Semi log canonical surface.
\newblock {\em Flips and Abundance for Algebraic Threefolds}, Ast\'erique
  211:139--154, 2002.

\bibitem[Ale94]{boundedness}
Valery Alexeev.
\newblock Boundedness and {$K^2$} for log surfaces.
\newblock {\em Internat. J. Math.}, 5(6):779--810, 1994.

\bibitem[Ale06]{aleicm}
Valery Alexeev.
\newblock Higher-dimensional analogues of stable curves.
\newblock In {\em International {C}ongress of {M}athematicians. {V}ol. {II}},
  pages 515--536. Eur. Math. Soc., Z\"urich, 2006.

\bibitem[Ale08]{aleproper}
Valery Alexeev.
\newblock Limits of stable pairs.
\newblock {\em Pure Appl. Math. Q.}, 4(3, Special Issue: In honor of Fedor
  Bogomolov. Part 2):767--783, 2008.

\bibitem[Ale15]{barcelona}
Valery Alexeev.
\newblock {\em Moduli of weighted hyperplane arrangements}.
\newblock Advanced Courses in Mathematics. CRM Barcelona.
  Birkh\"auser/Springer, Basel, 2015.
\newblock Edited by Gilberto Bini, Mart{\'{\i}} Lahoz, Emanuele Macr{\`{\i}}
  and Paolo Stellari.

\bibitem[Art66]{artin}
Michael Artin.
\newblock On isolated rational singularities of surfaces.
\newblock {\em American Journal of Mathematics}, 88(1):129--136, 1966.

\bibitem[AV97]{av}
Dan Abramovich and Angelo Vistoli.
\newblock Complete moduli for fibered surfaces.
\newblock In {\em Recent progress in intersection theory (Bologna, 1997)}.
  Birkh\"auser Boston, 1997.

\bibitem[AV02]{av2}
Dan Abramovich and Angelo Vistoli.
\newblock Compactifying the space of stable maps.
\newblock {\em J. Amer. Math. Soc.}, 15(1):27--75, 2002.

\bibitem[Bru15]{brunyate}
Adrian Brunyate.
\newblock {\em A Modular Compactification of the Space of Elliptic K3
  Surfaces}.
\newblock PhD thesis, The University of Georgia, 2015.

\bibitem[{Deo}15]{deopurkar}
A.~{Deopurkar}.
\newblock {Covers of stacky curves and limits of plane quintics}.
\newblock {\em arXiv:1507.03252}, July 2015.

\bibitem[FS13]{fedorchuk}
Maksym Fedorchuk and David~Ishii Smyth.
\newblock Alternate compactifications of moduli spaces of curves.
\newblock In {\em Handbook of moduli. {V}ol. {I}}, volume~24 of {\em Adv. Lect.
  Math. (ALM)}, pages 331--413. Int. Press, Somerville, MA, 2013.

\bibitem[Fuj14]{fundamental}
Osamu Fujino.
\newblock Fundamental theorems for semi log canonical pairs.
\newblock {\em Algebr. Geom.}, 1(2):194--228, 2014.

\bibitem[Has01]{hasproper}
Brendan Hassett.
\newblock Stable limits of log surfaces and {C}ohen-{M}acaulay singularities.
\newblock {\em J. Algebra}, 242(1):225--235, 2001.

\bibitem[Has03]{has}
Brendan Hassett.
\newblock Moduli spaces of weighted pointed stable curves.
\newblock {\em Advances in Mathematics}, pages 316--352, 2003.

\bibitem[HL02]{hl}
Gert Heckman and Eduard Looijenga.
\newblock The moduli space of rational elliptic surfaces.
\newblock In {\em Algebraic Geometry 2000}, volume~36 of {\em Adv. Stud. Pure
  Math}. Math. Soc. Japan, 2002.

\bibitem[HR14]{hr}
J.~{Hall} and D.~{Rydh}.
\newblock {Coherent Tannaka duality and algebraicity of Hom-stacks}.
\newblock {\em arXiv:1405.7680}, May 2014.

\bibitem[Iit82]{iitaka}
Shigeru Iitaka.
\newblock {\em Algebraic geometry}, volume~76 of {\em Graduate Texts in
  Mathematics}.
\newblock Springer-Verlag, New York-Berlin, 1982.
\newblock An introduction to birational geometry of algebraic varieties,
  North-Holland Mathematical Library, 24.

\bibitem[Inc18]{giovanni}
Giovanni Inchiostro.
\newblock Moduli of weierstrass fibrations.
\newblock {\em In Preparation}, 2018.

\bibitem[Kaw92]{kawamata}
Yujiro Kawamata.
\newblock Abundance theorem for minimal threefolds.
\newblock {\em Invent. Math.}, 108(2):229--246, 1992.

\bibitem[KK02]{kk}
S.~Keel and J.~Koll\'ar.
\newblock Log canonical flips.
\newblock {\em Flips and Abundance for Algebraic Threefolds}, Ast\'erique
  211:139--154, 2002.

\bibitem[KM98]{km}
J{\'a}nos Koll{\'a}r and Shigefumi Mori.
\newblock {\em Birational geometry of algebraic varieties}, volume 134 of {\em
  Cambridge Tracts in Mathematics}.
\newblock Cambridge University Press, Cambridge, 1998.
\newblock With the collaboration of C. H. Clemens and A. Corti, Translated from
  the 1998 Japanese original.

\bibitem[Knu83]{knudsen}
Finn~F. Knudsen.
\newblock The projectivity of the moduli space of stable curves. {II}. {T}he
  stacks {$M_{g,n}$}.
\newblock {\em Math. Scand.}, 52(2):161--199, 1983.

\bibitem[Kol13]{singmmp}
J{\'a}nos Koll{\'a}r.
\newblock {\em Singularities of the minimal model program}, volume 200 of {\em
  Cambridge Tracts in Mathematics}.
\newblock Cambridge University Press, Cambridge, 2013.
\newblock With a collaboration of S{\'a}ndor Kov{\'a}cs.

\bibitem[{Kol}18a]{kollarpluri}
J.~{Koll{\'a}r}.
\newblock {Log-plurigenera in stable families}.
\newblock {\em arXiv:1801.05414}, January 2018.

\bibitem[Kol18b]{kollarmodulibook}
J\'anos Koll\'ar.
\newblock {\em Families of varieties of general type}.
\newblock 2018.

\bibitem[KP17]{kp}
S\'andor~J. Kov\'acs and Zsolt Patakfalvi.
\newblock Projectivity of the moduli space of stable log-varieties and
  subadditivity of log-{K}odaira dimension.
\newblock {\em J. Amer. Math. Soc.}, 30(4):959--1021, 2017.

\bibitem[KSB88]{ksb}
J.~Koll{\'a}r and N.~I. Shepherd-Barron.
\newblock Threefolds and deformations of surface singularities.
\newblock {\em Invent. Math.}, 91(2):299--338, 1988.

\bibitem[LN02]{ln}
Gabrielle La~Nave.
\newblock Explicit stable models of elliptic surfaces with sections.
\newblock {\em arXiv: 0205035}, 2002.

\bibitem[LO16]{ol}
R.~{Laza} and K.~G. {O'Grady}.
\newblock {Birational geometry of the moduli space of quartic K3 surfaces}.
\newblock {\em arXiv:1607.01324}, July 2016.

\bibitem[Mir81]{mir}
Rick Miranda.
\newblock The moduli of weierstrass fibrations over $\mathbb{P}^1$.
\newblock {\em Math. Ann.}, 255(3):379--394, 1981.

\bibitem[Mir89]{mir3}
Rick Miranda.
\newblock {\em The basic theory of elliptic surfaces}.
\newblock Dottorato di Ricerca in Matematica. [Doctorate in Mathematical
  Research]. ETS Editrice, Pisa, 1989.

\bibitem[Oss]{osserman}
Brian Osserman.
\newblock {Notes on cohomology and base change}.
\newblock URL:
  \url{https://www.math.ucdavis.edu/~osserman/math/cohom-base-change.pdf}.

\bibitem[{Sta}16]{stacks-project}
The {Stacks Project Authors}.
\newblock \itshape stacks project.
\newblock \url{http://stacks.math.columbia.edu}, 2016.

\end{thebibliography}
\end{document}